\documentclass[11pt]{article}
\usepackage{amssymb,amsmath}
\title{\bf\boldmath 
Remarks on 
$\tau$-functions for the difference Painlev\'e equations
of type $E_8$}
\author{
Masatoshi NOUMI\footnote{
Department of Mathematics, Kobe University, 
Rokko, Kobe  657-8501, Japan
\newline\quad
\newline\quad
determinant, elliptic hypergeometric integral}
}
\date{}
\newtheorem{thm}{Theorem}[section]
\newtheorem{prop}[thm]{Proposition}

\newtheorem{lem}[thm]{Lemma}

\newtheorem{dfn}[thm]{Definition}
\newtheorem{exm}[thm]{Example}
\numberwithin{equation}{section}
\newcommand{\bC}{\mathbb{C}}

\newcommand{\bN}{\mathbb{N}}
\newcommand{\bP}{\mathbb{P}}
\newcommand{\bR}{\mathbb{R}}
\newcommand{\bZ}{\mathbb{Z}}

\newcommand{\cC}{\mathcal{C}}

\newcommand{\cF}{\mathcal{F}}
\newcommand{\cG}{\mathcal{G}}

\newcommand{\frh}{\mathfrak{h}}
\newcommand{\la}{\langle}
\newcommand{\ra}{\rangle}
\newcommand{\isoto}{\stackrel{\sim}{\rightarrow}}
\newcommand{\isofrom}{\stackrel{\sim}{\leftarrow}}

\newcommand{\bsf}[1]{\mbox{\sf #1}}
\newcommand{\bssf}[1]{\mbox{\scriptsize\sf #1}}
\newcommand{\br}[1]{\{#1\}}
\newcommand{\brm}[2]{\big\{#1\,\big|\,#2\big\}}

\newcommand{\ipr}[2]{(#1|#2)}
\newcommand{\hf}{\tfrac{1}{2}}
\newcommand{\shf}{{\mbox{\scriptsize $\frac{1}{2}$}}}
\newcommand{\I}{{\rm I}}
\newcommand{\II}{{\rm II}}
\newcommand{\sI}{{\mbox{\scriptsize\rm I}}}
\newcommand{\sII}{\mbox{{\scriptsize\rm II}}}
\newcommand{\ep}{\epsilon}
\newcommand{\vep}{\varepsilon}
\renewcommand{\Im}{{\rm Im}\,}

\newcommand{\sgn}{{\rm sgn}}
\newcommand{\Hom}{{\rm Hom}}
\newcommand{\scirc}{\mbox{\tiny$\circ$}}
\newcommand{\qed}{\hfill $\square$\par\smallskip}
\newenvironment{proof}[1]{
\par\smallskip\noindent{{\sl #1}\,:}
}{\hfill $\square$\par\smallskip}
\begin{document}
\maketitle
\begin{abstract}
We investigate the structure of $\tau$-functions for 
the elliptic difference Painlev\'{e} equation of type $E_8$.  
Introducing the notion of ORG $\tau$-functions 
for the $E_8$ lattice,  
we construct some particular solutions which 
are expressed in terms of elliptic hypergeometric integrals.  
Also, we discuss how this construction is related to 
the framework of 
lattice $\tau$-functions associated with the configuration of 
generic nine points in the projective plane. 
\end{abstract}

\section*{Introduction}
There are several approaches to difference (or discrete) 
Painlev\'e equations associated with the root system of 
type $E_8$.  In the geometric approach of Sakai 
\cite{Sakai2001}, there are three 
discrete Painlev\'e equations with affine Weyl group 
symmetry of type $E^{(1)}_8$, which may be called 
rational, trigonometric and elliptic.  
They are formulated 
in the language of 
certain rational surfaces 
obtained 
from $\bP^2$ by blowing-up at generic nine points, 
and regarded as master families of second order 
discrete Painlev\'{e} equations. 
On the other hand, 
Ohta--Ramani--Grammaticos \cite{ORG2001} introduced 
the elliptic (difference) Painlev\'{e} equation of type 
$E_8$ and its $\tau$-functions 
as a discrete system on the root lattice 
of type $E_8$. 
Equivalence of these two approaches has been 
clarified by Kajiwara et al.\,\cite{KMNOY2003, KMNOY2006} 
in the framework of the birational affine Weyl group action 
on the configuration space 
of generic nine points in $\bP^2$ and the associated 
lattice $\tau$-functions. 
Besides these approaches, the elliptic Painlev\'{e} 
equation is interpreted as the compatibility condition of 
certain linear difference equations 
in two ways 
by Rains \cite{Rains2011} 
and by Noumi--Tsujimoto--Yamada \cite{NTY2013}. 
Also, it is known that the elliptic difference Painlev\'{e} equation 
has particular solutions which are expressible in terms of 
elliptic hypergeometric functions as in 
Kajiwara et al.\,\cite{KMNOY2003}, 
Rains \cite{Rains2005, Rains2011} and 
Noumi--Tsujimoto--Yamada \cite{NTY2013}.  

\par\medskip
In this paper we introduce the notion of {\em ORG $\tau$-functions}, 
which is a reformulation of 
$\tau$-functions associated with the $E_8$ lattice proposed 
by Ohta--Ramani--Grammaticos \cite{ORG2001}.  
We fix a realization of the root lattice 
$P=Q(E_8)$ of type $E_8$ in the 8-dimensional complex 
vector space $V=\bC^8$ 
endowed with the canonical symmetric bilinear form 
$\ipr{\cdot}{\cdot}:\ V\times V\to\bC$ (Section \ref{sec:1}),  
and take a subset $D\subseteq V$ such that 
$D+P\delta=D$, where 
$\delta\in\bC^\ast$ is a nonzero constant.  
A function $\tau(x)$ defined on $D$ is called 
an {\em ORG $\tau$-function} if it satisfies a system
of non-autonomous Hirota equations 
\begin{equation}\label{eq:Hirotaabc}
[\ipr{b\pm c}{x}]\tau(x\pm a\delta)
+[\ipr{c\pm a}{x}]\tau(x\pm b\delta)
+[\ipr{a\pm b}{x}]\tau(x\pm c\delta)
=0
\end{equation}
for all $C_3$-frames $\br{\pm a, \pm b,\pm c}$ 
relative to $P$ (Section \ref{sec:2}).  Here $[z]$ denotes 
any nonzero entire function in $z\in\bC$ which satisfies 
the three term relation \eqref{eq:three-term}, 
and each double sign in \eqref{eq:Hirotaabc}
indicates the product of two functions with different signs. 
When the domain $D$ is a disjoint union 
\begin{equation}
D_c=\bigsqcup_{n\in\bZ} H_{c+n\delta},
\quad
H_{c+n\delta}=\brm{x\in V}{\ipr{\phi}{x}=c+n\delta},
\vspace{-4pt}
\end{equation}
of parallel hyperplanes perpendicular to 
$\phi=(\hf,\hf,\ldots,\hf)$, each ORG $\tau$-function 
$\tau=\tau(x)$ on $D_c$ (of type $E_8$) 
is regarded as an infinite chain of ORG $\tau$-functions 
$\tau^{(n)}=\tau|_{H_{c+n\delta}}$ on $H_{c+n\delta}$ 
of type $E_7$ ($n\in\bZ$) (Section \ref{sec:3}).
We say that a meromorphic ORG $\tau$-function 
$\tau(x)$ on $D_c$ 
is a {\em hypergeometric $\tau$-function} 
if $\tau^{(n)}(x)=0$ $(n<0)$ and 
$\tau^{(0)}(x)\not\equiv 0$. 
Supposing that $\omega\in\Omega$ is a period of $[z]$, 
consider the case where $D=D_\omega$.  
In such a case, on the basis of a recursion 
theorem (Theorem \ref{thm:3C})
one can show
that, 
if a given pair of functions 
$\tau^{(0)}(x)$ ($x\in H_{\omega})$ and 
$\tau^{(1)}(x)$ ($x\in H_{\omega+\delta}$) 
satisfies certain initial conditions, then 
there exists a unique hypergeometric $\tau$-function 
$\tau(x)$ $(x\in D_\omega)$ 
having those $\tau^{(0)}(x)$, $\tau^{(1)}(x)$ 
as the first two components (Theorem \ref{thm:4A}). 
Furthermore, if one can specify a gauge factor for 
$\tau^{(1)}$ with respect to a $C_3$-frame of 
type $\II_1$, then 
for each $n=0,1,2,\ldots$
the $n$th component $\tau^{(n)}(x)$ 
($x\in H_{\omega+n\delta}$)
of the hypergeometric $\tau$-function 
is expressed in terms of a {\em 2-directional Casorati 
determinant} with respect to the $C_3$-frame
of type $\II_1$ (Theorem \ref{thm:4B}). 
A proof of the recursion theorem (Thoerem \ref{thm:3C}) 
will be given in Appendix A. 

In the latter half of this paper, we apply our arguments 
to the elliptic case for constructing 
hypergeometric ORG $\tau$-functions which are 
expressible in terms of 
elliptic hypergeometric integrals of 
Spiridonov \cite{Spiridonov2003, Spiridonov2005}
and 
Rains \cite{Rains2005, Rains2010}. 
The fundamental elliptic hypergeometric integral
is the meromorphic 
function 
$I(u;p,q)$ in eight variables $u=(u_0,u_1,\ldots,u_7)\in
(\bC^\ast)^8$ defined by
\begin{equation}
I(u;p,q)=
\frac{(p;p)_\infty(q;q)_\infty}{4\pi\sqrt{-1}}
\int_{C}
\frac{\prod_{k=0}^{7}\Gamma(u_kz^{\pm1};p,q)}{\Gamma(z^{\pm2};p,q)}\dfrac{dz}{z}. 
\end{equation}
After recalling basic facts concerning 
elliptic hypergeometric 
integrals in Section \ref{sec:5}, 
we present in Section \ref{sec:6} two types of explicit 
representations 
for the $W(E_7)$-invariant hypergeometric 
ORG $\tau$-functions, one by determinants 
(Theorem \ref{thm:6A})
and the other by multiple integrals (Theorem \ref{thm:6B}). 
Theorems \ref{thm:6A} and \ref{thm:6B} will be proved in 
Section \ref{sec:7}.  
In particular we show there 
how the 2-directional Casorati 
determinant gives rise to the multiple elliptic hypergeometric 
integral of Rains \cite{Rains2010}: 
\begin{equation}\label{eq:multellint}
\begin{split}
&I_n(u;p,q)
\\
&=
\frac{(p;p)_{\infty}^n(q;q)_\infty^{n}}
{2^n n! (2\pi\sqrt{-1})^n}
\int_{C^n}
\prod_{i=1}^{n}
\frac{
\prod_{k=0}^{7}\Gamma(u_kz_i^{\pm1};p,q)
}{\Gamma(z_i^{\pm2};p,q)}
\prod_{1\le i<j\le n}
\theta(z_i^{\pm1}z_j^{\pm1};p)
\frac{dz_1\cdots dz_n}{z_1\cdots z_n}.
\end{split}
\end{equation}
We also give some remarks in Section \ref{sec:8} on 
variations of hypergeometric 
ORG $\tau$-functions obtained by 
transformations in Theorem \ref{thm:2B}. 

In the final section, we discuss how the notion of ORG 
$\tau$-functions is related to that of 
lattice $\tau$-functions as discussed by 
Kajiwara et al.\,\cite{KMNOY2006}
in the context of the configuration space of generic 
nine points in $\bP^2$.  Some remarks are also 
given on the similar picture in the case of the 
configuration space of generic 
eight points in $\bP^1\times\bP^1$ 
as in Kajiwara--Noumi--Yamada \cite{KNY2015}.  

\par\medskip
On this occasion I would like to 
give some personal comments on the position 
of this paper.  
The contents of this paper are not completely new, 
and many things presented here may be found 
in the literature. 
In fact, in the group of coauthors of \cite{KMNOY2006}, 
basic structures of the ORG $\tau$-functions were already known around the end of 2004, including 
the 
2-directional Casorati determinant representation 
of the hypergeometric $\tau$-functions.  
(Some part of our discussion is reflected in the work 
on Masuda \cite{Masuda2011}.)
It was almost at the same time that 
Rains \cite{Rains2005} clarified that 
his multiple elliptic hypergeometric integrals 
\eqref{eq:multellint} 
satisfy certain quadratic relations of Hirota type that should 
be understood in the context of the 
elliptic Painlev\'{e} equation.  
It took a couple of years, however, for us to be 
able to confirm that 
the 2-directional Casorati determinants 
for a particular choice of $C_3$-frame
of type $\II_1$ certainly give rise to the multiple elliptic 
hypergeometric integrals of Rains.  
As to further delay in presenting the detail of 
such an argument, 
I apologize just adding that
it requires much more time and effort 
than might be imagined  
to accomplish a satisfactory paper 
in the language of a cultural sphere to which 
its author does not belong. 

{
\small
\baselineskip=0pt
\setcounter{tocdepth}{1}
\tableofcontents
}

\newpage
\section{\boldmath $E_8$ lattice and $C_l$-frames}
\label{sec:1}
We begin by recalling some basic 
facts concerning the root lattice of type $E_8$. 
\par\medskip
Let 
$V=\bC^8=\bC v_0\oplus\bC v_1\oplus\cdots\oplus\bC v_7$ 
be the 8-dimensional complex vector space 
with canonical basis $\br{v_0,v_1,\ldots,v_7}$, 
and $\ipr{\cdot}{\cdot}: V\times V\to \bC$ the scalar product
(symmetric bilinear form) such that 
$\ipr{v_i}{v_j}=\delta_{ij}$ ($i,j\in\br{0,1,\ldots,7}$). 
Setting
\begin{equation}
\phi=(\hf,\hf,\ldots,\hf)=\hf(v_0+v_1+\cdots+v_7)\in V,
\end{equation}
we realize the root lattice $Q(E_8)$ and the root 
system $\Delta(E_8)$ of type $E_8$ as
\begin{equation}
P=\brm{a\in\bZ^8\cup(\phi+\bZ^8)}{\ipr{\phi}{a}\in\bZ}\subset V,
\quad 
\Delta(E_8)=\brm{\alpha\in P}{\ipr{\alpha}{\alpha}=2}, 
\end{equation}
respectively (see \cite{CS1993} for instance). 
This set $P$ forms a free $\bZ$-module of rank 8 and 
the scalar product takes integer values on $P$.  
The theta series of the lattice 
$P=Q(E_8)$ is given by
\begin{equation}
\sum_{a\in P}q^{\ipr{a}{a}}=
1+240q^2+2160q^4+6720q^6+17520q^8
+30240 q^{10}+\cdots, 
\end{equation}
and $\Delta(E_8)$ 
consists of the following 240 
vectors in $P$:
\begin{equation}
\begin{array}{llllr}
(1) & \pm v_i\pm v_j &(0\le i<j\le 7)
&\cdots&\binom{8}{2}\cdot 4=112,
\\[4pt]
(2) & \hf(\pm v_0\pm v_1\pm\cdots\pm v_7)\quad
&(\mbox{even number of $-$ signs})
&\cdots&2^7=128.
\end{array}
\end{equation}
In this root system, we take the {\em simple roots}
\begin{equation}
\alpha_0=\phi-v_0-v_1-v_2-v_3,\quad
\alpha_j=v_j-v_{j+1}\ \ (j=1,\ldots,6),\quad
\alpha_7=v_7+v_0
\end{equation}
corresponding to the Dynkin diagram
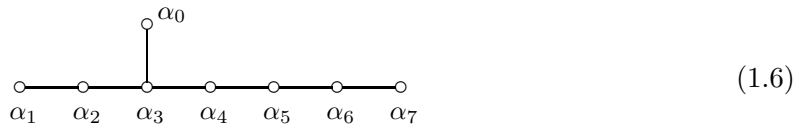
\begin{equation}\label{eq:DynkinE8}
\begin{picture}(170,30)(0,20)
\multiput(0,20)(24,0){7}{\circle{4}}
\multiput(2,20)(24,0){6}{\line(1,0){20}}
\put(48,44){\circle{4}}\put(48,22){\line(0,1){20}}
\put(-4,8){\small $\alpha_1$}
\put(20,8){\small $\alpha_2$}
\put(44,8){\small $\alpha_3$}
\put(68,8){\small $\alpha_4$}
\put(92,8){\small $\alpha_5$}
\put(116,8){\small $\alpha_6$}
\put(140,8){\small $\alpha_7$}
\put(52,46){\small $\alpha_0$}
\end{picture}
\vspace{10pt}
\end{equation}
so that 
$P=Q(E_8)=\bZ\alpha_0\oplus
\bZ\alpha_1\oplus\cdots\oplus\bZ\alpha_7$.  
(For mutually distinct $i,j\in\br{0,1,\ldots,7}$,
$\ipr{\alpha_i}{\alpha_j}=-1$ if the two nodes 
named $\alpha_i$ and $\alpha_j$ are connected by an edge, 
and $\ipr{\alpha_i}{\alpha_j}=0$ otherwise.)
The vector 
\begin{equation}\label{eq:phiE8}
\phi=3\alpha_0+2\alpha_1+4\alpha_2+6\alpha_3
+5\alpha_4+4\alpha_5+3\alpha_6+2\alpha_7
\end{equation}
is called the {\em highest root} with respect to the simple roots 
$\alpha_0,\alpha_1,\ldots,\alpha_7$. 
Note also that 
the weight lattice $P(E_8)$ coincides with 
the root lattice $Q(E_8)$ in this $E_8$ case.  

For each $\alpha\in V$ with $\ipr{\alpha}{\alpha}\ne 0$, 
we define the {\em reflection} $r_\alpha: V\to V$ with 
respect to $\alpha$ by 
\begin{equation}
r_\alpha(v)=v-\ipr{\alpha^\vee}{v}\alpha\qquad(v\in V), 
\end{equation}
where $\alpha^\vee=2\alpha/\ipr{\alpha}{\alpha}$. 
The Weyl group 
$W(E_8)=\la r_{\alpha}\ (\alpha\in \Delta(E_8))\ra$ 
of type $E_8$ acts on $V$ as a group of isometries;
it stabilizes the root lattice $P=Q(E_8)$
and the root system $\Delta(E_8)$. 
We denote by $s_j=r_{\alpha_j}$  ($j=0,1,\ldots,7$) 
the {\em simple reflections}. Then, 
$W(E_8)=\la s_0,s_1,\ldots,s_7\ra$ is the Coxeter group 
associated with the Dynkin diagram \eqref{eq:DynkinE8}. 
This group is generated by 
$s_0, s_1,\ldots,s_7$ with fundamental relations 
$s_j^2=1$ ($j=0,1,\ldots,7$), 
$s_is_j=s_js_i$ for distinct $i,j\in\br{0,1,\ldots,7}$ 
with $\ipr{\alpha_i}{\alpha_j}=0$ and 
$s_is_js_i=s_js_is_j$ for distinct $i,j\in\br{0,1,\ldots,7}$ 
with $\ipr{\alpha_i}{\alpha_j}=-1$. 

\par\medskip
In the $E_8$ lattice $P=Q(E_8)$, the root lattice 
$Q(E_7)$ and the root system $\Delta(E_7)$ 
of type $E_7$ are 
realized as 
\begin{equation}
\begin{split}
Q(E_7)&=\brm{a\in P}{\ipr{\phi}{a}=0}=
\bZ\,\alpha_0\oplus
\bZ\,\alpha_1\oplus
\cdots\oplus
\bZ\,\alpha_6,
\\[4pt]
\Delta(E_7)&=\brm{\alpha\in\Delta(E_8)}{\ipr{\phi}{\alpha}=0},
\end{split}
\end{equation}
respectively. 
The root system $\Delta(E_7)$ consists of the following 
126 vectors:
\begin{equation}
\begin{array}{llllr}
(1) & \pm(v_i-v_j) & (0\le i<j\le 7) &\cdots& \binom{8}{2}\cdot 2=56,
\\[4pt]
(2) & \hf(\pm v_0\pm v_1\pm\cdots\pm v_7)
&(\mbox{four $-$ signs})&\cdots & 
\binom{8}{4}=70.
\end{array}
\end{equation}
The highest root of $\Delta(E_7)$ with respect to 
the simple roots $\alpha_0,\alpha_1,\ldots,\alpha_6$ 
is given by
\begin{equation}
v_1-v_0=2\alpha_0+2\alpha_1+3\alpha_2+
4\alpha_3+3\alpha_4+2\alpha_5+\alpha_6. 
\end{equation}
We denote by $W(E_7)=\la r_{\alpha}\ (\alpha\in\Delta(E_7))\ra
=\la s_0,s_1,\ldots,s_6\ra$ the Weyl group of type $E_7$. 
Note 
that $W(E_7)$ contains the symmetric group 
$\mathfrak{S}_8=\la r_{v_0-v_1}, s_1,\ldots,s_6\ra$ acting 
on $V$ through the permutation of $v_0,v_1,\ldots,v_7$, 
and is generated by this $\mathfrak{S}_8$ together with 
the reflection $s_0$ with respect to $\alpha_0=\phi-v_0-v_1-v_2-v_3$. 

\par\medskip
In this paper, 
the following notion of $C_l$-frames 
plays a fundamental role. 

\begin{dfn} \rm 
For $l=1,2,\ldots,8$, 
a set $A=\br{\pm a_0,\pm a_1,\ldots,\pm a_{l-1}}$ 
of $2l$ vectors in $V$ is called a {\em $C_{l}$-frame} 
({\em relative to} $P$), if 
the following two conditions are satisfied:
\smallskip
\newline
\quad 
$(1)$\quad
$\ipr{a_i}{a_j}=\delta_{ij}$\quad$(0\le i,j<l)$, 
\smallskip
\newline
\quad
$(2)$\quad
$a_i\pm a_j\in P$\ \ $(0\le i<j<l)$,\quad 
$2a_i\in P$\ \ $(0\le i<l)$. 
\end{dfn}
This condition for $A=\br{\pm a_0,\ldots,\pm a_{l-1}}$ 
means that the set of $2l^2$ vectors 
\begin{equation}
\Delta_{A}(C_l)=\brm{\pm a_i\pm a_j}{0\le i<j<l)}\cup
\brm{\pm 2a_i}{0\le i<l}\subset P
\end{equation}
form a root system of type $C_l$. 
For $l=1,2,\ldots,8$,
we denote by $\cC_l$ the set of all $C_l$-frames.
\begin{exm}[Typical examples of $C_8$-frames]
\label{exm:C8frames}
\begin{equation*}
\begin{split}
(0)\quad &A_0=\br{\pm v_0,\pm v_1,\ldots,\pm v_7}.
\\
(1)\quad &A_1=\br{\pm a_0,\pm a_1,\ldots,\pm a_7}, 
\\
&\begin{array}{ll}
a_0=\hf(v_0+v_1+v_2+v_3), 
&
a_4=\hf(v_4-v_5-v_6+v_7),
\\[4pt]
a_1=\hf(v_0+v_1-v_2-v_3), 
&
a_5=\hf(-v_4+v_5-v_6+v_7),
\\[4pt]
a_2=\hf(v_0-v_1+v_2-v_3), 
&
a_6=\hf(-v_4-v_5+v_6+v_7),
\\[4pt]
a_3=\hf(v_0-v_1-v_2+v_3), 
&
a_7=\hf(v_4+v_5+v_6+v_7).
\end{array}
\\
(2)\quad &A_2=\br{\pm a_0,\pm a_1,\ldots,\pm a_7},
\\
&\ 
a_0=\hf(\phi+v_0-v_7),\quad 
a_7=\hf(\phi-v_0+v_7),
\\
&\ 
a_j=v_j+\hf(v_0+v_7-\phi)\quad(j=1,\ldots,6).
\end{split}
\end{equation*}
\end{exm}

In the following, 
we denote by $N(v)=\ipr{v}{v}$ 
the square norm of $v\in V$, 
and by $\varphi(v)=\ipr{\phi}{v}$ 
the scalar product of $v$ with $\phi$.
Also, for a subset $S\subseteq V$ given, 
we use the notations
\begin{equation}
S_{N=k}=\brm{v\in S}{\ipr{v}{v}=k},\quad
S_{\varphi=k}=\brm{v\in S}{\ipr{\phi}{v}=k}
\qquad(k\in\bC)
\end{equation}
to refer to the level sets of $N$ and $\varphi$ 
respectively.  

Note that each $C_l$-frame ($l=1,2,\ldots,8$)
is formed by vectors 
in $(\hf P)_{N=1}=\hf(P_{N=4})$.  The set $P_{N=4}$
of all vectors in $P$ with square norm 4 
consists 
of the following $2160$ vectors that are 
classified into three groups under the action of 
the symmetric group $\mathfrak{S}_8$: 
\begin{equation}
\arraycolsep=2pt
\begin{array}{llllr}
(0) & \pm 2v_0 
&&\cdots&8\cdot 2=16,
\\[4pt]
(1) & \pm v_0\pm v_1\pm v_2\pm v_3 
&&\cdots&\binom{8}{4}\cdot 2^4=1120,
\\[4pt]
(2) & \hf(\pm 3v_0\pm v_1\pm\cdots\pm v_7)\quad
&(\mbox{odd number of $-$ signs})
&\cdots&8\cdot 2^7=1024. 
\end{array}
\end{equation}
The Weyl group $W(E_8)$ acts on $P_{N=4}$ transitively.  
In fact we have $P_{N=4}\cap P^+=\br{\phi-v_0+v_1}$, 
where $P^+=\brm{v\in P}{\ipr{\alpha_j}{v}\ge 0 (0\le j\le 7)}$ 
stands for the cone of dominant integral weights. 
It turns out that the stabilizer of $\phi-v_0+v_1$ is $W(D_7)$ 
and that 
\begin{equation}
P_{N=4}\isofrom W(E_8)/W(D_7),\quad
|P_{N=4}|=|W(E_8)/W(D_7)|=2160. 
\end{equation}
The following two propositions can be verified 
directly 
on the basis of this transitive action of $W(E_8)$ on $P_{N=4}$. 

\begin{prop}\label{prop:1A}\ 
\par\smallskip\noindent
\quad$(1)$ For each $a\in(\hf P)_{N=1}$, there exists a unique 
$C_8$-frame containing $a$.  
\par\smallskip\noindent
\quad$(2)$ The set $(\hf P)_{N=1}$ is the disjoint union 
of all $C_8$-frames$:$ $(\hf P)_{N=1}=\bigsqcup_{A\in\cC_8} A$.  
\par\smallskip\noindent
\quad$(3)$ The number of $C_8$-frames is given by 
$|\cC_8|=2160/16=135$. 
\qed
\end{prop}

\begin{prop}\label{prop:1B} Fix a positive integer 
$l\in\br{1,\ldots,8}$.  
\par\smallskip\noindent
\quad$(1)$  The Weyl group 
$W(E_8)$ acts transitively on the set $\cC_l$ of all $C_l$-frames. 
\par\smallskip\noindent
\quad$(2)$  Each $C_l$-frame is contained 
in a unique $C_8$-frame.
\par\smallskip\noindent
\quad$(3)$ The number of $C_l$-frames is given by 
$|\cC_l|=135\cdot \binom{8}{l}$. 
\qed
\end{prop}
\section{\boldmath ORG $\tau$-functions}
\label{sec:2}
In this section we introduce the notion of 
{\em ORG $\tau$-functions}, which is a reformulation 
of $\tau$-functions associated with the $E_8$ lattice 
proposed by Ohta--Ramani--Grammaticos \cite{ORG2001}. 
\par\medskip
We fix once for all a nonzero entire function $[z]$ 
in $z\in\bC$ satisfying the three-term relation
\begin{equation}\label{eq:three-term}
[\beta\pm \gamma][z\pm \alpha]+
[\gamma\pm \alpha][z\pm \beta]+
[\alpha\pm \beta][z\pm \gamma]=0
\qquad(z,\alpha,\beta,\gamma\in\bC). 
\end{equation}
Throughout this paper, we use the abbreviation 
$[\alpha\pm\beta]=[\alpha+\beta][\alpha-\beta]$ 
with a double sign 
indicating the product of two factors. 
From \eqref{eq:three-term} it follows that 
that $[z]$ is an 
odd function ($[-z]=-[z], [0]=0$). 
We remark that the three-term relation \eqref{eq:three-term} 
can be written alternatively as
\begin{equation}
[z\pm \alpha][w\pm \beta]-
[z\pm \beta][w\pm \alpha]=
[z\pm w][\alpha\pm\beta]\qquad(z,w,\alpha,\beta\in\bC), 
\end{equation}
or 
\begin{equation}
\frac{[z\pm\alpha]}{[z\pm\beta]}
-
\frac{[w\pm\alpha]}{[w\pm\beta]}=
\frac{[z\pm w][\alpha\pm\beta]}
{[z\pm\beta][w\pm\beta]}\qquad(z,w,\alpha,\beta\in\bC). 
\end{equation}

It is known that 
the functional equation \eqref{eq:three-term} for $[z]$ 
implies that 
the set of zeros
$\Omega=\brm{\omega\in\bC}{[\omega]=0}$ 
form a closed discrete subgroup of the additive group 
$\bC$. 
Furthermore, 
such a function $[z]$ belongs to 
one of the following three classes, 
{\em rational}, {\em trigonometric} or {\em elliptic}, 
according to the rank of $\Omega$ (\cite{WW1927}):
\begin{equation}
\begin{array}{llll}
(0)& \mbox{rational} 
&: [z]=e(c_0z^2+c_1)\,z &(\Omega=0),\\
(1)& \mbox{trigonometric} &: [z]=e(c_0z^2+c_1)\,
\sin(\pi z/\omega_1)&(\Omega=\bZ\,\omega_1),\\
(2)& \mbox{elliptic} 
&: [z]=e(c_0z^2+c_1)\,\sigma(z|\Omega) &(\Omega=
\bZ\,\omega_1\oplus\bZ\,\omega_2), 
\end{array}
\end{equation}
where $e(z)=e^{2\pi\sqrt{-1}z}$, and 
$c_0,c_1\in\bC$.  In the elliptic case, 
$\Omega$ is generated by complex numbers 
$\omega_1,\omega_2$ which are linearly independent 
over $\bR$, and 
$\sigma(z|\Omega)$ stands for the 
Weierstrass sigma function associated with the 
period lattice $\Omega=\bZ\,\omega_1\oplus\bZ\,\omega_2$. 
In the trigonometric and elliptic cases, 
$[z]$ is quasi-periodic with respect to 
$\Omega$ in the following sense:
\begin{equation}
[z+\omega]=\ep_\omega 
e(\eta_\omega(z+\tfrac{\omega}{2}))[z]
\quad(\omega\in \Omega), 
\end{equation}
where 
$\eta_\omega\in\bC$ $(\omega\in\Omega)$ 
are constants 
such that 
$\eta_{\omega+\omega'}=\eta_\omega+\eta_{\omega'}$ 
$(\omega,\omega'\in\Omega)$, and  
$\ep_\omega=+1$ or $-1$ according as 
$\omega\in 2\Omega$ or $\omega\not\in 2\Omega$. 

\par\medskip
In what follows we fix a nonzero constant $\delta\in\bC$ 
such that $\bZ\delta\cap\Omega=\{0\}$. 
Let $D$ be a subset of $V=\bC^8$ stable under the 
translation by $P\delta$, namely $D+P\delta=D$.  
\begin{dfn}\label{dfn:ORGtaufn}\rm 
A function $\tau(x)$ defined over $D$ is called 
an {\em ORG $\tau$-function} if it satisfies the 
non-autonomous Hirota equations
\begin{equation}\label{eq:bilinHirota}
{\rm H}(a,b,c):\quad
[\ipr{b\pm c}{x}]\tau(x\pm a\delta)+
[\ipr{c\pm a}{x}]\tau(x\pm b\delta)+
[\ipr{a\pm b}{x}]\tau(x\pm c\delta)=0
\end{equation}
for all $C_3$-frames $\br{\pm a, \pm b, \pm c}$ 
relative to $P$. 
\end{dfn}

A $C_3$-frame $\br{\pm a, \pm b, \pm c}$ 
defines an octahedron in $V$ 
of which the twelve edges and the three diagonals 
are vectors in the $E_8$ lattice $P$.  
Hence, if one of the six vertices  
$\br{x\pm a\delta, x\pm b\delta, x\pm c\delta}$ 
belongs to $D$, the other five belong to $D$ as well 
by the property of a $C_3$-frame. 
Also, by Proposition \ref{prop:1B} the number of 
$C_3$-frames is
$|\cC_3|=135\cdot\binom{8}{3}=7560$.  
Hence, the equation to be satisfied by 
an ORG $\tau$-function is a system of $7560$ 
non-autonomous Hirota equations, 
which we call the {\em ORG system} of type $E_8$.
(A bilinear equation of the form \eqref{eq:bilinHirota} 
is also called a {\em Hirota-Miwa equation}.)

In Definition \ref{dfn:ORGtaufn}, as the independent 
variables of $\tau(x)$ one can take both discrete and
continuous variables.  The two extreme cases of the 
domain $D$ are: 
\begin{equation}
(1)\ \ D=v+P\delta\quad(\mbox{fully discrete}),
\qquad
(2)\ \ D=V\qquad(\mbox{fully continuous}). 
\end{equation}
There are intermediate cases where $D$ is a disjoint 
union of a countable family of affine subspaces. 
In such cases, we assume that 
$\tau(x)$ is a holomorphic (or meromorphic) function
on $D$. 
\begin{figure}
$$
\unitlength=0.9pt
\begin{picture}(240,200)(-20,0)
\put(0,100){\line(1,0){200}}
\put(100,0){\line(0,1){200}}
\put(100,100){\line(2,3){60}}
\put(100,100){\line(-2,-3){60}}
{\thicklines
\multiput(20,100)(0.5,0.5){160}{\line(1,0){1}}
\multiput(20,100)(3,-3){27}{\line(1,0){1}} %
\multiput(20,100)(0.345,-0.50){118}{\line(1,0){1}}
\multiput(20,100)(3,1.49){40}{\line(1,0){1}} %
\multiput(180,100)(-0.5,0.5){160}{\line(1,0){1}}
\multiput(180,100)(-0.5,-0.5){160}{\line(1,0){1}}
\multiput(180,100)(-0.345,0.5){118}{\line(1,0){1}}
\multiput(180,100)(-0.5,-0.245){240}{\line(1,0){1}}
\multiput(100,180)(0.5,-0.27){80}{\line(1,0){1}}
\multiput(100,180)(-0.143,-0.5){278}{\line(1,0){1}}
\multiput(100,20)(-0.5,0.27){80}{\line(1,0){1}}
\multiput(100,20)(0.85,3){46}{\line(1,0){1}} %
\put(20,35){$x+a\delta$}
\put(148,160){$x-a\delta$}
\put(185,105){$x+b\delta$}
\put(-17,105){$x-b\delta$}
\put(58,180){$x+c\delta$}
\put(108,12){$x-c\delta$}
\put(105,90){$x$}
}
\end{picture}
$$
\centerline{\small $[\ipr{b\pm c}{x}]\tau(x\pm a\delta)
+[\ipr{c\pm a}{x}]\tau(x\pm b\delta)
+[\ipr{a\pm b}{x}]\tau(x\pm c\delta)=0$}
\caption{Non-autonomous Hirota equation}
\end{figure}

\begin{prop}\label{prop:2A}\ 
For any constant $c\in\bC$, 
the entire function
\begin{equation}\label{eq:cansol}
\tau(x)=[\tfrac{1}{2\delta}\ipr{x}{x}+c]\qquad(x\in V)
\end{equation}
is an ORG $\tau$-function on $V$.
\end{prop}
\begin{proof}{Proof}
Noting that 
$\tau(x\pm a\delta)=
[\tfrac{1}{2\delta}\ipr{x}{x}+\tfrac{\delta}{2}+c
\pm\ipr{a}{x}]$, set
\begin{equation}
z=\tfrac{1}{2\delta}\ipr{x}{x}+\tfrac{\delta}{2}+c,
\quad
\alpha=\ipr{a}{x},\quad
\beta=\ipr{b}{x},\quad
\gamma=\ipr{c}{x}.
\end{equation}
Then the Hirota equation ${\rm H}(a,b,c)$ 
reduces to the functional equation \eqref{eq:three-term}. 
\end{proof}
The ORG $\tau$-function \eqref{eq:cansol} can be regarded 
as the {\em canonical solution} of the ORG system. 
We give below some remarks 
on transformations of an ORG $\tau$-function. 
\begin{thm}\label{thm:2B}\ 
Let $D$ be a subset of $V$ with $D+P\delta=D$, and 
$\tau(x)$ an ORG $\tau$-function on $D$. 
\newline
$(1)$ $($Multiplication by an exponential function\,$)$\ 
For any constants $k,c\in\bC$, 
any vector $v\in V$ and $\ep=\pm1$, the function
\begin{equation}
\widetilde{\tau}(x)=e\big(k\ipr{x}{x}+\ipr{v}{x}+c\big)\tau(\ep x)
\qquad(x\in \ep D)
\end{equation}
is an ORG $\tau$-function on $\ep D$. 
\newline
$(2)$ $($Transformation by $W(E_8)$\,$)$\ 
For any $w\in W(E_8)$, the function $w.\tau$ defined by 
\begin{equation}
(w.\tau)(x)=\tau(w^{-1}.x)\qquad (x\in w.D)
\end{equation}
is an ORG $\tau$-function on $w.D$. 
\newline
$(3)$ $($Translation by a period\,$)$\ 
For any period $\omega\in\Omega$ and any $v\in P$, 
the function $\widetilde{\tau}$ defined by 
\begin{equation}
\begin{split}
&\widetilde{\tau}(x)=e(S(x;v,\omega))\tau(x-v\omega)
\qquad(x\in D+v\omega),
\\
&S(x;v,\omega)=\tfrac{\eta_\omega}
{2\delta^2}\ipr{v}{x}\ipr{x}{x-v\omega}, 
\end{split}
\end{equation}
is an ORG $\tau$-function on $D+v\omega$. 
\end{thm}

\begin{proof}{Proof} 
Since Statements (1) and (2) are straightforward, 
we give a proof of (3) only.   Set $y=x-v\omega\in D$ 
so that
\begin{equation}
[\ipr{b\pm c}{x}]\widetilde{\tau}(x\pm a\delta)=
[\ipr{b\pm c}{y+v\omega}]
e(S(x\pm a\delta;v,\omega))\tau(y\pm a\delta).
\end{equation}
Since
\begin{equation}
\begin{split}
[\ipr{b+c}{y+v\omega}]
&=[\ipr{b+c}{y}+\ipr{b+c}{v}\omega]
\\
&=\ep_{\ipr{b+c}{v}\omega}e(\eta_{\omega}
\ipr{b+c}{v}\ipr{b+c}{y+v\tfrac{\omega}{2}})[\ipr{b+c}{y}], 
\end{split}
\end{equation}
we have
\begin{equation}
\begin{split}
&[\ipr{b\pm c}{y+v\omega}]
\\
&=\ep_{\ipr{b+c}{v}\omega}\ep_{\ipr{b-c}{v}\omega}
[\ipr{b\pm c}{y}]
e\left(2\eta_{\omega}
\ipr{\ipr{b}{v}b+\ipr{c}{v}c}{y+v\tfrac{\omega}{2}}
\right)
\\
&=\ep_{\ipr{b+c}{v}\omega}\ep_{\ipr{b-c}{v}\omega}
[\ipr{b\pm c}{y}]
e\left(2\eta_{\omega}
\ipr{\ipr{b}{v}b+\ipr{c}{v}c}{x-v\tfrac{\omega}{2}}
\right).
\end{split}
\end{equation}
On the other hand, 
\begin{equation}
S(x+a\delta;v,w)+S(x-a\delta;v,w)
=
\tfrac{\eta_\omega}{\delta^2}\ipr{v}{x}\big(\ipr{v}{x-v\omega}+\delta^2\big)
+2\eta_\omega\ipr{a}{v}\ipr{a}{x-v\tfrac{\omega}{2}}. 
\end{equation}
This implies
\begin{equation}
\begin{split}
&[\ipr{b\pm c}{x}]\widetilde{\tau}(x\pm a\delta)
\\
&=
\ep_{\ipr{b+c}{v}\omega}\ep_{\ipr{b-c}{v}\omega}
[\ipr{b\pm c}{y}]\tau(y\pm a\delta)
\\
&\quad\cdot
e\big(
\tfrac{\eta_{\omega}}{\delta^2}
\ipr{v}{x}(\ipr{v}{x-v\omega}+\delta^2)
\big)
e\big(2\eta_{\omega}
\ipr{\ipr{a}{v}a+\ipr{b}{v}b+\ipr{c}{v}c}
{x-v\tfrac{\omega}{2}}\big).
\end{split}
\end{equation}
Hence, validity 
of the Hirota equation for $\widetilde{\tau}(x)$ reduces 
to proving
\begin{equation}
\ep_{\ipr{b+c}{v}\omega}\ep_{\ipr{b-c}{v}\omega}
=
\ep_{\ipr{c+a}{v}\omega}\ep_{\ipr{c-a}{v}\omega}
=
\ep_{\ipr{a+b}{v}\omega}\ep_{\ipr{a-b}{v}\omega}. 
\end{equation}
Since this holds trivially for $\omega\in 2\Omega$, 
we assume $\omega\notin 2\Omega$. 
In view of the transitive action of 
$W(E_8)$ on $\cC_3$, we may assume 
$\br{\pm a,\pm b,\pm c}=\br{\pm v_0,\pm v_1,\pm v_2}$. 
Then, for distinct $i,j\in\br{0,1,2}$, we have 
\begin{equation}
\begin{array}{cl}
v\in\bZ^8\quad&\Longrightarrow\quad 
\ipr{v_i+v_j}{v}\equiv \ipr{v_i-v_j}{v}\ \ \mbox{mod}\ 2,
\\[4pt]
v\in\phi+\bZ^8\quad&\Longrightarrow\quad 
\ipr{v_i+v_j}{v}\not\equiv \ipr{v_i-v_j}{v}\ \ \mbox{mod}\ 2.
\end{array}
\end{equation}
Since 
\begin{equation}
\ep_{k\omega}\ep_{l\omega}=
\left\{\begin{array}{ll}
+1 & (k\equiv l\ \ \mbox{\rm mod}\ 2),\\[4pt]
-1 & (k\not\equiv l\ \ \mbox{\rm mod}\ 2)
\end{array}
\right.
\end{equation}
for $\omega\notin 2\Omega$, 
$\ep_{\ipr{v_i+v_j}{v}\omega}
\ep_{\ipr{v_i-v_j}{v}\omega}$ takes the value 
$+1$ or $-1$ according as $v\in P$ belongs 
to $\bZ^8$ or $\phi+\bZ^8$, 
regardless of the choice of the pair $i,j$. 
This completes the proof of (3). 
\end{proof}

We remark that the composition of two translations 
of (3) by $a\,\omega$ and by $b\,\omega$ 
for $a,b\in P$ and $\omega\in\Omega$ 
results essentially in the same transformation as the 
translation by $(a+b)\omega$.  In fact we have 
\begin{equation}
\begin{split}
&e(S(x;b\,\omega)e(S(x-b\,\omega;a,\omega))\tau(x-(a+b)\omega)
\\
&=
e(k\ipr{x}{x}+\ipr{v}{x}+c)
S(x;a+b,\omega)\tau(x-(a+b)\omega)
\end{split}
\end{equation}
for some $k,c\in\bC$ and $v\in V$.  

\section{\boldmath $E_8$ $\tau$-function as an 
infinite chain of $E_7$ $\tau$-functions}
\label{sec:3}

Recall that the root lattice $Q(E_7)$ of type $E_7$ 
is the orthogonal complement of $\phi$ in $P=Q(E_8)$. 
In what follows, we denote by 
\begin{equation}
H_\kappa=\brm{x\in V}{\ipr{\phi}{x}=\kappa}\quad
(\kappa\in\bC)
\end{equation}
the hyperplanes in $V$ defined as the level sets 
of $\varphi=\ipr{\phi}{\cdot}$.  
Fixing a constant $c\in\bC$, 
we now consider the case where 
the domain $D$ of an ORG $\tau$-function is a disjoint 
union of parallel hyperplanes
\begin{equation}
D_c=\bigsqcup_{n\in\bZ}H_{c+n\delta}. 
\end{equation}
Note that each component 
$H_{c+n\delta}$ ($n\in\bZ$) is 
invariant under the action of $W(E_7)$ and 
the translation by $Q(E_7)\delta$, and 
that the whole set $D$ is invariant under 
the translation by $Q(E_8)\delta$. 
In this situation, 
we regard a function 
$\tau(x)$ on $D_c$ as an infinite family of functions 
$\tau^{(n)}(x)$ on $H_{c+n\delta}$ $(n\in \bZ)$ 
defined by restriction as 
$\tau^{(n)}=\tau\big|_{H_{c+n\delta}}$ for $n\in\bZ$. 
In order to investigate the system of Hirota equations
for $\tau^{(n)}(x)$ ($n\in\bZ$), 
we classify them under the action of 
the Weyl group $W(E_7)$. 

For a $C_l$-frame $A=\br{\pm a_0,\ldots,\pm a_{l-1}}$ 
given, we consider the {\em multiset}
\begin{equation}
\varphi(A)=\br{\pm \varphi(a_0),\ldots,\pm \varphi(a_{l-1})},
\end{equation}
where $\varphi(v)=\ipr{\phi}{v}$ ($v\in V$).  
As to the three $C_8$-frames of 
Example \ref{exm:C8frames}, we have
\begin{equation}
\varphi(A_0)=\br{(\pm\hf)^8},
\quad
\varphi(A_1)=\varphi(A_2)=\br{(\pm1)^2, 0^{12}}. 
\end{equation}
Here the symbol $c^n$ (resp.\ $(\pm c)^n$)
indicates that $c$ appears (resp.\ both $+c$ and $-c$ appear)
with multiplicity $n$ in the multiset. 
We say that a $C_8$-frame is {\em of type $\I$} if 
$\varphi(A)=\br{(\pm\hf)^8}$, and {\em of type $\II$}
if $\varphi(A)=\br{(\pm1)^2, 0^{12}}$.  
\begin{equation}
\begin{picture}(100,80)(10,-10)
\unitlength=1.2pt
\multiput(50,30)(0.10,1){19}{\line(0,1){1}}
\multiput(50,30)(0.30,1){24}{\line(0,1){1}}
\multiput(50,30)(0.50,1){20}{\line(0,1){1}}
\multiput(50,30)(0.70,1){22}{\line(0,1){1}}
\multiput(50,30)(-0.08,1){24}{\line(0,1){1}}
\multiput(50,30)(-0.30,1){19}{\line(0,1){1}}
\multiput(50,30)(-0.46,1){23}{\line(0,1){1}}
\multiput(50,30)(-0.74,1){21}{\line(0,1){1}}
\multiput(50,30)(0.10,-1){19}{\line(0,1){1}}
\multiput(50,30)(0.30,-1){24}{\line(0,1){1}}
\multiput(50,30)(0.50,-1){20}{\line(0,1){1}}
\multiput(50,30)(0.70,-1){22}{\line(0,1){1}}
\multiput(50,30)(-0.08,-1){24}{\line(0,1){1}}
\multiput(50,30)(-0.30,-1){19}{\line(0,1){1}}
\multiput(50,30)(-0.46,-1){23}{\line(0,1){1}}
\multiput(50,30)(-0.74,-1){21}{\line(0,1){1}}
\put(18,-8){\small $C_8$-frame of type I}
\put(120,20){\vector(0,1){35}}
\put(122,37){$\varphi$}
\end{picture}
\qquad\qquad\qquad
\begin{picture}(100,80)(10,-10)
\unitlength=1.2pt
\multiput(50,30)(0.16,1){30}{\line(0,1){1}}
\multiput(50,30)(-0.16,1){30}{\line(0,1){1}}
\multiput(50,30)(0.16,-1){30}{\line(0,1){1}}
\multiput(50,30)(-0.16,-1){30}{\line(0,1){1}}
\multiput(50,30)(0.5,1){15}{\line(0,1){1}}
\multiput(50,30)(0.8,0.6){22}{\line(1,0){1}}
\multiput(50,30)(1,0.2){26}{\line(1,0){1}}
\multiput(50,30)(0.5,-1){15}{\line(0,1){1}}
\multiput(50,30)(0.8,-0.6){22}{\line(1,0){1}}
\multiput(50,30)(1,-0.2){26}{\line(1,0){1}}
\multiput(50,30)(-0.5,1){15}{\line(0,1){1}}
\multiput(50,30)(-0.8,0.6){22}{\line(1,0){1}}
\multiput(50,30)(-1,0.2){26}{\line(1,0){1}}
\multiput(50,30)(-0.5,-1){15}{\line(0,1){1}}
\multiput(50,30)(-0.8,-0.6){22}{\line(1,0){1}}
\multiput(50,30)(-1,-0.2){26}{\line(1,0){1}}
\put(17,-8){\small $C_8$-frame of type $\II$}
\end{picture}
\end{equation}

\begin{prop}\label{prop:3A}
Any $C_8$-frame is either of type $\I$ or of type $\II$. 
Furthermore, these two types give the decomposition
of the set $\cC_8$ of all $C_8$-frames into 
$W(E_7)$-orbits\,$:$
\begin{equation}
\begin{split}
&\cC_8=\cC_{8,\sI}\sqcup \cC_{8,\sII},
\qquad|\cC_{8,\sI}|=72,\quad|\cC_{8,\sII}|=63; 
\\
&\cC_{8,\sI}=W(E_7)A_0,\quad
\cC_{8,\sII}=W(E_7)A_1=W(E_7)A_2. 
\end{split}
\end{equation}
\end{prop}
In order to analyze the $W(E_7)$-orbits in $\cC_8$, 
we first decompose $P_{N=4}$ into $W(E_7)$-orbits.  
As a result, 
$P_{N=4}$ decomposes 
into the form
\begin{equation}
P_{N=4}=
P_{N=4,\varphi=2}\sqcup
P_{N=4,\varphi=1}\sqcup
P_{N=4,\varphi=0}\sqcup
P_{N=4,\varphi=-1}\sqcup
P_{N=4,\varphi=-2},
\end{equation}
and each level set of $\varphi$ forms a 
single $W(E_7)$-orbit. 
The five $W(E_7)$-orbits are described as follows.  
\begin{equation}
\arraycolsep=4pt
\renewcommand{\arraystretch}{1.2}
\begin{array}{|c||c|c|c|c|c|}
\hline
\varphi & 2 & 1 & 0 & -1 & -2\\
\hline\hline
\mbox{representative}&
\phi-v_0+v_1 & \phi-2v_0 & \phi-2v_0-v_6-v_7 &
-2v_0 & -\phi-v_0+v_1\\
\hline
\mbox{stabilizer}&
W(D_6)&W(A_6)&W(D_5\times A_1)&W(A_6)&W(D_6)
\\
\hline
\mbox{cardinality}&
126 & 576 & 756 & 576 & 126 
\\
\hline
\end{array}
\end{equation}
The ``representative'' indicates a unique vector 
$v$ in the orbit such that 
$\ipr{\alpha_j}{v}\ge0$ ($j=0,1,\ldots,6$).  
According to this decomposition of $P_{N=4}$, 
$(\hf P)_{N=1}=\hf(P_{N=4})$ decomposes into the 
five $W(E_7)$-orbits with $\varphi=1,\hf,0,-\hf,-1$.  
Proposition \ref{prop:3A} follows from the fact 
that 
$\hf\phi-v_0$ and $v_0$ belong to 
$C_8$-frames 
of type $\I$, 
and 
$\hf(\phi-v_0+v_1)$, 
$\hf(\phi-v_6-v_7)-v_0$, 
and 
$\hf(-\phi-v_0+v_1)$
to those of type $\II$. 
Note that, among the 63 $C_8$-frames of type $\II$, 
35 are obtained from $A_1$, 
and 28 from $A_2$ of Example \ref{exm:C8frames} 
by the action of the symmetric group $\mathfrak{S}_8
\subset W(E_7)$.  

We remark that in any 
$C_8$-frame $\br{\pm a_0,\pm a_1,\ldots,\pm a_7}$ with 
\begin{equation}
(\I):\quad
\ipr{\phi}{a_j}=\hf\quad(j=0,1,\ldots,7), 
\end{equation}
we have 
$\hf(a_0+a_1+\cdots+a_7)=\phi$. 
Also, in any $C_8$-frame $\br{\pm a_0,\pm a_1,\ldots,\pm a_7}$ with 
\begin{equation}
(\II):\quad
\ipr{\phi}{a_0}=\ipr{\phi}{a_7}=1,\quad
\ipr{\phi}{a_j}=0\quad(j=1,\ldots,6), 
\end{equation}
we have $a_0+a_7=\phi$.  
Since these statements are $W(E_7)$-invariant, 
by Proposition \ref{prop:3A} 
we have 
only to check the cases of $A_0$ and $A_1$
of Example \ref{exm:C8frames}, respectively. 

By Proposition \ref{prop:1B}, each $C_3$-frame is 
contained in a unique $C_8$-frame.  Hence, by 
Proposition \ref{prop:3A} we obtain the following 
classification of $C_3$-frames. 
\begin{prop}\label{prop:3B}
The set $\cC_3$ of all $C_3$-frames decomposes 
into four $W(E_7)$-orbits\,$:$
\begin{equation}
\cC_{3}=\cC_{3,\I}
\sqcup \cC_{3,\II_0}
\sqcup \cC_{3,\II_1}
\sqcup \cC_{3,\II_2}. 
\end{equation}
These four $W(E_7)$-orbits are characterized as follows. 
\begin{equation}
\renewcommand{\arraystretch}{1.2}
\begin{array}{|c||c|c|c|c|}
\hline
\mbox{\rm type} & \I & \II_0 & \II_1 & \II_2\\
\hline\hline
\varphi & (\pm\hf)^3 & 0^6 & (\pm1)\,0^4 & (\pm1)^2\,0^2\\
\hline
\mbox{\rm cardinality}&
56\cdot 72 & 20\cdot 63 & 30\cdot 63 & 6\cdot 63
\\
\hline
\end{array}
\end{equation}
\end{prop}

\begin{figure}
$$
\begin{picture}(380,65)(-60,-5)
\unitlength=1.2pt
\put(35,25){\line(0,1){12}}
\put(35,25){\line(3,2){16}}
\put(35,25){\line(-3,2){16}}
\put(35,25){\line(0,-1){12}}
\put(35,25){\line(3,-2){16}}
\put(35,25){\line(-3,-2){16}}
\put(0,40){\small (\I)}
\put(105,25){\line(1,0){30}}
\put(105,25){\line(6,1){24}}
\put(105,25){\line(6,-1){24}}
\put(105,25){\line(-1,0){30}}
\put(105,25){\line(-6,1){24}}
\put(105,25){\line(-6,-1){24}}
\put(70,40){\small $(\II_0)$ }
\put(175,25){\line(0,1){22}}
\put(175,25){\line(6,1){24}}
\put(175,25){\line(6,-1){24}}
\put(175,25){\line(0,-1){22}}
\put(175,25){\line(-6,1){24}}
\put(175,25){\line(-6,-1){24}}
\put(140,40){\small $(\II_1)$ }
\put(245,25){\line(1,0){22}}
\put(245,25){\line(1,5){4.5}}
\put(245,25){\line(1,-5){4.5}}
\put(245,25){\line(-1,0){22}}
\put(245,25){\line(-1,5){4.5}}
\put(245,25){\line(-1,-5){4.5}}
\put(210,40){\small $(\II_2)$ }
\put(-33,25){\vector(0,1){24}}
\put(-50,30){\line(1,0){40}}
\put(-60,20){\line(1,0){40}}
\put(-60,20){\line(1,1){10}}
\put(-20,20){\line(1,1){10}}
\put(-45,40){$\varphi$}
\put(10,-10){\small $56\cdot 72=4032$}
\put(80,-10){\small $20\cdot 63=1260$}
\put(150,-10){\small $30\cdot 63=1890$}
\put(225,-10){\small $6\cdot 63=378$}
\end{picture}
$$
\caption{Four types of 7560 $C_3$-frames}
\end{figure}
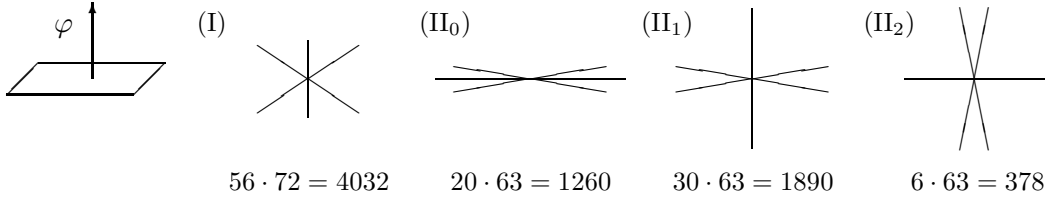

According to the four types of $C_3$-frames, 
the Hirota equations for $\tau^{(n)}(x)$ 
are classified as follows. 
For each $C_3$-frame $\br{\pm a_0,\pm a_1,\pm a_2}$
of type $\I$ with
\begin{equation}\label{eq:C3frameI}
(\I):\quad \ipr{\phi}{a_0}=\ipr{\phi}{a_1}=
\ipr{\phi}{a_2}=\hf, 
\end{equation}
the Hirota equation $H(a_0,a_1,a_2)$ takes the form
\begin{equation}\label{eq:HirotaEqI}
\begin{split}
(\I)_{n+1/2}:\quad
&[\ipr{a_1\pm a_2}{x}]\tau^{(n)}(x\!-\!a_0\delta)
\tau^{(n+1)}(x\!+\!a_0\delta)
\\
&\mbox{}
+
[\ipr{a_2\pm a_0}{x}]\tau^{(n)}(x\!-\!a_1\delta)
\tau^{(n+1)}(x\!+\!a_1\delta)
\\
&\mbox{}
+[\ipr{a_0\pm a_1}{x}]\tau^{(n)}(x\!-\!a_2\delta)
\tau^{(n+1)}(x\!+\!a_2\delta)=0
\end{split}
\end{equation}
for $x\in H_{c+(n+1/2)\delta}$. 
This bilinear equation describes the relationship
({\em B\"acklund transformation})
between the two $\tau$-functions $\tau^{(n)}(x)$ 
and $\tau^{(n+1)}(x)$ on $H_{c+n\delta}$ and $H_{c+(n+1)\delta}$, 
respectively. 
When $\br{\pm a_0,\pm a_1,\pm a_2}$ is $C_3$-frame 
of type $\II_0$, $\II_1$, $\II_2$, 
 we choose $a_0$, $a_1$, $a_2$ so that 
\begin{equation}\label{eq:C3frameII}
\begin{split}
(\II_0):\quad &\ipr{\phi}{a_0}=\ipr{\phi}{a_1}=\ipr{\phi}{a_2}=0,
\\[2pt]
(\II_1):\quad &\ipr{\phi}{a_0}=1,\quad \ipr{\phi}{a_1}=\ipr{\phi}{a_2}=0,
\\[2pt]
(\II_2):\quad &\ipr{\phi}{a_0}=\ipr{\phi}{a_1}=1,\quad\ipr{\phi}{a_2}=0.
\end{split}
\end{equation}
Then the corresponding Hirota equations are given by
\begin{equation} \label{eq:HirotaEqII}
\begin{split}
(\II_0)_{n}:\quad
&
[\ipr{a_1\pm a_2}{x}]\tau^{(n)}(x\pm a_0\delta)+[\ipr{a_2\pm a_0}{x}]\tau^{(n)}(x\pm a_1\delta)
\\
&\quad\mbox{}
+[\ipr{a_0\pm a_1}{x}]\tau^{(n)}(x\pm a_2\delta)=0, 
\\
(\II_1)_{n}:\quad
&[\ipr{a_1\pm a_2}{x}]\tau^{(n-1)}(x-a_0\delta)\tau^{(n+1)}(x+a_0\delta)
\\
&=
[\ipr{a_0\pm a_2}{x}]\tau^{(n)}(x\pm a_1\delta)
-[\ipr{a_0\pm a_1}{x}]\tau^{(n)}(x\pm a_2\delta), 
\\
(\II_2)_{n}:\quad
&[\ipr{a_1\pm a_2}{x}]\tau^{(n-1)}(x-a_0\delta)\tau^{(n+1)}(x+a_0\delta)
\\
&\mbox{}
-[\ipr{a_0\pm a_2}{x}]\tau^{(n-1)}(x-a_1\delta)\tau^{(n+1)}(x+a_1\delta)
\\
&=
[\ipr{a_1\pm a_0}{x}]\tau^{(n)}(x\pm a_2\delta)
\end{split}
\end{equation}
for $x\in H_{c+n\delta}$. 
Note that the Hirota equation of type $(\II_0)_n$ is an 
equation for $\tau^{(n)}(x)$ on $H_{c+n\delta}$ only, while 
those of types $(\II_1)_n$ and $(\II_2)_n$ are equations  
among the three $\tau$-functions 
$\tau^{(n-1)}(x)$, $\tau^{(n)}(x)$, $\tau^{(n+1)}(x)$. 
A bilinear equation of type $(\II_1)_{n}$ can be regarded 
as a discrete version of the {\em Toda equation.} 

For each $n\in\bZ$, 
the system of 1260 Hirota equations $(\II_0)_{n}$ 
for $\tau^{(n)}(x)$ on $H_{c+n\delta}$ can be regarded as 
an ORG system of type $E_7$.  In this way, 
the whole ORG system of type $E_8$ 
for $\tau(x)$ on $D_c$ 
can be regarded as 
an infinite chain of ORG systems of type $E_7$
for $\tau^{(n)}(x)$ on $H_{c+n\delta}$ ($n\in\bZ$). 

\begin{thm}\label{thm:3C}
For an integer $n\in\bZ$, 
let 
$\tau^{(n-1)}(x)$ and $\tau^{(n)}(x)$ be 
meromorphic functions 
on $H_{c+(n-1)\delta}$ and $H_{c+n\delta}$, respectively.  
Suppose that $\tau^{(n-1)}(x)\not\equiv 0$ and that 
the following two conditions are 
satisfied\,$:$
\par\smallskip\noindent
\quad{\rm (A1)}$:$\quad
$\tau^{(n-1)}(x)$ and $\tau^{(n)}(x)$ 
satisfy all the bilinear equations of type $(\I)_{n-1/2}$.
\par\noindent
\quad{\rm (A2)}$:$\quad
$\tau^{(n)}(x)$ 
satisfies all the bilinear equations of type $(\II_0)_{n}$.
\par\smallskip\noindent
Then there exists a unique meromorphic function 
$\tau^{(n+1)}(x)$ on $H_{c+(n+1)\delta}$ such that
\par\smallskip\noindent
\quad{\rm (B)}$:$\quad
$\tau^{(n-1)}(x)$, $\tau^{(n)}(x)$, $\tau^{(n+1)}(x)$  
satisfy all the bilinear equations of type $(\II_1)_{n}$.
\par\smallskip\noindent
Furthermore, this $\tau^{(n+1)}(x)$ satisfies the 
following conditions\,$:$
\par\smallskip\noindent
\quad{\rm (C1)}$:$\quad
$\tau^{(n-1)}(x)$, $\tau^{(n)}(x)$, $\tau^{(n+1)}(x)$  
satisfy all the bilinear equations of type $(\II_2)_{n}$.
\par\smallskip\noindent
\quad{\rm (C2)}$:$\quad
$\tau^{(n)}(x)$, $\tau^{(n+1)}(x)$  
satisfy all the bilinear equations of type $(\I)_{n+1/2}$.
\par\noindent
\quad{\rm (C3)}$:$\quad
$\tau^{(n+1)}(x)$  
satisfies all the bilinear equations of type $(\II_0)_{n+1}$.
\end{thm}
This theorem can be proved essentially by 
the same argument as that of Masuda \cite[Section 3]{Masuda2011}. 
For completeness, we include a proof of Theorem \ref{thm:3C} 
in Appendix A.  

\section{\boldmath Hypergeometric ORG $\tau$-functions}
\label{sec:4}
Keeping the notations in the previous section, 
we consider an ORG $\tau$-function $\tau=\tau(x)$ 
on 
\begin{equation}
D_c=\bigsqcup_{n\in Z}H_{c+n\delta},
\quad
H_{c+n\delta}=\brm{x\in V}{\ipr{\phi}{x}=c+n\delta}
\quad(n\in\bZ),
\end{equation}
where $c\in\bC$.  For each $n\in\bZ$ 
we denote by $\tau^{(n)}=\tau\big|_{H_{c+n\delta}}$ 
the restriction of $\tau$ to $H_{c+n\delta}$. 
\begin{dfn}\rm
A meromorphic ORG $\tau$-function $\tau(x)$ on 
$D_c$ is called 
a {\em hypergeometric $\tau$-function} if 
$\tau^{(n)}(x)=0$ for $n<0$, and $\tau^{(0)}(x)\not\equiv0$. 
\end{dfn}

We now apply Theorem \ref{thm:3C} for constructing 
hypergeometric $\tau$-functions.
Since $\tau^{(-1)}(x)=0$ $(x\in H_{c-\delta})$,  
for any $C_8$-frame $\br{\pm a_0,\ldots,\pm a_7}$ 
of type $\II$ with 
\begin{equation}\label{eq:C8typeII}
\ipr{\phi}{a_0}=\ipr{\phi}{a_7}=1,\quad
\ipr{\phi}{a_i}=0\quad(i=1,\ldots,6), 
\end{equation}
$\tau^{(0)}(x)$ $(x\in H_{c})$ must satisfy 
the following three types of equations: 
\begin{equation}
\begin{split}
(\II_2)_{0}:\quad&
[\ipr{a_0\pm a_7}{x}]\tau^{(0)}(x\pm a_i\delta)=0,
\\
(\II_1)_{0}:\quad&
[\ipr{a_r\pm a_j}{x}]\tau^{(0)}(x\pm a_i\delta)
=[\ipr{a_r\pm a_i}{x}]\tau^{(0)}(x\pm a_j\delta),
\\
(\II_0)_{0}:\quad&
[\ipr{a_j\pm a_k}{x}]\tau^{(0)}(x\pm a_i\delta)
+[\ipr{a_k\pm a_i}{x}]\tau^{(0)}(x\pm a_j\delta)
\\
&\qquad\qquad\mbox{}
+[\ipr{a_i\pm a_j}{x}]\tau^{(0)}(x\pm a_k\delta)=0, 
\end{split}
\end{equation}
where $r=0,7$ and $i,j,k\in\br{1,\ldots,6}$. 
Noting that $a_0+a_7=\phi$, 
in order to fulfill $(\II_2)_0$ for any $C_8$-frame of type $\II$, 
we consider the case where $c=\omega\in \Omega$ 
is a period of $[z]$, so that 
$[\ipr{a_0+a_7}{x}]=[\ipr{\phi}{x}]=[\omega]=0$. 
Equations of type $(\II_0)_0$ follow from those of type $(\II_0)_1$.  In fact,
since 
\begin{equation}
\tau^{(0)}(x\pm a_j\delta)=\frac{[\ipr{a_0\pm a_j}{x}]}{[\ipr{a_0\pm a_k}{x}]}
\tau^{(0)}(x\pm a_k\delta)
\end{equation}
for any distinct $j,k\in\br{1,\ldots,6}$, equations $(\II_0)_0$ reduce 
to the functional equation \eqref{eq:three-term} of $[z]$.  

\begin{thm}\label{thm:4A}
Let $\omega\in\Omega$ be a period of the function $[z]$. 
Let $\tau^{(0)}(x)$ and $\tau^{(1)}(x)$ be nonzero meromorphic 
functions on 
$H_{\omega}$ and $H_{\omega+\delta}$, respectively.  
Suppose that 
\begin{equation}\label{eq:4AII1}
\frac{\tau^{(0)}(x\pm a_1\delta)}{\tau^{(0)}(x\pm a_2\delta)}
=
\frac{[\ipr{a_0\pm a_1}{x}]}{[\ipr{a_0\pm a_2}{x}]}
\quad(x\in H_{\omega})
\end{equation}
for any $C_3$-frame $\br{\pm a_0,\pm a_1,\pm a_2}$ 
of type $\II_1$ with $\ipr{\phi}{a_0}=1$, and 
\begin{equation}\label{eq:4AI}
\begin{split}
&[\ipr{a_1\pm a_2}{x}]
\tau^{(0)}(x-a_0\delta)
\tau^{(1)}(x+a_0\delta)
+
[\ipr{a_2\pm a_0}{x}]
\tau^{(0)}(x-a_1\delta)
\tau^{(1)}(x+a_1\delta)
\\
&\qquad\qquad\mbox{}
+
[\ipr{a_0\pm a_1}{x}]
\tau^{(0)}(x-a_2\delta)
\tau^{(1)}(x+a_2\delta)
=0
\quad(x\in H_{\omega+\delta/2})
\end{split}
\end{equation}
for any $C_3$-frame $\br{\pm a_0,\pm a_1,\pm a_2}$ 
with $\ipr{\phi}{a_0}=\ipr{\phi}{a_1}=\ipr{\phi}{a_2}=\hf$. 
Then there exists a unique hypergeometric $\tau$-function
$\tau=\tau(x)$ on $D_\omega$ such that 
$\tau^{(n)}(x)=0$ for $n<0$ and 
\begin{equation}
\tau^{(0)}(x)=\tau(x)\quad(x\in H_\omega), 
\quad
\tau^{(1)}(x)=\tau(x)\quad(x\in H_{\omega+\delta}). 
\end{equation}
\end{thm}
\begin{proof}{Proof} 
We apply Theorem \ref{thm:3C} to nonzero meromorphic functions 
$\tau^{(n-1)}(x)$, $\tau^{(n)}(x)$ on 
$H_{\omega+(n-1)\delta}$, $H_{c+n\delta}$,  
for constructing $\tau^{(n+1)}(x)$ on 
$H_{\omega+(n+1)\delta}$ 
recursively for $n=1,2,\ldots$. 
At each step, we need to show that the meromorphic 
function 
$\tau^{(n+1)}(x)$ determined by Theorem \ref{thm:3C} is not 
identically zero. 
If $\tau^{(n+1)}(x)\equiv 0$, the bilinear equations 
$(\II_2)_{n}$ imply
\begin{equation}
[\ipr{a_0\pm a_1}{x}]\tau^{(n)}(x\pm a_2\delta)=0\quad(x\in H_{\omega+n\delta})
\end{equation}
for any $C_3$-frame $\br{\pm a_0,\pm a_1,\pm a_2}$ 
of type $\II_2$ with $\ipr{\phi}{a_0}=\ipr{\phi}{a_1}=1$, $\ipr{\phi}{a_2}=0$.
Since $a_0+a_1=\phi$, 
$[\ipr{a_0+a_1}{x}]=[\omega+n\delta]\ne 0$.  
Also, since $[\ipr{a_0-a_1}{x}]\not\equiv 0$, we have 
$\tau^{(n)}(x\pm a_2\delta)=0$ and hence $\tau^{(n)}(x)\equiv 0$ on 
$H_{\omega+n\delta}$, contrarily to the hypothesis.  
\end{proof}

We now fix a $C_3$-frame $\br{\pm a_0,\pm a_1,\pm a_2}$ 
of type $\II_1$ with 
$\ipr{\phi}{a_0}=1$, $\ipr{\phi}{a_1}=\ipr{\phi}{a_2}=0.$
Then the $\tau$-functions $\tau^{(n)}$ on 
$H_{\omega+n\delta}$ 
for $n=2,3,\ldots$ are uniquely determined by the bilinear 
equations
\begin{equation}\label{eq:Todatype}
\begin{split}
(\II_1)_{n}:\quad
&[\ipr{a_1\pm a_2}{x}]\tau^{(n-1)}(x-a_0\delta)\tau^{(n+1)}(x+a_0\delta)
\\
&
=
[\ipr{a_0\pm a_2}{x}]\tau^{(n)}(x\pm a_1\delta)
-[\ipr{a_0\pm a_1}{x}]\tau^{(n)}(x\pm a_2\delta)
\end{split}
\end{equation}
of Toda type.  
From this recursive structure,  it follows that 
the $\tau$-functions $\tau^{(n)}(x)$ are expressed in terms of 
{\em 2-directional Casorati determinants}.  

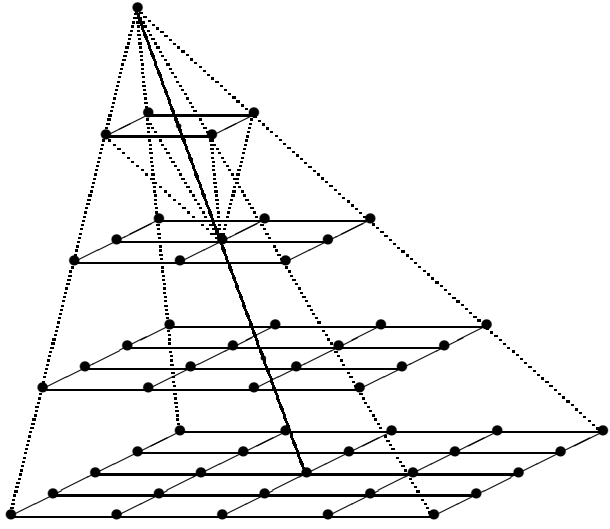
\begin{figure}
$$
\unitlength=0.8pt
\begin{picture}(280,250)
\put(0,0){\line(1,0){200}}
\put(20,10){\line(1,0){200}}
\put(40,20){\line(1,0){200}}
\put(60,30){\line(1,0){200}}
\put(80,40){\line(1,0){200}}
\put(0,0){\line(2,1){80}}
\put(50,0){\line(2,1){80}}
\put(100,0){\line(2,1){80}}
\put(150,0){\line(2,1){80}}
\put(200,0){\line(2,1){80}}
\multiput(-2,-2)(50,0){5}{\small$\bullet$}
\multiput(18,8)(50,0){5}{\small$\bullet$}
\multiput(38,18)(50,0){5}{\small$\bullet$}
\multiput(58,28)(50,0){5}{\small$\bullet$}
\multiput(78,38)(50,0){5}{\small$\bullet$}
\put(15,60){\line(1,0){150}}
\put(35,70){\line(1,0){150}}
\put(55,80){\line(1,0){150}}
\put(75,90){\line(1,0){150}}
\put(15,60){\line(2,1){60}}
\put(65,60){\line(2,1){60}}
\put(115,60){\line(2,1){60}}
\put(165,60){\line(2,1){60}}
\multiput(13,58)(50,0){4}{\small$\bullet$}
\multiput(33,68)(50,0){4}{\small$\bullet$}
\multiput(53,78)(50,0){4}{\small$\bullet$}
\multiput(73,88)(50,0){4}{\small$\bullet$}
\put(30,120){\line(1,0){100}}
\put(50,130){\line(1,0){100}}
\put(70,140){\line(1,0){100}}
\put(30,120){\line(2,1){40}}
\put(80,120){\line(2,1){40}}
\put(130,120){\line(2,1){40}}
\multiput(28,118)(50,0){3}{\small$\bullet$}
\multiput(48,128)(50,0){3}{\small$\bullet$}
\multiput(68,138)(50,0){3}{\small$\bullet$}
\put(45,180){\line(1,0){50}}
\put(65,190){\line(1,0){50}}
\put(45,180){\line(2,1){20}}
\put(95,180){\line(2,1){20}}
\multiput(43,178)(50,0){2}{\small$\bullet$}
\multiput(63,188)(50,0){2}{\small$\bullet$}
\put(58,238){\small$\bullet$}
\multiput(0,0)(0.5,2){120}{\line(1,0){1}}
\multiput(80,40)(-0.2,2){100}{\line(1,0){1}}
\multiput(200,0)(-1.17,2){120}{\line(1,0){1}}
\multiput(280,40)(-2,1.82){110}{\line(1,0){1}}
\multiput(140,20)(-0.365,1){220}{\line(1,0){1}}
\put(120.5,75){\circle*{2}}
\put(80.5,185){\circle*{2}}
\multiput(100,130)(-2.2,2){26}{\line(1,0){1}}
\multiput(100,130)(-1.2,2){30}{\line(1,0){1}}
\multiput(100,130)(0.5,2){31}{\line(1,0){1}}
\multiput(100,130)(-0.2,2){26}{\line(1,0){1}}
\end{picture}
$$
\caption{2-Directional Casorati determinants}
\end{figure}

\begin{thm}\label{thm:4B} 
Under the assumption of Theorem \ref{thm:4A}, suppose 
that $\tau^{(1)}(x)$ on $H_{\omega+\delta}$ is expressed 
in the form $\tau^{(1)}(x)=g^{(1)}(x)\psi(x)$ 
with a nonzero meromorphic function $g^{(1)}(x)$ 
such that 
\begin{equation}
\frac{g^{(1)}(x\pm a_1\delta)}{g^{(1)}(x\pm a_2\delta)}
=
\frac{[\ipr{a_0\pm a_1}{x}]}{[\ipr{a_0\pm a_2}{x}]}
\qquad(x\in H_{\omega+\delta}).
\end{equation}
for a $C_3$-frame $\br{\pm a_0,\pm a_1,\pm a_2}$ 
of type $\II_1$ with $\ipr{\phi}{a_0}=1$, 
$\ipr{\phi}{a_1}=\ipr{\phi}{a_2}=0$. 
Then the 
components $\tau^{(n)}(x)$ of the hypergeometric 
$\tau$-function $\tau(x)$ are expressed as follows 
in terms of 2-directional Casorati determinants\,$:$
\begin{equation}\label{eq:2dirCasorati}
\tau^{(n)}(x)=g^{(n)}(x) K^{(n)}(x),\quad
K^{(n)}(x)=\det\big(\psi^{(n)}_{ij}(x)\big)_{i,j=1}^{n}
\quad(x\in H_{\omega+n\delta})
\end{equation}
for $n=0,1,2,\ldots$, 
where
\begin{equation}
\psi_{ij}^{(n)}(x)=
\psi(x-(n-1)a_0\delta+
(n+1-i-j)a_1\delta+(j-i)a_2\delta)
\end{equation}
for $i,j=1,\ldots,n$. 
The gauge factors $g^{(n)}(x)$ are determined 
inductively from $g^{(0)}(x)=\tau^{(0)}(x)$ and 
$g^{(1)}(x)$ by
\begin{equation}\label{eq:recgamma}
\frac{g^{(n-1)}(x-a_0\delta)
g^{(n+1)}(x+a_0\delta)}
{g^{(n)}(x\pm a_1\delta)}
=
\frac{[\ipr{a_0\pm a_2}{x}]
}{[\ipr{a_1\pm a_2}{x}]}
\quad(x\in H_{\omega+n\delta})
\end{equation}
for $n=1,2,\ldots$. 
\end{thm}

\begin{lem}\label{lem:gauge}
The gauge factors 
$g^{(n)}(x)$ $(x\in H_{\omega+n\delta})$
defined by \eqref{eq:recgamma} 
satisfy 
\begin{equation}\label{eq:condgamma}
\frac{g^{(n)}(x\pm a_1\delta)}
{g^{(n)}(x\pm a_2\delta)}
=
\frac
{[\ipr{a_0\pm a_1}{x}]}
{[\ipr{a_0\pm a_2}{x}]}
\qquad(n=0,1,2,\ldots). 
\end{equation}
\end{lem}
\begin{proof}{Proof}
Formulas \eqref{eq:condgamma} for $n=0,1$ are 
included in the 
assumption for $g^{(0)}(x)=\tau^{(0)}(x)$ and 
$g^{(1)}(x)$.  
For $n=1,2,\ldots$, we show inductively 
that $g^{(n+1)}(x)$ 
defined by \eqref{eq:recgamma} satisfies 
this condition.
From \eqref{eq:condgamma} for $g^{(n)}(x)$, 
we have 
\begin{equation}
\begin{split}
g^{(n+1)}(x+a_0\delta)
=
\frac{[\ipr{a_0\pm a_2}{x}]}
{[\ipr{a_1\pm a_2}{x}]}
\frac{g^{(n)}(x\pm a_1\delta)}
{g^{(n-1)}(x-a_0\delta)}
=
\frac{[\ipr{a_0\pm a_1}{x}]}
{[\ipr{a_1\pm a_2}{x}]}
\frac{g^{(n)}(x\pm a_2\delta)}
{g^{(n-1)}(x-a_0\delta)}
\end{split}
\end{equation}
and hence
\begin{equation}
\begin{split}
g^{(n+1)}(x)
&=
\frac{[\ipr{a_0\pm a_2}{x}-\delta]}
{[\ipr{a_1\pm a_2}{x}]}
\frac{g^{(n)}(x-a_0\delta\pm a_1\delta)}
{g^{(n-1)}(x-2a_0\delta)}
\\
&=
\frac{[\ipr{a_0\pm a_1}{x}-\delta]}
{[\ipr{a_1\pm a_2}{x}]}
\frac{g^{(n)}(x-a_0\delta\pm a_2\delta)}
{g^{(n-1)}(x-2a_0\delta)}.  
\end{split}
\end{equation}
From these two expressions of $g^{(n+1)}(x)$ 
we obtain 
\begin{equation}
\begin{split}
g^{(n+1)}(x\pm a_2\delta)
&=
\frac{[\ipr{a_0\pm a_2}{x}][\ipr{a_0\pm a_2}{x}-2\delta]}
{[\ipr{a_1\pm a_2}{x}\pm\delta]}
\frac{g^{(n)}(x-a_0\delta\pm a_1\delta\pm a_2\delta)
}
{g^{(n-1)}(x-2a_0\delta\pm a_2\delta)}
\\
g^{(n+1)}(x\pm a_1\delta)
&=
\frac{[\ipr{a_0\pm a_1}{x}][\ipr{a_0\pm a_1}{x}-2\delta]}
{[\ipr{a_1\pm a_2}{x}\pm\delta]}
\frac{g^{(n)}(x-a_0\delta\pm a_1\delta\pm a_2\delta)
}
{g^{(n-1)}(x-2a_0\delta\pm a_1\delta)}. 
\end{split}
\end{equation}
Then by \eqref{eq:condgamma} for $g^{(n-1)}(x)$ 
we obtain 
\begin{equation}
\begin{split}
\frac{g^{(n+1)}(x\pm a_1\delta)
}{g^{(n+1)}(x\pm a_2\delta)}
&=
\frac{[\ipr{a_0\pm a_1}{x}][\ipr{a_0\pm a_1}{x}-2\delta]}
{[\ipr{a_0\pm a_2}{x}][\ipr{a_0\pm a_2}{x}-2\delta]}
\frac{g^{(n-1)}(x-2a_0\delta\pm a_2\delta)}{g^{(n-1)}(x-2a_0\delta\pm a_1\delta)}
\\
&=
\frac{[\ipr{a_0\pm a_1}{x}]}
{[\ipr{a_0\pm a_2}{x}]}
\end{split}
\end{equation}
as desired. 
\end{proof}
\begin{proof}{Proof of Theorem \ref{thm:4B}}
Using the gauge factors $g^{(n)}(x)$ defined as above, 
we set 
\begin{equation}
\tau^{(0)}(x)=g^{(0)}(x),\quad
\tau^{(1)}(x)=g^{(1)}(x)\psi(x), 
\end{equation}
and define $K^{(n)}(x)$ by 
\begin{equation}
\tau^{(n)}(x)=g^{(n)}(x)K^{(n)}(x)\quad(n=0,1,2,\ldots).
\end{equation}
Then the bilinear equation \eqref{eq:Todatype} is written as 
\begin{equation}
\begin{split}
&[\ipr{a_1\pm a_2}{x}]
g^{(n-1)}(x-a_0\delta)g^{(n+1)}(x+a_0\delta)
\\
&\quad\cdot
K^{(n-1)}(x-a_0\delta)K^{(n+1)}(x+a_0\delta)
\\
&
=
[\ipr{a_0\pm a_2}{x}]g^{(n)}(x\pm a_1\delta)
K^{(n)}(x\pm a_1\delta)
\\
&\quad\mbox{}
-[\ipr{a_0\pm a_1}{x}]
g^{(n)}(x\pm a_2\delta)
K^{(n)}(x\pm a_2\delta). 
\end{split}
\end{equation}
By Lemma \ref{lem:gauge}, we have
\begin{equation}
\begin{split}
&[\ipr{a_1\pm a_2}{x}]
g^{(n-1)}(x-a_0\delta)g^{(n+1)}(x+a_0\delta)
\\
&=
[\ipr{a_0\pm a_2}{x}]g^{(n)}(x\pm a_1\delta)
=
[\ipr{a_0\pm a_1}{x}]
g^{(n)}(x\pm a_2\delta).  
\end{split}
\end{equation}
Therefore, the main factors $K^{(n)}(x)$ are determined by 
\begin{equation}\label{eq:recKn}
K^{(n-1)}(x-a_0\delta)K^{(n+1)}(x+a_0\delta)
=
K^{(n)}(x\pm a_1\delta)
-
K^{(n)}(x\pm a_2\delta)
\end{equation}
for $n=1,2,\ldots$
starting from $K^{(0)}(x)=1$, $K^{(1)}(x)=\psi(x)$. 
For example, we have
\begin{equation}
\begin{split}
&K^{(2)}(x+a_0\delta)=\psi(x\pm a_1\delta)-\psi(x\pm a_2\delta)=
\det\begin{bmatrix}
\psi(x+a_1\delta) & \psi(x+a_2\delta)\\[2pt]
\psi(x-a_2\delta) & \psi(x-a_1\delta)
\end{bmatrix},
\\
&K^{(3)}(x+2a_0\delta)
=
\arraycolsep=2pt
\det\begin{bmatrix}
\psi(x+2a_1\delta) & \psi(x+a_1\delta+a_2\delta) 
& \psi(x+2a_2\delta)\\[2pt]
\psi(x+a_1\delta-a_2\delta) & \psi(x) & \psi(x-a_1\delta+a_2\delta)\\[2pt]
\psi(x-2a_2\delta) & \psi(x-a_1\delta-a_2\delta) 
& \psi(x-2a_1\delta)
\end{bmatrix}. 
\end{split}
\end{equation}
In general, this recurrence \eqref{eq:recKn}
for $K^{(n)}(x)$ is solved by the 
Lewis Carroll formula for the Casorati determinants 
with respect to the two directions $a_1+a_2$ and 
$a_1-a_2$.  Namely, 
for $n=1,2,\ldots$, we have
\begin{equation}
K^{(n)}(x+(n-1)a_0\delta)=
\det\big(\psi(x+(n+1-i-j)a_1\delta+(j-i)a_2\delta)\big)_{i,j=1}^{n}
\end{equation}
with the vectors 
$v_{ij}=(n+1-i-j)a_1+(j-i)a_2$ $(i,j=1,\ldots,n)$ 
arranged as follows.   
\begin{equation}
\begin{picture}(140,140)
\unitlength=1.2pt
\put(0,60){\line(1,0){120}}
\put(60,0){\line(0,1){120}}
\put(58,58){$\bullet$}
\multiput(28,58)(15,15){3}{$\bullet$}
\multiput(43,43)(15,15){3}{$\bullet$}
\multiput(58,28)(15,15){3}{$\bullet$}
\multiput(13,58)(15,15){4}{$\bullet$}
\multiput(28,43)(15,15){4}{$\bullet$}
\multiput(43,28)(15,15){4}{$\bullet$}
\multiput(58,13)(15,15){4}{$\bullet$}
\multiput(15,60)(2,2){22}{\line(1,0){1}}
\multiput(30,60)(2,2){15}{\line(1,0){1}}
\multiput(45,60)(2,2){8}{\line(1,0){1}}
\multiput(60,15)(2,2){22}{\line(1,0){1}}
\multiput(60,30)(2,2){15}{\line(1,0){1}}
\multiput(60,45)(2,2){8}{\line(1,0){1}}
\multiput(15,60)(2,-2){22}{\line(1,0){1}}
\multiput(30,60)(2,-2){15}{\line(1,0){1}}
\multiput(45,60)(2,-2){8}{\line(1,0){1}}
\multiput(60,105)(2,-2){22}{\line(1,0){1}}
\multiput(60,90)(2,-2){15}{\line(1,0){1}}
\multiput(60,75)(2,-2){8}{\line(1,0){1}}
\put(74,53){\small $a_1$}
\put(52,71){\small $a_2$}
\put(108,65){\small $v_{11}$}
\put(93,37){\small $v_{21}$}
\put(78,22){\small $v_{31}$}
\put(63,7){\small $v_{41}$}
\put(93,80){\small $v_{12}$}
\put(78,95){\small $v_{13}$}
\put(63,110){\small $v_{14}$}
\end{picture}
\end{equation}
This implies the expression \eqref{eq:2dirCasorati} 
for $K^{(n)}(x)$. 
\end{proof}
\section{Elliptic hypergeometric integrals}
\label{sec:5}
In this section, we recall fundamental facts concerning 
the elliptic hypergeometric integrals 
of Spiridonov \cite{Spiridonov2003,Spiridonov2005} and 
Rains \cite{Rains2005,Rains2010}. 
\par\medskip
Fixing two bases $p,q\in\bC^\ast=\bC\backslash\br{0}$ 
with $|p|<1$, $|q|<1$, we use the multiplicative notations 
\begin{equation}
\theta(z;p)=(z;p)_\infty(p/z;p)_\infty,\qquad
(z;p)_\infty=\prod_{i=0}^{\infty}(1-p^iz), 
\end{equation}
for the {\em Jacobi theta function}, 
and 
\begin{equation}
\Gamma(z;p,q)=\frac{(pq/z;p,q)_\infty}{(z;p,q)_\infty},
\qquad
(z;p,q)_\infty=\prod_{i,j=0}^{\infty}(1-p^iq^jz)
\end{equation}
for the {\em Ruijsenaars elliptic gamma function}. 
These functions satisfy the functional equations
\begin{equation}
\begin{split}
&
\theta(pz;p)=-z^{-1}\theta(z;p),\quad \theta(p/z;p)=\theta(z;p), 
\\
&
\Gamma(qz;p,q)=\theta(z;p)\Gamma(z;p,q),
\qquad
\Gamma(pq/u;p,q)=\Gamma(u;p,q)^{-1}. 
\end{split}
\end{equation}
The multiplicative theta function $\theta(z;p)$ satisfies
the three-term relation
\begin{equation}\label{eq:three-term-theta}
c\,\theta(bc^{\pm1};p)\theta(az^{\pm1};p)+
a\,\theta(ca^{\pm1};p)\theta(bz^{\pm1};p)+
b\,\theta(ab^{\pm1};p)\theta(cz^{\pm1};p)=0, 
\end{equation}
corresponding to \eqref{eq:three-term}. 
Here we have used the abbreviation 
$\theta(ab^{\pm1};p)=\theta(ab;p)\theta(ab^{-1};p)$ 
to refer to the product of two factors with different signs. 
Note also that
\begin{equation}
\frac{1}{\Gamma(z^{\pm1};p,q)}=(1-z^{\pm1})(pz^{\pm1};p)_\infty
(qz^{\pm1};q)_\infty=-z^{-1}\theta(z;p)\theta(z;q). 
\end{equation}

Following Spiridonov \cite{Spiridonov2003,Spiridonov2005}, 
we consider the elliptic hypergeometric integral 
$I(u;p,q)$ in eight variables $u=(u_0,u_1,\ldots,u_7)$ 
defined by
\begin{equation}\label{eq:ehi8}
I(u;p,q)=
\frac{(p;p)_\infty(q;q)_\infty}{4\pi\sqrt{-1}}
\int_{C}
\frac{\prod_{k=0}^{7}\Gamma(u_kz^{\pm1};p,q)}{\Gamma(z^{\pm2};p,q)}\dfrac{dz}{z}. 
\end{equation}
Here we assume that $u=(u_0,u_1,\ldots,u_7)$ is generic 
in the sense that 
$u_ku_l\notin p^{-\bN}q^{-\bN}$ for any 
$k,l\in\br{0,1,\ldots,7}$ ($\bN=\br{0,1,2,\ldots}$).  
This condition is equivalent to saying that the two sets 
\begin{equation}
\begin{split}
S_0&=\brm{p^iq^ju_k}{i,j\in\bN,\ k\in\br{0,1,\ldots,7}},
\\
S_\infty&=\brm{p^{-i}q^{-j}u_k^{-1}}
{i,j\in\bN,\ k\in\br{0,1,\ldots,7}}
\end{split}
\end{equation}
of possible poles of the integrand are disjoint.  
For the contour $C$ we take a homology cycle in 
$\bC^\ast\backslash(S_0\cup S_\infty)$ 
such that $n(C,a)=1$ for all $a\in S_0$ and 
$n(C,a)=0$ for all $a\in S_\infty$, 
where $n(C;a)$ stands for the winding number of 
$C$ around $z=a$.  
Note also that, if $|u_k|<1$ $(k=0,1,\ldots,7)$, one can 
take the unit circle $|z|=1$ oriented positively as the cycle 
$C$. 

The following transformation formulas are 
due to Spiridonov \cite{Spiridonov2003} and Rains \cite{Rains2010}. 
\begin{thm}\label{thm:5A}
Suppose that the parameters $u=(u_0,u_1,\ldots,u_7)$ 
satisfy the
balancing condition $u_0u_1\cdots u_7=p^2q^2$.
Then the following transformation formulas hold\,$:$
\begin{equation}
\label{eq:Bailey1}
\begin{split}
(1)\quad&
I(u;p,q)=I(\widetilde{u};p,q)
\prod_{0\le i<j\le3}\Gamma(u_iu_j;p,q)
\prod_{4\le i<j\le7}\Gamma(u_iu_j;p,q), 
\\
&\widetilde{u}=
(\widetilde{u}_0,\widetilde{u}_1,\ldots,\widetilde{u}_7),
\quad
\widetilde{u}_i=
\begin{cases}
u_i\sqrt{pq/u_0u_1u_2u_3}\quad&(i=0,1,2,3),\\
u_i\sqrt{pq/u_4u_5u_6u_7}\quad&(i=4,5,6,7),
\end{cases}
\end{split}
\end{equation}
and
\begin{equation}\label{eq:Bailey2}
\begin{split}
(2)\quad&I(u;p,q)=I(\widehat{u};p,q)
\prod_{0\le i<j\le7}\Gamma(u_iu_j;p,q),
\\
&\widehat{u}=(\widehat{u}_0,\ldots,\widehat{u}_7),\quad
\widehat{u}_i=\sqrt{pq}/u_i
\quad(i=0,1,\ldots,7).
\phantom{\quad(i=0,1,\ldots)}
\end{split}
\end{equation}
\end{thm}
Note that, if the parameters 
$u=(u_0,u_1,\ldots,u_7)$ satisfy 
$u_0u_1\cdots u_7=p^2q^2$ and 
$|u_k|=|pq|^{\frac{1}{4}}$ ($k=0,1,\ldots,7$), 
then both $\widetilde{u}$ and $\widehat{u}$ satisfy 
the two conditions as well. 

\par\medskip
Taking another base $r\in\bC^\ast$ with $|r|<1$, we set
\begin{equation}\label{eq:Psi}
\Psi(u;p,q,r)=
I(u;p,q)\,
\prod_{0\le i<j\le 7}\Gamma(u_iu_j;p,q,r), 
\end{equation}
where 
\begin{equation}
\begin{split}
\Gamma(z;p,q,r)&=(z;p,q,r)_\infty(pqr/z;p,q,r)_\infty,
\\
(z;p,q,r)_\infty&=\prod_{i,j,k=0}^{\infty}(1-p^iq^jr^k z).  
\end{split}
\end{equation}
Note that
\begin{equation}
\Gamma(rz;p,q,r)=\Gamma(z;p,q)\Gamma(z;p,q,r),
\quad
\Gamma(pqr/z;p,q,r)=\Gamma(z;p,q,r).  
\end{equation}
\begin{prop}\label{prop:5A}
Under the condition $u_0u_1\cdots u_7=p^2q^2$, 
the function $\Psi(u;p,q,r)$ defined by \eqref{eq:Psi}
is invariant with respect to the 
transformations $u\to\widetilde{u}$ and $u\to\widehat{u}$.  
\end{prop}
\begin{proof}{Proof}
When $\br{i,j,k,l}=\br{0,1,2,3}$ or $\br{4,5,6,7}$, one has
\begin{equation}
\begin{split}
\Gamma(\widetilde{u}_i\widetilde{u}_j;p,q,r)
&=
\Gamma(pq/u_ku_l;p,q,r)
=
\Gamma(ru_ku_l;p,q,r)
\\
&=\Gamma(u_k,u_l;p,q)\Gamma(u_ku_l;p,q,r).  
\end{split}
\end{equation}
Also, for distinct $i,j\in\br{0,1,\ldots,7}$, 
\begin{equation}
\begin{split}
\Gamma(\widehat{u}_i\widehat{u}_j;p,q,r)
&=\Gamma(pq/u_iu_j;p,q,r)
=\Gamma(ru_iu_j;p,q,r)
\\
&=\Gamma(ru_iu_j;p,q)\Gamma(u_iu_j;p,q,r).
\end{split}
\end{equation}
Using these formulas, 
Theorem \ref{thm:5A} can be reformulated as 
\begin{equation}
\Psi(u;p,q,r)=\Psi(\widetilde{u};p,q,r)=\Psi(\widehat{u};p,q,r)
\end{equation}
under the condition $u_0u_1\cdots u_7=p^2q^2$. 
\end{proof}

\par\medskip
Returning to the elliptic hypergeometric integral 
\eqref{eq:ehi8}, we notice that the integrand
\begin{equation}
H(z,u;p,q)=\frac{\prod_{k=0}^{7}\Gamma(u_kz^{\pm1};p,q)}{\Gamma(z^{\pm2};p,q)}
\end{equation}
satisfies 
\begin{equation}
T_{q,u_k}H(z,u;p,q)=\theta(u_kz^{\pm1};p)H(z,u;p,q)
\end{equation}
with respect to the $q$-shift operator $T_{q,u_k}$ in $u_k$
$(k\in\br{0,1,\ldots,7})$:
\begin{equation}
T_{q,u_k}f(u_0,u_1,\ldots,u_7)=f(u_0,\ldots,qu_k,\ldots,u_7). 
\end{equation}
Hence, by the functional equation 
\eqref{eq:three-term-theta} we have
\begin{equation}
\big(u_k\theta(u_ju_k^{\pm1};p)T_{q,u_i}+
u_i\theta(u_ku_i^{\pm1};p)T_{q,u_j}+
u_j\theta(u_iu_j^{\pm1}p)T_{q,u_k}\big)H(z,u;p,q)=0
\end{equation}
for any triple $i,j,k\in\br{0,1,\ldots,7}$. 
Passing to the integral, we obtain the 
following contiguity relations 
for the elliptic hypergeometric integral.  
\begin{prop}\label{prop:5B}
The elliptic hypergeometric integral \eqref{eq:ehi8} satisfies 
the three-term relation
\begin{equation}\label{eq:contiguityI}
\big(u_k\theta(u_ju_k^{\pm1};p)T_{q,u_i}+
u_i\theta(u_ku_i^{\pm1};p)T_{q,u_j}+
u_j\theta(u_iu_j^{\pm1};p)T_{q,u_k}\big)I(u;p,q)=0
\end{equation}
for any triple $i,j,k\in\br{0,1,\ldots,7}$. 
\end{prop}

\par\medskip
It is known that the elliptic hypergeometric integral 
$I(u;p,q)$ of 
\eqref{eq:ehi8} gives rise to terminating elliptic 
hypergeometric series 
in the special cases 
where $pq/u_ku_l=p^{-M}q^{-N}$ 
for distinct $k,l\in\br{0,1,\ldots,7}$ and 
$M,N\in\bN$ (see for example Komori \cite{Komori2005}). 
Here we give a remark on the case where 
$pq/u_0u_7=q^{-N}$ for simplicity. 
We use the notation of very well-poised elliptic hypergeometric 
series
\begin{equation}
\begin{split}
&{}_{12}V_{11}(a_0;a_1,\ldots,a_7;q,p)
=\sum_{k=0}^{\infty}
\frac{\theta(q^{2k}a_0;p)}{\theta(a_0;p)}
\left(\prod_{i=0}^{7}
\frac{\theta(a_i;p;q)_{k}}{\theta(qa_0/a_i;p;q)_{k}}
\right)q^k, 
\\
&\theta(z;p;q)_{k}=
\frac{\Gamma(q^kz;p,q)}{\Gamma(z;p,q)}=
\theta(z;p)\theta(qz;p)\cdots \theta(q^{k-1}z;p)
\quad(k=0,1,2,\ldots),
\end{split}
\end{equation}
assuming that $a_i\in p^{\bZ}q^{-N}$ 
for some $i\in\br{0,1,\ldots,7}$ and $N=0,1,2,\ldots$. 
\begin{prop}
Under the balancing condition $u_0u_1\cdots u_7=q^2$,
we assume either $q/u_0u_i=q^{-N}$ for some $i\in\br{1,\ldots,6}$ or $q/u_0u_7=pq^{-N}$, where $N=0,1,2,\ldots$. 
Then we have 
\begin{equation}\label{eq:ItoV}
\begin{split}
&I(pu_0,u_1,\ldots,u_{6},pu_7;p,q)
\\
&=\prod_{1\le k<l\le 6}
\Gamma(u_ku_l;p,q)
\frac{\Gamma(q^2/u_0^2;p,q)\Gamma(u_0/u_7;p,q)
}{\prod_{k=1}^{6}\Gamma(qu_i/u_0;p,q)\Gamma(q/u_iu_7;p,q)}
\\
&\quad\cdot
{}_{12}V_{11}(q/u_0^2;q/u_0u_1,\ldots,q/u_0u_6,q/u_0u_7;q,p). 
\end{split}
\end{equation}
\end{prop}
\begin{proof}{Sketch of proof}
Under the balancing condition $u_0u_1\cdots u_7=p^2q^2$, 
we set $t=(t_0,t_1,\ldots,t_7)$, 
$t_i=\sqrt{pq}/u_i$ $(i=0,1,\ldots,7)$, so that 
\begin{equation}
I(u;p,q)=I(t;p,q)\prod_{0\le k<l\le 7}\Gamma(u_ku_l;p,q)
=\frac{I(t;p,q)}{\prod_{0\le k<l\le 7}\Gamma(t_kt_l;p,q)}. 
\end{equation}
We investigate the behavior of the both sides in 
the limit as $pq/u_0u_7\to q^{-N}$, 
namely, $t_0t_7\to q^{-N}$.  In this limit 
$u_7\to pq^{1+N}/u_0$, $t_7\to q^{-N}/t_0$, 
the integral $I(u;p,q)$ on the left-hand side 
has a finite limit, 
while $I(t;p,q)$ gives rise to singularities due 
to pinching of the contour at
\begin{equation}
q^k t_7\to q^{-N+k}/t_0,\quad
q^{-k}t_7^{-1}\to q^{N-k}t_0\quad(k=0,1,\ldots,N). 
\end{equation}
By the residue calculus around these points, we can compute 
the limit 
\begin{equation}
\begin{split}
&\lim_{t_7\to q^{-N}t_0}\frac{I(t;p,q)}
{\prod_{0\le k<l\le 7}\Gamma(t_kt_l;p,q)}
\\
&=\frac{1}{\prod_{0\le i<j\le 6}\Gamma(t_it_j;p,q)}
\prod_{\nu=0}^{N-1}
\frac{\prod_{i=1}^{6}\theta(q^{-N+\nu}t_i/t_0;p)}
{\theta(q^{-N+\nu}/t_0^2;p)}
\\
&\quad\cdot
{}_{12}V_{11}(pt_0^2;t_0t_1,\ldots,t_0t_6,pq^{-N};q,p).  
\end{split}
\end{equation}
Hence, under the conditions 
$u_0u_1\cdots u_7=p^2q^2$ 
and $pq/u_0u_7=q^{-N},$ we obtain 
\begin{equation}\label{eq:sameformula}
\begin{split}
I(u;p,q)&=\Gamma(p^2q^2/u_0^2;p,q)
\prod_{i=1}^{7}\Gamma(u_0/u_i;p,q)\prod_{1\le i<j\le 7}\Gamma(u_iu_j;p,q)
\\
&\quad\cdot
{}_{12}V_{11}(p^2q/u_0^2;pq/u_0u_1,\ldots,pq/u_0u_6,
p^2q/u_0u_7;q,p). 
\end{split}
\end{equation}
Noting that ${}_{12}V_{11}(a_0;a_1,\ldots,a_7;q,p)$ 
is invariant under the $p$-shift 
operator $T_{p,a_i}T_{p,a_j}^{-1}$
for distinct $i,j\in\br{1,\ldots,7}$, we see that the 
same formula \eqref{eq:sameformula} holds 
if we replace the condition ``$pq/u_0u_7=q^{-N}$''
by ``$pq/u_0u_i=q^{-N}$
for some $i\in\br{1,\ldots,6}$''. 
Then, 
replacing $u_0$, $u_7$ by $pu_0$, $pu_7$ respectively, 
we obtain \eqref{eq:ItoV}. 
\end{proof}
\section{\boldmath $W(E_7)$-invariant hypergeometric $\tau$-function}
\label{sec:6}
In the following, we present an explicit hypergeometric 
$\tau$-function for the ORG system of type $E_8$ 
in terms of elliptic hypergeometric integrals. 

\par\medskip
We denote by $x=(x_0,x_1,\ldots,x_7)$ the canonical 
coordinates of $V=\bC^8$ so that
\begin{equation}
x=(x_0,x_1,\ldots,x_7)=x_0v_0+\cdots+x_7v_7;
\quad x_i=\ipr{v_i}{x}\quad(i=0,1,\ldots,7). 
\end{equation}
Note that the highest root $\phi=\hf(v_0+v_1+\cdots+v_7)$ 
of $\Delta(E_8)$ corresponds to 
the linear function
\begin{equation}
\ipr{\phi}{x}=\hf(x_0+x_1+\cdots+x_7). 
\end{equation}
We relate the additive coordinates $x=(x_0,x_1,\ldots,x_7)$ 
and the multiplicative coordinates $u=(u_0,u_1,\ldots,u_7)$ 
through $u_i=e(x_i)=e^{2\pi\sqrt{-1}x_i}$  ($i=0,1,\ldots,7$). 
We also use the notation of exponential functions 
\begin{equation}
u^{\lambda}=e(\ipr{\lambda}{x})\quad(\lambda\in P),
\end{equation}
so that 
$u^{v_i}=u_i$ ($i=0,1,\ldots,7$)
and $u^\phi=(u_0u_1\cdots u_7)^{\shf}$. 

We now consider the case where the fundamental function 
$[\zeta]$ $(\zeta\in\bC)$ is quasi-periodic with respect to the 
$\Omega=\bZ1\oplus\bZ\varpi$, $\Im(\varpi)>0$, 
and is expressed as 
\begin{equation}
[\zeta]=z^{-\hf}\theta(z;p),\quad z=e(\zeta)
\end{equation}
with base $p=e(\varpi)$, $|p|<1$.  
This function has the quasi-periodicity 
\begin{equation}
[\zeta+1]=-[\zeta], \quad 
[\zeta+\varpi]=-e(-\zeta-\tfrac{\varpi}{2})[\zeta]
\end{equation}
and hence $\eta_{1}=0$, $\eta_\varpi=-1$.  
Note also that 
\begin{equation}
[\alpha\pm\beta]=a^{-1}\theta(ab^{\pm1};p)
\quad(a=e(\alpha), b=e(\beta)),
\end{equation}
and that 
the three-term relation \eqref{eq:three-term-theta} for 
$\theta(z;p)$ corresponds to \eqref{eq:three-term} for 
$[\zeta]$. 
As to the constant $\delta\in\bC$, we 
assume $\Im(\delta)>0$, and set $q=e(\delta)$ so that 
$|q|<1$. 

As in Section \ref{sec:1}, we take the simple roots
\begin{equation}
\alpha_0=\phi-v_0-v_1-v_2-v_3,\quad \alpha_j=v_j-v_{j+1}
\quad(j=1,\ldots,6)
\end{equation}
for the root system $\Delta(E_7)$.  Since $v_1-v_0$ is 
the highest root, we see that the Weyl group $W(E_7)$ 
is generated by 
$\mathfrak{S}_8=\la r_{v_j-v_{j+1}} (j=0,1,\ldots,6)\ra$
and the reflection $s_0=r_{\alpha_0}$ by 
$\alpha_0=r_{\phi-v_0-v_1-v_2-v_3}$.  
The symmetric group $\mathfrak{S}_8$ 
acts on the coordinates $x=(x_0,x_1,\ldots,x_7)$ 
and $u=(u_0,u_1,\ldots,u_7)$  
through the permutation of indices, 
while $s_0$ acts on the additive coordinates as 
\begin{equation}
s_0(x_i)=
\begin{cases}
x_i+\hf(\ipr{\phi}{x}-x_0-x_1-x_2-x_3)\quad
& (i=0,1,2,3),\\
x_i+\hf(\ipr{\phi}{x}-x_4-x_5-x_6-x_7) \quad
&(i=4,5,6,7), 
\end{cases}
\end{equation}
and on the multiplicative 
coordinates $u_0,u_1,\ldots,u_7$ as 
\begin{equation}
s_0(u_i)=
\begin{cases}
u_i(u^{\phi}/u_0u_1u_2u_3)^{\shf}\quad
& (i=0,1,2,3),\\
u_i(u^{\phi}/u_4u_5u_6u_7)^{\shf}\quad
&(i=4,5,6,7).  
\end{cases}
\end{equation}
We now restrict the coordinates $x_i$
and $u_i$  
to the level set
\begin{equation}
H_{\kappa}=\brm{x\in V}{\ipr{\phi}{x}=\hf(x_0+x_1+\cdots+x_7)=\kappa}
\quad(\kappa\in\bC)
\end{equation}
so that $u^\phi=(u_0u_1\cdots u_7)^{\frac{1}{2}}=e(\kappa)$.  
Then the action of $s_0$ is given by
\begin{equation}
s_0(x_i)=
\begin{cases}
x_i+\hf(\kappa-x_0-x_1-x_2-x_3)\quad
& (i=0,1,2,3),\\
x_i+\hf(\kappa-x_4-x_5-x_6-x_7)\quad
&(i=4,5,6,7). 
\end{cases}
\end{equation}
and by
\begin{equation}
s_0(u_i)=
\begin{cases}
u_i\sqrt{e(\kappa)/u_0u_1u_2u_3}\quad
& (i=0,1,2,3),\\
u_i\sqrt{e(\kappa)/u_4u_5u_6u_7}\quad
&(i=4,5,6,7). 
\end{cases}
\end{equation}
respectively.  
Suppose $\kappa=\varpi+\delta$ so that 
$e(\kappa)=e(\varpi+\delta)=pq$.
In this case, we have $u_0u_1\cdots u_7=p^2q^2$ 
on $H_{\varpi+\delta}$, and 
the action of 
$s_0$ coincides with the 
transformation $u_i\to\widetilde{u}_i$ in \eqref{eq:Bailey1}. 
Proposition \ref{prop:5A} thus implies that the function
\begin{equation}
\Psi(u;p,q,r)=
I(u;p,q)\,
\prod_{0\le k<l\le 7}\Gamma(u_ku_l;p,q,r)
\end{equation}
regarded as a function on $H_{\varpi+\delta}$, 
$u^\phi=pq$, 
is invariant under the action of $s_0$.  
Since $\Psi(u;p,q,r)$ is manifestly symmetric with respect to 
$u=(u_0,u_1,\ldots,u_7)$, we see that $\Psi(u;p,q,r)$ is a 
$W(E_7)$-invariant meromorphic function on 
$H_{\varpi+\delta}$.  
We remark that the transformation $u_i\to \widehat{u}_i$ 
in \eqref{eq:Bailey2} coincides with the action of 
\begin{equation}\label{eq:ww}
w=r_{07}r_{12}r_{34}r_{56}r_{0127}r_{0347}r_{0567}\in W(E_7), 
\end{equation}
where $r_{ij}=r_{v_i-v_j}$ and $r_{ijkl}=r_{\phi-v_i-v_j-v_k-v_l}$. 
This means that 
the transformation formula 
(2) of Theorem \ref{thm:5A} follows 
from (1).

\par\medskip
We rewrite the contiguity relations of \eqref{eq:contiguityI} 
as 
\begin{equation}
\begin{split}
&u_j^{-1}\theta(u_ju_k^{\pm1};p)u_i^{-1}T_{q,u_i}I(u;p,q)+
u_k^{-1}\theta(u_ku_i^{\pm1};p)u_j^{-1}T_{q,u_j}I(u;p,q)
\\
&
\qquad\qquad\quad\mbox{}+
u_i^{-1}\theta(u_iu_j^{\pm1};p)u_k^{-1}T_{q,u_k}I(u;p,q)=0. 
\end{split}
\end{equation}
Since $u_i^{-1}\theta(u_iu_j^{\pm1};p)=[x_i\pm x_j]$, 
this means that
\begin{equation}\label{eq:contiguityIx}
\begin{split}
&[x_j\pm x_k]u_i^{-1}T_{q,u_i}I(u;p,q)+
[x_k\pm x_i]u_j^{-1}T_{q,u_j}I(u;p,q)
\\
&
\qquad\qquad\quad\mbox{}+
[x_i\pm x_j]u_k^{-1}T_{q,u_k}I(u;p,q)=0.
\end{split}
\end{equation}
In view of this formula, we set
\begin{equation}
J(x)=e(-Q(x))I(u;p,q),\quad Q(x)=\tfrac{1}{2\delta}\ipr{x}{x}. 
\end{equation}
Note that 
\begin{equation}
Q(x+a\delta)=Q(x)+\ipr{a}{x}+\hf\ipr{a}{a}\delta
\end{equation}
for any $a\in P$.  
Since 
\begin{equation}
Q(x+v_i\delta)
=Q(x)+\ipr{v_i}{x}+\hf\delta
=Q(x)+x_i+\hf\delta
\end{equation}
we have
\begin{equation}
J(x+v_i\delta)=e(-Q(x))\,q^{-\hf} u_i^{-1}T_{q,u_i}I(u;p,q)
\quad(i\in\br{0,1,\ldots,7}). 
\end{equation}
Hence by \eqref{eq:contiguityIx} 
we obtain the three-term relations
\begin{equation}\label{eq:three-termJ}
[x_j\pm x_k]J(x+v_i\delta)+
[x_k\pm x_i]J(x+v_j\delta)+
[x_i\pm x_j]J(x+v_k\delta)=0, 
\end{equation}
namely
\begin{equation}\label{eq:contiguityJ}
[\ipr{v_j\pm v_k}{x}]J(x+v_i\delta)+
[\ipr{v_k\pm v_i}{x}]J(x+v_j\delta)+
[\ipr{v_i\pm v_j}{x}]J(x+v_k\delta)=0
\end{equation}
for any triple $i,j,k\in\br{0,1,\ldots,7}$. 

\par\medskip
On the basis of these observations, 
we construct a hypergeometric 
$\tau$-function on 
\begin{equation}
D_\varpi=\bigsqcup_{n\in\bZ} H_{\varpi+n\delta}\subset V
\end{equation}
with the initial level $c=\varpi$.
For this purpose we introduce the holomorphic function 
\begin{equation}
\cF(x)=\prod_{0\le i<j\le 7}\Gamma(u_iu_j;p,q,q)\qquad
(x\in V). 
\end{equation}
%
\begin{thm}\label{thm:6A}
There exists a unique hypergeometric $\tau$-function 
$\tau(x)$ on $D_\varpi$ such that 
$\tau^{(n)}(x)=0$ $(n<0)$ and 
\begin{equation}
\begin{split}
\tau^{(0)}(x)&=\cF(x+\phi\delta)=
\prod_{0\le i<j\le 7}\Gamma(qu_iu_j;p,q,q)
\quad (x\in H_{\varpi}), 
\\
\tau^{(1)}(x)&=\cF(x)J(x)
\\
&
=
e(-Q(x))I(u;p,q)
\prod_{0\le i<j\le 7}\Gamma(u_iu_j;p,q,q)
\quad(x\in H_{\varpi+\delta}). 
\end{split}
\end{equation}
Furthermore $\tau(x)$ is a $W(E_7)$-invariant meromorphic 
function on $D_\varpi$. 
\end{thm}
\begin{proof}{Proof}
We need to show that $\tau^{(0)}(x)$ and $\tau^{(1)}(x)$ 
satisfy the two conditions of Theorem \ref{thm:4A}. 
We first show that $\tau^{(0)}(x)$ ($x\in H_{\varpi}$) 
is $W(E_7)$-invariant.  Since the 
$\mathfrak{S}_8$-invariance is manifest, we have only 
to show that it is invariant under the action of $s_0$.  
Noting that 
\begin{equation}
s_0(u_iu_j)=
\begin{cases}
p/u_ku_l & (\br{i,j,k,l}=\br{0,1,2,3}),\\
u_iu_j & (i\in\br{0,1,2,3}, j\in\br{4,5,6,7}),\\
p/u_ku_l & (\br{i,j,k,l}=\br{4,5,6,7})
\end{cases}
\end{equation}
for $x\in H_{\varpi}$, 
we have
\begin{equation}
s_0\big(\Gamma(qu_iu_j;p,q,q)\big)
=\Gamma(pq/u_ku_l;p,q,q)=\Gamma(qu_ku_l;p,q,q)
\end{equation}
for $\br{i,j,k,l}=\br{0,1,2,3}$ or $\br{4,5,6,7}$.  
Since $\Gamma(u_iu_j;p,q,q)$ is $s_0$-invariant 
for $i\in\br{0,1,2,3}$, $j\in\br{4,5,6,7}$, 
we have 
$s_0(\tau^{(0)}(x))=
\tau^{(0)}(x)$. 
We now verify that our $\tau^{(0)}(x)$ satisfy the 
condition \eqref{eq:4AII1} for any $C_3$-frame
$\br{\pm a_0,\pm a_1,\pm a_2}$ of type $\II_1$
as in \eqref{eq:C3frameII}.  
Since the $C_3$-frames of type $\II_1$ form a single 
$W(E_7)$-orbit, by the $W(E_7)$-invariance of $\tau^{(0)}(x)$ 
we may take
\begin{equation}
\begin{split}
&a_0=\hf(v_0+v_1+v_2+v_3),\\
&a_1=\hf(v_0+v_1-v_2-v_3),\quad
a_2=\hf(v_0-v_1+v_2-v_3)
\end{split}
\end{equation}
so that 
\begin{equation}
\begin{split}
\br{a_0\pm a_1}=\br{v_0+v_1,v_2+v_3},
\quad
\br{a_0\pm a_2}=\br{v_0+v_2,v_1+v_3}.
\end{split}
\end{equation}
In this case one can directly check 
\begin{equation}
\frac{\tau^{(0)}(x\pm a_1\delta)}
{\tau^{(0)}(x\pm a_2\delta)}
=
\frac
{\theta(u_0u_1;p)\theta(u_2u_3;p)}
{\theta(u_0u_2;p)\theta(u_1u_3;p)}
=
\frac
{[x_0+x_1][x_2+x_3]}
{[x_0+x_2][x_1+x_3]}=\frac{[\ipr{a_0\pm a_1}{x}]}{[\ipr{a_0\pm a_2}{x}]}. 
\end{equation}

We next verify that $\tau^{(0)}(x)$ and $\tau^{(1)}(x)$ satisfy 
\eqref{eq:4AI} for any $C_3$-frame $\br{\pm a_0,\pm a_1,\pm a_2}$ 
with $\ipr{\phi}{a_i}=\hf$ $(i=0,1,2)$. 
Since $\tau^{(1)}(x)=e(-Q(x))\Psi(u;p,q,q)$ is $W(E_7)$-invariant, 
we have only to check \eqref{eq:4AI} for 
$\br{\pm a_0,\pm a_1,\pm a_2}=\br{\pm v_0,\pm v_1,\pm v_2}$, namely, 
\begin{equation}
\begin{split}
&[x_1\pm x_2]\tau^{(0)}(x-v_0\delta)\tau^{(1)}(x+v_0\delta)
+[x_2\pm x_0]\tau^{(0)}(x-v_1\delta)\tau^{(1)}(x+v_1\delta)
\\
&\qquad\qquad\qquad\mbox{}
+[x_0\pm x_1]\tau^{(0)}(x-v_2\delta)\tau^{(1)}(x+v_2\delta)
=0. 
\end{split}
\end{equation}
Since 
\begin{equation}
\cF(x+\phi\delta-v_k\delta)
\cF(x+v_k\delta)=
\cF(x+\phi\delta)
\cF(x), 
\end{equation}
we have 
\begin{equation}
\begin{split}
\tau^{(0)}(x-v_k\delta)\tau^{(1)}(x+v_k\delta)
&=\cF(x+\phi\delta-v_k\delta)
\cF(x+v_k\delta)J(x+v_k\delta)
\\
&=\cF(x+\phi\delta)
\cF(x)
J(x+v_k\delta)
\end{split}
\end{equation}
for each $k=0,1,\ldots,7$.  Hence the three-term relations 
\eqref{eq:three-termJ} for $J(x)$ imply
\begin{equation}
\begin{split}
&[x_j\pm x_k]\tau^{(0)}(x-v_i\delta)\tau^{(1)}(x+v_i\delta)
+[x_k\pm x_i]\tau^{(0)}(x-v_j\delta)\tau^{(1)}(x+v_j\delta)
\\
&\qquad\qquad\qquad\mbox{}
+[x_i\pm x_j]\tau^{(0)}(x-v_k\delta)\tau^{(1)}(x+v_j\delta)
=0 
\end{split}
\end{equation}
for any tripe $i,j,k\in\br{0,1,\ldots,7}$. 
The $W(E_7)$-invariance of $\tau(x)$ on $D_\varpi$ follows from 
the uniqueness of $\tau(x)$ and the $W(E_7)$-invariance of $\tau^{(0)}(x)$ 
and $\tau^{(1)}(x)$.  
\end{proof}

\par\medskip
We next investigate the determinant formula of 
Theorem \ref{thm:4B} for the hypergeometric $\tau$-function 
of Theorem \ref{thm:6A} with initial condition 
\begin{equation}
\begin{split}
\tau^{(0)}(x)&=
\prod_{0\le i<j\le 7}\Gamma(qu_iu_j;p,q,q)
\quad(x\in H_{\varpi}), 
\\
\tau^{(1)}(x)&=
e(-Q(x))\,I(u;p,q)\,
\prod_{0\le i<j\le 7}\Gamma(u_iu_j;p,q,q)
\quad
(x\in H_{\varpi+\delta}). 
\end{split}
\end{equation}
For the recursive construction of $\tau^{(n)}(x)$ $(x\in H_{\varpi+n\delta}$) for $n=2,3,\ldots$, 
we use the 
$C_8$-frame $A_1=\br{\pm a_0,\pm a_1,\ldots,\pm a_7}$ 
of type $\II$ of Example \ref{exm:C8frames}, where 
\begin{equation}\label{eq:aiA1}
\begin{array}{lll}
&a_0=\hf(v_0+v_1+v_2+v_3),\quad 
a_4=\hf(v_4-v_5-v_6+v_7),\\[2pt]
&a_1=\hf(v_0+v_1-v_2-v_3),\quad
a_5=\hf(-v_4+v_5-v_6+v_7).\\[2pt]
&a_2=\hf(v_0-v_1+v_2-v_3),\quad 
a_6=\hf(-v_4-v_5+v_6+v_7),\\[2pt]
&a_3=\hf(v_0-v_1-v_2+v_3),\quad
a_7=\hf(v_4+v_5+v_6+v_7).
\end{array}
\end{equation}
Note here that 
$\ipr{\phi}{a_0}=\ipr{\phi}{a_7}=1$, $\ipr{\phi}{a_i}=0$ $(i=1,\ldots,6)$ 
and $a_0+a_7=\phi$.   
This $C_8$-frame $A_1$ 
contains the following 30 
$C_3$-frame of type $\II_1$:
\begin{equation}
\br{\pm a_0,\pm a_i,\pm a_j},\quad 
\br{\pm a_7,\pm a_i,\pm a_j}\quad
(1\le i<j\le 6). 
\end{equation}
Since 
\begin{equation}
\alpha_0=\phi-v_0-v_1-v_2-v_3=a_7-a_0, 
\end{equation}
we have 
\begin{equation}
s_0(a_0)=a_7,\quad s_0(a_7)=a_0,\quad
s_0(a_i)=a_i\quad(i=1,\ldots,6).  
\end{equation}
This means that 
the $C_3$-frames 
$\br{\pm a_0,\pm a_i,\pm a_j}$ and
$\br{\pm a_7,\pm a_i,\pm a_j}$
are transformed to each other by $s_0$.  
To fix the idea, we consider below the cases of  
$C_3$-frames 
$\br{\pm a_0,\pm a_1,\pm a_2}$ and 
$\br{\pm a_7,\pm a_1,\pm a_2}$. 

\begin{thm}[Determinant formula]\label{thm:6B}
The $W(E_7)$-invariant 
hypergeometric $\tau$-function 
of Theorem \ref{thm:6A} is expressed in terms 
of 2-directional Casorati determinants as 
\begin{equation}
\tau^{(n)}(x)=g^{(n)}(x) \det\big(\psi^{(n)}_{ij}(x)\big)_{i,j=1}^{n}
\quad(x\in H_{\varpi+n\delta})
\end{equation}
for $n=0,1,2,\ldots$, where the gauge factors 
$g^{(n)}(x)$ and the matrix elements 
$\psi_{ij}^{(n)}(x)$ are given as follows according to the 
$C_3$-frame of type $\II_1$ chosen for the recurrence.  
\newline
$(1)$ Case of the $C_3$-frame $\br{\pm a_0,\pm a_1, \pm a_2$}\,$:$ 
\begin{equation}
\begin{split}
g^{(n)}(x)&=
\frac{p^{\binom{n}{2}}e(-nQ(x))}{d^{(n)}(x)}
\prod_{\substack{0\le i<j\le 3\\[1pt]
\mbox{\scriptsize\rm or}\ 4\le i<j\le 7}}
\Gamma(qu_iu_j;p,q,q)
\prod_{\substack{0\le i\le 3\\[1pt] 4\le j\le 7}}
\Gamma(q^{1-n}u_iu_j;p,q,q), 
\\
d^{(n)}(x)&=q^{2\binom{3}{2}}(pq/u_0u_1)^{\binom{n}{2}}
\prod_{k=1}^{n}
\theta(q^{1-n}u_0u_3;p;q)_{k-1}
\theta(q^{k-n}u_0/u_3;p;q)_{k-1},
\\
&\quad\cdot
\prod_{k=1}^{n}
\theta(q^{1-n}u_1u_2;p;q)_{k-1}
\theta(q^{k-n}u_1/u_2;p;q)_{k-1}, 
\\
\psi_{ij}^{(n)}(x)&=
I(
q^{n-i}t_0,q^{n-j}t_1,q^{j-1}t_2,q^{i-1}t_3,
t_4,t_5,t_6,t_7;p,q),
\\
&\quad
t_k=\begin{cases}
u_k\sqrt{pq/u_0u_1u_2u_3} & (k=0,1,2,3),
\\
u_k\sqrt{pq/u_4u_5u_6u_7} & (k=4,5,6,7). 
\end{cases}
\end{split}
\end{equation}
$(2)$ Case of the $C_3$-frame 
$\br{\pm a_7,\pm a_1,\pm a_2}$\,$:$
\begin{equation}
\begin{split}
g^{(n)}(x)&=\frac{p^{\binom{n}{2}}e(-nQ(x))}
{d^{(n)}(x)}
\prod_{0\le i<j\le 7}
\Gamma(q^{1-n}u_iu_j;p,q,q),
\\
d^{(n)}(x)&=
q^{-\binom{n+1}{3}}(u_2u_3)^{\binom{n}{2}}
\prod_{k=1}^{n}
\theta(q^{1-k}u_0u_3;p;q)_{k-1}
\theta(q^{k-n}u_0/u_3;p;q)_{k-1}
\\
&\quad\cdot
\prod_{k=1}^{n}
\theta(q^{1-k}u_1u_2;p;q)_{k-1}
\theta(q^{k-n}u_1/u_2;p;q)_{k-1},
\\
\psi^{(n)}_{ij}(x)&=
I(q^{n-i}t_0,q^{n-j}t_1,q^{1-j}t_2,q^{1-i}t_3,t_4,t_5,t_6,t_7;p,q),
\quad t_k=q^{\shf(1-n)}u_k.
\end{split}
\end{equation}
\end{thm}

Rewriting
the 2-directional Casorati determinants above, 
we obtain expressions of the 
$W(E_7)$-invariant 
hypergeometric 
$\tau$-function in terms of multiple elliptic 
hypergeometric integrals.  
For the variables $t=(t_0,t_1,\ldots,t_7)$, 
we consider the multiple integrals as in 
Rains \cite{Rains2005,Rains2010}: 
\begin{equation}\label{eq:RainsIn}
\begin{split}
&I_n(t;p,q)=I_n(t_0,t_1,\ldots,t_7;p,q)
\\
&=
\frac{(p;p)_{\infty}^n(q;q)_\infty^{n}}
{2^n n! (2\pi\sqrt{-1})^n}
\int_{C^n}
\prod_{i=1}^{n}
\frac{
\prod_{k=0}^{7}\Gamma(t_kz_i^{\pm1};p,q)
}{\Gamma(z_i^{\pm2};p,q)}
\prod_{1\le i<j\le n}
\theta(z_i^{\pm1}z_j^{\pm1};p)
\frac{dz_1\cdots dz_n}{z_1\cdots z_n}.
\end{split}
\end{equation}
We remark that $I_n(t;p,q)$ is a special case (with $s=q$)
of the $BC_n$ elliptic hypergeometric 
integral
\begin{equation}
\frac{(p;p)_{\infty}^n(q;q)_\infty^{n}}
{2^n n! (2\pi\sqrt{-1})^n}
\int_{C^n}
\prod_{i=1}^{n}
\frac{
\prod_{k=0}^{7}\Gamma(t_kz_i^{\pm1};p,q)
}{\Gamma(z_i^{\pm2};p,q)}
\prod_{1\le i<j\le n}
\frac{\Gamma(sz_i^{\pm1}z_j^{\pm1};p,q)}
{\Gamma(z_i^{\pm1}z_j^{\pm1};p,q)}
\frac{dz_1\cdots dz_n}{z_1\cdots z_n}
\end{equation}
of type II. 
\begin{thm}[Multiple integral representation]
\label{thm:6C}
The $W(E_7)$-invariant hypergeometric $\tau$-function 
of Theorem \ref{thm:6A} is expressed as follows 
in terms of multiple elliptic hypergeometric integrals\,$:$
\begin{equation}\label{eq:intreptau}
\begin{split}
&\tau^{(n)}(x)
=
p^{\binom{n}{2}}e(-nQ(x))\,
I_n(q^{\shf(1-n)}u;p,q)
\prod_{0\le i<j\le 7}
\Gamma(q^{1-n}u_iu_j;p,q,q)
\\
&=
p^{\binom{n}{2}}e(-nQ(x))\,
I_n(\widetilde{u};p,q)
\prod_{\substack{0\le i<j\le 3
\\[1pt]\mbox{\scriptsize\rm or}\,4\le i<j\le 7}}
\Gamma(qu_iu_j;p,q,q)
\prod_{\substack{0\le i\le 3\\[1pt] 4\le j\le 7}}
\Gamma(q^{1-n}u_iu_j;p,q,q)
\\
&(
x\in H_{\varpi+n\delta},\ n=0,1,2,\ldots),
\end{split}
\end{equation}
where 
\begin{equation}
\widetilde{u}_k=
\begin{cases}
u_k\sqrt{pq/u_0u_1u_2u_3} & (k=0,1,2,3),\\
u_k\sqrt{pq/u_4u_5u_6u_7} & (k=4,5,6,7).
\end{cases}
\end{equation}
\end{thm}
Theorems \ref{thm:6B} and \ref{thm:6C}
will be proved in the next section. 

We remark here that the equality of two expressions in 
\eqref{eq:intreptau} implies the transformation 
formula 
\begin{equation}\label{eq:transIn1}
I_n(t;p,q)=I_n(\widetilde{t};p,q)
\prod_{0\le i<j\le 3}\frac{\Gamma(q^nt_it_j;p,q,q)}
{\Gamma(t_it_j;p,q,q)}
\prod_{4\le i<j\le 7}\frac{\Gamma(q^nt_it_j;p,q,q)}
{\Gamma(t_it_j;p,q,q)}
\end{equation}
for the hypergeometric integral $I_n(t;p,q)$ 
on $H_{\varpi+(2-n)\delta}$ ($t_0t_1\cdots t_7=
p^2q^{4-2n}$), where 
\begin{equation}
\widetilde{t}=(\widetilde{t}_0,\widetilde{t}_1,
\ldots,\widetilde{t}_7),\quad
\widetilde{t_k}=s_0(t_k)=
\begin{cases}
t_k\sqrt{pq^{2-n}/t_0t_1t_2t_3} & (k=0,1,2,3),\\
t_k\sqrt{pq^{2-n}/t_4t_5t_6t_7} & (k=4,5,6,7).
\end{cases}
\end{equation}
This is a special case of 
a transformation formula of Rains \cite{Rains2010}.  
Note also, that the meromorphic function 
\begin{equation}\label{eq:Psin}
\Psi_n(t;p,q)=
I_n(t;p,q)
\prod_{0\le i<j\le 7}\Gamma(t_it_j;p,q,q)
\end{equation}
on $H_{\varpi+(2-n)\delta}$ is $W(E_7)$-invariant.  
The invariance of $\Psi_n(t;p,q)$ with respect to 
$w$ of \eqref{eq:ww} gives rise to the transformation formula
\begin{equation}\label{eq:transIn2}
\begin{split}
&I_n(t;p,q)
=I_n(\widehat{t};p,q)
\prod_{0\le i<j\le 7}
\frac{\Gamma(q^nt_it_j;p,q,q)}{\Gamma(t_it_j;p,q,q)}
\\
&\widehat{t}=(\widehat{t}_0,\widehat{t}_1,
\ldots,\widehat{t}_7),\quad
\widehat{t}_k=w(t_k)=\sqrt{pq^{2-n}}/t_k
\quad(k=0,1,\ldots,7)
\end{split}
\end{equation}
under the condition $t_0t_1\cdots t_7=p^2q^{4-2n}$. 

Applying this transformation formula, we obtain another 
expression of $\tau(x)$ of Theorem \ref{thm:6C}:
\begin{equation}
\begin{split}
\tau^{(n)}(x)
&=p^{\binom{n}{2}}e(-nQ(x))\,
I_n(\sqrt{pq}u^{-1};p,q)
\prod_{0\le i<j\le 7}\Gamma(qu_iu_j;p,q,q)
\\
&=p^{\binom{n}{2}}e(-nQ(x))\,
I_n(\sqrt{pq}u^{-1};p,q)
\prod_{0\le i<j\le 7}\Gamma(pq/u_iu_j;p,q,q). 
\end{split}
\end{equation}
In the notation of \eqref{eq:Psin}, the $W(E_7)$-invariant 
hypergeometric $\tau$-function is expressed as 
\begin{equation}
\begin{split}
\tau^{(n)}(x)
&=
p^{\binom{n}{2}}e(-nQ(x))\Psi_{n}(q^{\shf(1-n)}u;p,q)
\\
&=
p^{\binom{n}{2}}e(-nQ(x))\Psi_{n}(p^{\shf}q^{\shf}u^{-1};p,q)
\end{split}
\end{equation}
for $x\in H_{\varpi+n\delta}$  $(n=0,1,2,\ldots)$.

\section{Proof of Theorems \ref{thm:6B}\ and \ref{thm:6C}}
\label{sec:7}
In this section we prove Theorems 
\ref{thm:6B} and \ref{thm:6C} simultaneously.  
\par\medskip
We take the $C_3$-frame 
$\br{\pm a_0,\pm a_1,\pm a_2}$ of type $\II_1$ 
as in \eqref{eq:aiA1} 
for constructing $\tau^{(n)}(x)$ $n=2,3,\ldots$.
In this case, we have 
\begin{equation}
\begin{split}
&\br{a_0\pm a_1}=\br{v_0+v_1,v_2+v_3},
\quad
\br{a_0\pm a_2}=\br{v_0+v_2,v_1+v_3}
\\
&\br{a_1\pm a_2}=\br{v_0-v_3,v_1-v_2}. 
\end{split}
\end{equation}
Hence, 
in order to apply Theorem \ref{thm:4B}, we need 
to decompose $\tau^{(1)}(x)$ into the product
\begin{equation}
\tau^{(1)}(x)=g^{(1)}(x)\psi(x)\quad(x\in H_{\varpi+\delta})
\end{equation}
with a gauge factor satisfying 
\begin{equation}\label{eq:condg1}
\frac{g^{(1)}(x\pm a_1\delta)}{g^{(1)}(x\pm a_2\delta)}
=
\dfrac
{[\ipr{a_0\pm a_1}{x}]}
{[\ipr{a_0\pm a_2}{x}]}
=
\dfrac
{\theta(u_0u_1;p)\theta(u_2u_3;p)}
{\theta(u_0u_2;p)\theta(u_1u_3;p)}
\quad(x\in H_{\varpi+\delta}). 
\end{equation}
Then the gauge factors $g^{(n)}(x)$ $n=2,3,\ldots$ 
are determined by the recurrence formula 
\begin{equation}\label{eq:recgn}
\begin{split}
\frac{
g^{(n-1)}(x-a_0\delta)
g^{(n+1)}(x+a_0\delta)}
{g^{(n)}(x\pm a_1\delta)}
=\frac{[\ipr{a_0\pm a_2}{x}]}
{[\ipr{a_1\pm a_2}{x}]}
=\frac{\theta(u_0u_2;p)\theta(u_1u_3;p)}
{u_2u_3\theta(u_0/u_3;p)\theta(u_1/u_2;p)}
\end{split}
\end{equation}
for $n=1,2,\ldots$ 
starting from $g^{(0)}(x)=\tau^{(0)}(x)$ and 
$g^{(1)}(x)$.  

For this purpose, using the transformation formula 
(1) of Theorem \ref{thm:5A}, 
we rewrite $\tau^{(1)}(x)$ as 
\begin{equation}
\begin{split}
\tau^{(1)}(x)
&=e(-Q(x))\,
I(\widetilde{u};p,q)
\prod_{0\le i<j\le 7}\Gamma(u_iu_j;p,q,q)
\prod_{\substack{
0\le i<j\le 3\\[1pt]
\mbox{\scriptsize\rm or}\,
4\le i<j\le 7}}\Gamma(u_iu_j;p,q)
\\
&=
e(-Q(x))\,
I(\widetilde{u};p,q)
\prod_{\substack{0\le i<j\le 3\\[1pt]
\mbox{\scriptsize\rm or}\,4\le i<j\le 7}}
\Gamma(qu_iu_j;p,q,q)
\prod_{\substack{0\le i\le 3\\[1pt] 4\le j\le 7}}
\Gamma(u_iu_j;p,q,q),
\end{split}
\end{equation}
where 
\begin{equation}\label{eq:deftildeu}
\widetilde{u}=
(\widetilde{u}_0,\widetilde{u}_1,\ldots,\widetilde{u}_7),
\quad
\widetilde{u}_i
=\begin{cases}
u_i\sqrt{pq/u_0u_1u_2u_3}\quad(i=0,1,2,3),\\
u_i\sqrt{pq/u_4u_5u_6u_7}\quad(i=4,5,6,7),  
\end{cases}
\end{equation}
and set
\begin{equation}
\begin{split}
g^{(1)}(x)&=
e(-Q(x))
\prod_{\substack{0\le i<j\le 3\\[1pt]
\mbox{\scriptsize\rm or}\,4\le i<j\le 7}}
\Gamma(qu_iu_j;p,q,q)
\prod_{0\le i\le 3, 4\le j\le 7}
\Gamma(u_iu_j;p,q),
\\
\psi(x)&=I(\widetilde{u};p,q)\quad(x\in H_{\varpi+\delta}). 
\end{split}
\end{equation}
Then one can directly verify 
that $g^{(1)}(x)$ 
satisfies the condition 
\eqref{eq:condg1}. 
For the moment we set $t=(t_0,t_1,\ldots,t_7)$, 
$t_i=\widetilde{u}_i$ $(i=0,1,\ldots,7)$, 
so that $\psi(x)=I(t;p,q)$. 

\par\medskip
We now compute the determinant
\begin{equation}
\begin{split}
&K^{(n)}(x)=
\det\big(\psi^{(n)}_{ij}(x)\big)_{i,j=1}^{n},\quad
\psi^{(n)}_{ij}(x)=\psi(x+v^{(n)}_{ij}\delta), 
\\
&v^{(n)}_{ij}=(1-n)a_0+(n+1-i-j)a_1+(j-i)a_2
\\
&\phantom{v^{(n)}_{ij}}=(1-i,1-j,j-n,i-n,0,0,0,0).
\end{split}
\end{equation}
Noting that 
the multiplicative coordinates of $x+v^{(n)}_{ij}\delta$ 
are given by 
\begin{equation}
(q^{1-i}u_0,q^{1-j}u_1,q^{j-n}u_2,q^{i-n}u_3,u_4,u_5,u_6), 
\end{equation}
we obtain
\begin{equation}
\psi^{(n)}_{ij}(x)=I(q^{n-i}t_0,q^{n-j}t_1,q^{j-1}t_2,q^{i-1}t_3,t_4,t_5,t_6,t_7;p,q). 
\end{equation}
Hence $\psi_{ij}^{(n)}(x)$ is expressed as 
\begin{equation}
\begin{split}
\psi_{ij}^{(n)}(x)&=\kappa
\int_{C}h(z)f_i(z)g_j(z)\dfrac{dz}{z},
\quad 
\kappa=\frac{(p;p)_\infty(q;q)_\infty}
{4\pi\sqrt{-1}}, 
\\
h(z)&=\frac{
\prod_{k=0}^{7}\Gamma(t_kz^{\pm1};p,q)
}{\Gamma(z^{\pm2};p,q)},
\\
f_i(z)&=\theta(t_0z^{\pm1};p;q)_{n-i}
\,\theta(t_3z^{\pm1};p;q)_{i-1},
\\
g_j(z)&=\theta(t_1z^{\pm1};p;q)_{n-j}\,\theta(t_2z^{\pm1};p;q)_{j-1}, 
\end{split}
\end{equation}
for $i,j=1,2,\ldots,n$, 
where 
$\theta(z;p;q)_{k}=\theta(z;p)\theta(qz;p)\cdots\theta(q^{k-1}z;p)$ 
$(k=0,1,2,\ldots)$. 
We rewrite the determinant 
$K^{(n)}(x)=\det\big(\psi_{ij}^{(n)}(x)\big)_{i,j=1}^{n}$
as
\begin{equation}
\begin{split}
K^{(n)}(x)
&=\frac{1}{n!}
\sum_{\sigma,\tau\in\mathfrak{S}_n}
\sgn(\sigma)\sgn(\tau)
\prod_{k=1}^{n}\psi^{(n)}_{\sigma(k),\tau(k)}(x)
\\
&=
\frac{\kappa^n}{n!}
\sum_{\sigma,\tau\in\mathfrak{S}_n}
\sgn(\sigma)\sgn(\tau)
\int_{C^n}
\prod_{k=1}^{n}
h(z_k)f_{\sigma(k)}(z_k)g_{\sigma(k)}(z_k)
\dfrac{dz_1\cdots dz_n}{z_1\cdots z_n}
\\
&=
\frac{\kappa^n}{n!}
\int_{C^n}
h(z_1)\cdots h(z_n)
\det(f_{i}(z_j))_{i,j=1}^n
\det(g_{i}(z_j))_{i,j=1}^n
\dfrac{dz_1\cdots dz_n}{z_1\cdots z_n}.  
\end{split}
\end{equation}
Then 
the determinants 
$\det(f_j(z_i))_{i,j=1}^{n}$, 
$\det(g_j(z_i))_{i,j=1}^{n}$ 
can be evaluated by means of Warnaar's elliptic 
extension of the Krattenthaler determinant 
\cite{Warnaar2002} (see also \cite{N2015}). 
\begin{lem}[Warnaar \cite{Warnaar2002}]
For a set of complex variables $(z_1,\ldots,z_n)$ and 
two parameters $a,b$, one has 
\begin{equation}
\begin{split}
&\det\big(
\theta(az_i^{\pm1};p;q)_{j-1}
\theta(bz_i^{\pm1};p;q)_{n-j}\big)_{i,j=1}^{n}
\\
&=
q^{\binom{n}{3}}a^{\binom{n}{2}}
\prod_{k=1}^{n}\theta(b(q^{k-1}a)^{\pm1};p;q)_{n-k}\,
\prod_{1\le i<j\le n}z_i^{-1}\theta(z_iz_j^{\pm};p). 
\end{split}
\end{equation}
\end{lem}
\begin{proof}{Sketch of proof}
The left-hand side 
is invariant 
under the inversion $z_i\to z_i^{-1}$ for each $i=1,\ldots,n$, 
and alternating with respect to $(z_1,\ldots,z_n)$.  
Hence it is divisible by 
$\prod_{1\le i<j\le n}z_i^{-1}\theta(z_iz_j^{\pm};p)$. 
Also,
the ratio of these two functions 
is elliptic with respect to the additive variable $\zeta_i$ 
with $z_i=e(\zeta_i)$ for each $i=1,\ldots,n$, 
and hence is constant.  
The constant on the right-hand side is determined 
by the substitution $z_i=q^{i-1}a$ $(i=1,\ldots,n)$
which makes the matrix on the left-hand side 
lower triangular. 
\end{proof}

\par\medskip
We thus obtain
\begin{equation}
K^{(n)}(x)=\det\big(\psi^{(n)}_{ij}(x)\big)_{i,j=1}^{n}
=d^{(n)}(x)\,I_n(t;p,q)
\end{equation}
where 
\begin{equation}
I_n(t;p,q)
=
\frac{
(p;p)_\infty^n(q;q)_\infty^n
}{2^n n!(2\pi\sqrt{-1})^n}
\int_{C^n}
\prod_{i=1}^{n}
h(z_i)
\prod_{1\le i<j\le n}
\theta(z_i^{\pm1}z_j^{\pm1};p)
\dfrac{dz_1\cdots dz_n}{z_1\cdots z_n},
\end{equation}
and
\begin{equation}
\begin{split}
&d^{(n)}(x)=
q^{2\binom{n}{3}}
(t_2t_3)^{\binom{n}{2}}
\prod_{k=1}^{n}
\theta(t_0(q^{k-1}t_3)^{\pm1};p;q)_{n-k}
\theta(t_1(q^{k-1}t_2)^{\pm1};p;q)_{n-k}
\\
&=
q^{2\binom{n}{3}}(pq/u_0u_1)^{\binom{n}{2}}
\prod_{(i,j)=(0,3),(1,2)}
\prod_{k=1}^{n}
\theta(q^{1-n}u_iu_j;p;q)_{k-1}
\theta(q^{k-n}u_i/u_k;p;q)_{k-1}. 
\end{split}
\end{equation}

Finally we determine the gauge factors $g^{(n)}(x)$ 
$n=2,3,\ldots$.  
We set
\begin{equation}
\cG^{(n)}(x)=
\prod_{\substack{0\le i<j\le 3\\[1pt]
\mbox{\scriptsize\rm or}\ 4\le i<j\le 7}}
\Gamma(qu_iu_j;p,q,q)
\prod_{\substack{0\le i\le 3\\[1pt] 4\le j\le 7}}
\Gamma(q^{1-n}u_iu_j;p,q,q), 
\end{equation}
so that $g^{(0)}(x)=\cG^{(0)}(x)$, 
$g^{(1)}(x)=e(-Q(x))\cG^{(1)}(x)$. 
By a direct computation, we see that 
 $\cG^{(n)}(x)$ satisfy the recurrence formula
 \begin{equation}
\frac{\cG^{(n-1)}(x-a_0\delta)\cG^{(n+1)}(x+a_0\delta)}
{\cG^{(n)}(x\pm a_1\delta)}
=\theta(u_0u_2;p)
\theta(u_0u_3;p)
\theta(u_1u_2;p)
\theta(u_1u_3;p).  
\end{equation}
If we set $g^{(n)}(x)=\cG^{(n)}(x)c^{(n)}(x)$, the 
recurrence formula to be satisfied by $c^{(n)}(x)$ is
given by 
\begin{equation}
\frac{c^{(n-1)}(x-a_0\delta)
c^{(n+1)}(x+a_0\delta)}
{c^{(n)}(x\pm a_1\delta)}
=\frac{1}{u_2u_3
\theta(u_0u_3^{\pm1};p)
\theta(u_1u_2^{\pm1};p)}
\end{equation}
with the initial conditions $c^{(0)}(x)=1$, 
$c^{(1)}(x)=e(-Q(x))$.  
On the other hand, 
one can verify that 
the functions 
$d^{(n)}(x)$, which appeared in the 
evaluation of the determinant $K^{(n)}(x)$, 
satisfy 
\begin{equation}
\frac{d^{(n-1)}(x-a_0\delta)d^{(n+1)}(x+a_0\delta)}
{d^{(n)}(x\pm a_1\delta)^2}
=
(p/u_0u_1)\theta(u_0u_3^{\pm1};p)
\theta(u_1u_2^{\pm1};p). 
\end{equation}
If we set $c^{(n)}(x)={e^{(n)}(x)}/{d^{(n)}(x)}$, the recurrence formula 
for $e^{(n)}(x)$ is determined as
\begin{equation}
\begin{split}
\frac{e^{(n-1)}(x-a_0\delta)
e^{(n+1)}(x+a_0\delta)}
{e^{(n)}(x\pm a_1\delta)}
=p/u_0u_1u_2u_3=pe(-2\ipr{a_0}{x}). 
\end{split}
\end{equation}
With the initial conditions 
$e^{(0)}(x)=1$, $e^{(1)}(x)=e(-Q(x))$, this recurrence is solved 
as 
\begin{equation}
e^{(n)}(x)=p^{\binom{n}{2}}e(-nQ(x))\quad(n=0,1,2,\ldots).  
\end{equation}
Hence the gauge factors $g^{(n)}(x)$ are determined 
as
\begin{equation}
\begin{split}
g^{(n)}(x)&=
\frac{p^{\binom{n}{2}}e(-nQ(x))}{d^{(n)}(x)}
\cG^{(n)}(x)
\quad(n=0,1,2,\ldots). 
\end{split}
\end{equation}
Since $K^{(n)}(x)=d^{(n)}(x) I_n(t;p,q)$
with $t=\widetilde{u}$ as in \eqref{eq:deftildeu}, 
we also obtain
\begin{equation}
\begin{split}
\tau^{(n)}(x)&
=
\frac{p^{\binom{n}{2}}e(-nQ(x))}{d^{(n)}(x)}
\cG^{(n)}(x)\det\big(\psi^{(n)}_{ij}(x)\big)_{i,j=1}^{n}\\
&=p^{\binom{n}{2}}e(-nQ(x))
\cG^{(n)}(x) I_n(\widetilde{u};p,q).
\end{split}
\end{equation}
This proves Theorem \ref{thm:6B}, (1) and 
the second equality of Theorem \ref{thm:6C}.  
\par\medskip
We already know by Theorem \ref{thm:6A}
that $\tau^{(n)}(x)$ $(x\in H_{\varpi+n\delta})$ are $W(E_7)$-invariant 
for all $n=0,1,2,\ldots$.
Namely
$w(\tau^{(n)}(x))=\tau^{(n)}(x)$ for any $w\in W(E_7)$. 
This means that, for each $w\in W(E_7)$, $\tau^{(n)}(x)$ 
$(x\in H_{\varpi+n\delta})$ has a determinant 
formula 
\begin{equation}\label{eq:wtaudet}
\tau^{(n)}(x)
=
\frac{p^{\binom{n}{2}}e(-nQ(x))}{\gamma^{(n)}(x)}
\cF^{(n)}(x)\det\big(\varphi^{(n)}_{ij}(x)\big)_{i,j=1}^{n}
\end{equation}
and a multiple integral representation 
\begin{equation}\label{eq:wtauintegral}
\tau^{(n)}(x)=p^{\binom{n}{2}}e(-nQ(x))
\cF^{(n)}(x) I_n(t;p,q), 
\end{equation}
where
$\gamma^{(n)}$, 
$\cF^{(n)}$,
$\varphi_{ij}^{(n)}$ 
($i,j=1,\ldots,n$)
and $t_k$ ($k=0,1,\ldots,7$)
are specified by applying $w$ to the functions 
$d^{(n)}$, 
$\cG^{(n)}$, 
$\psi_{ij}^{(n)}$
and 
$\widetilde{u}_k$ on $H_{\varpi+n\delta}$,
respectively. 
When $w=s_0$, by the transformation
\begin{equation}
s_0(u_i)=\begin{cases}
u_i\sqrt{pq^n/u_0u_1u_2u_3} & (i=0,1,2,3),\\
u_i\sqrt{pq^n/u_4u_5u_6u_7} & (i=4,5,6,7),
\end{cases}
\end{equation}
we obtain 
\begin{equation}
\begin{split}
&\gamma^{(n)}(x)=s_0(d^{(n)}(x))
\\
&=
q^{2\binom{n}{3}}(q^{1-n}u_2u_3)^{\binom{n}{2}}
\prod_{(i,j)=(0,3),(1,2)}
\prod_{k=1}^{n}
\theta(pq/u_iu_j;p;q)_{k-1}
\theta(q^{k-n}u_i/u_j;p;q)_{k-1}
\\
&=
q^{-\binom{n+1}{3}}(u_2u_3)^{\binom{n}{2}}
\prod_{(i,j)=(0,3),(1,2)}
\prod_{k=1}^{n}
\theta(q^{1-k}u_iu_j;p;q)_{k-1}
\theta(q^{k-n}u_i/u_j;p;q)_{k-1}
\end{split}
\end{equation}
and
\begin{equation}
\begin{split}
\cF^{(n)}(x)&=
s_0(\cG^{(n)}(x))
=\prod_{0\le i<j\le n}\Gamma(q^{1-n}u_iu_j;p,q),
\\
\varphi_{ij}^{(n)}(x)&=
s_0(\psi_{ij}^{(n)}(x))
=I(q^{n-i}t_0,
q^{n-j}t_1,
q^{j-1}t_2,
q^{i-1}t_3,
t_4,t_5,t_6,t_7;
p,q),
\\
t_k&=q^{\shf(1-n)}u_k. 
\end{split}
\end{equation}
Formulas \eqref{eq:wtaudet} and \eqref{eq:wtauintegral}
for this case $w=s_0$ 
give Theorem \ref{thm:6B}, (2) and the first 
equality of Theorem \ref{thm:6C}. 
\section{\boldmath Transformation of hypergeometric $\tau$-functions}
\label{sec:8}
From the $W(E_7)$-invariant hypergeometric $\tau$-function 
on $D_\varpi$ discussed above, 
one can construct 
a class of $\tau$-functions of hypergeometric type 
by the transformations in Theorem \ref{thm:2B}. 
In this section we give some remarks on this class of 
ORG $\tau$-functions. 
\par\medskip
In what follows, for a root $\alpha\in \Delta(E_8)$ 
and a constant $\kappa\in\bC$, 
we denote by 
\begin{equation}
H_{\alpha,\kappa}=\brm{x\in V}{\ipr{\alpha}{x}=\kappa}
\end{equation}
the hyperplane of level $\kappa$ with respect to $\alpha$.  
We consider a meromorphic function 
$\tau(x)$ on the disjoint union of parallel hyperplanes
\begin{equation}
D_{\alpha,\kappa}=
\bigsqcup_{n\in\bZ} H_{\alpha,\kappa+n\delta}
\subset V.  
\end{equation}
Denoting by $\tau^{(n)}(x)$ the restriction of $\tau(x)$ 
to the $n$th hyperplane $H_{\kappa+n\delta}$, 
we say that an ORG $\tau$-function $\tau(x)$ 
on $D_{\alpha,\kappa}$ is a {\em hypergeometric
$\tau$-function} of direction $\alpha$ with initial 
level $\kappa$ if $\tau^{(n)}(x)=0$ for 
$n<0$ and $\tau^{(0)}(x)\not\equiv 0$. 

In what follows, we use the notation
\begin{equation}
\Psi_n(t;p,q)=I_n(t;p,q)\,
\prod_{0\le i<j\le 7}\Gamma(t_it_j;p,q,q). 
\end{equation}
As we have seen in \eqref{eq:transIn2}, 
this function satisfies 
\begin{equation}
\Psi_n(t;p,q)=\Psi_n(p^{\shf}q^{\shf(2-n)}t^{-1};p,q)
\end{equation}
under the balancing condition 
$t_0t_1\cdots t_7=p^2q^{4-2n}$. 

\begin{thm}\label{thm:8A}
With respect to the directions $\pm\phi$ and the 
initial levels $\pm\varpi$ $(e(\varpi)=p)$, 
the following four functions 
are 
$W(E_7)$-invariant 
hypergeometric $\tau$-functions.
\par\smallskip\noindent
\quad$(0)$ $\tau_{++}(x)$ on 
$D_{\phi,\varpi}=
\bigsqcup_{n\in\bZ} H_{\phi,\varpi+n\delta}$ $:$
\begin{equation}\label{eq:hgtau++}
\tau_{++}^{(n)}(x)
=p^{\binom{n}{2}}e(-nQ(x))\Psi_n(q^{\shf(1-n)}u;p,q)
=p^{\binom{n}{2}}e(-nQ(x))\Psi_n(p^{\shf}q^{\shf}u^{-1};p,q).
\end{equation}
\quad$(1)$ $\tau_{+-}(x)$ on 
$D_{\phi,-\varpi}=
\bigsqcup_{n\in\bZ} H_{\phi,-\varpi+n\delta}$ $:$
\begin{equation}\label{eq:hgtau+-}
\tau_{+-}^{(n)}(x)
=\Psi_n(p^{\shf}q^{\shf(1-n)}u;p,q)
=\Psi_n(q^{\shf}u^{-1};p,q).
\end{equation}
\quad$(2)$ $\tau_{-+}(x)$ on 
$D_{-\phi,\varpi}=
\bigsqcup_{n\in\bZ} H_{-\phi,\varpi+n\delta}
=\bigsqcup_{n\in\bZ} H_{\phi,-\varpi-n\delta}
$ $:$
\begin{equation}\label{eq:hgtau-+}
\tau_{-+}^{(n)}(x)
=
p^{\binom{n}{2}}e(-nQ(x))
\Psi_n(p^{\shf}q^{\shf}u;p,q)
=
p^{\binom{n}{2}}e(-nQ(x))
\Psi_n(q^{\shf(1-n)}u^{-1};p,q). 
\end{equation}
\quad$(3)$ $\tau_{--}(x)$ on $D_{-\phi,-\varpi}
=\bigsqcup_{n\in\bZ}H_{-\phi,-\varpi+n\delta}
=\bigsqcup_{n\in\bZ}H_{\phi,\varpi-n\delta}$ $:$ 
\begin{equation}\label{eq:hgtau--}
\tau_{--}^{(n)}(x)
=\Psi_n(q^{\shf}u;p,q)
=\Psi_n(p^{\shf}q^{\shf(1-n)}u^{-1};p,q).
\end{equation}
\end{thm}

\begin{proof}{Proof}
The function $\tau_{++}(x)$ is 
the $W(E_8)$-invariant hypergeometric 
$\tau$-function of direction $\phi$ with initial level $\varpi$ 
discussed in previous sections. 
We apply 
the translation by $-\phi\,\varpi$ 
to $\tau(x)=\tau_{++}(x)$ 
as in Theorem 
\ref{thm:2B}, (3) to construct a $\tau$-function 
\begin{equation}
\begin{split}
\widetilde{\tau}(x)=e(S(x;-\phi,\varpi))\tau(x+\phi\,\varpi)
\end{split}
\end{equation}
on $D_{\phi,\varpi}-\phi\,\varpi=D_{\phi,-\varpi}$. 
We look at the prefactor of $\widetilde{\tau}^{(n)}(x)$
($x\in H_{\phi,-\varpi+n\delta})$: 
\begin{equation}\label{eq:prefactor}
e(S(x;-\phi,\varpi))p^{\binom{n}{2}}e(-nQ(x))
=
e(S(x;-\phi,\varpi))e(-nQ(x)+\tbinom{n}{2}\varpi). 
\end{equation}
Noting that $\eta_{\varpi}=-1$ in this case, 
we compute
\begin{equation}
\begin{split}
S(x;-\phi,\varpi)&=\tfrac{1}{2\delta^2}\ipr{\phi}{x}
\ipr{x}{x+\phi\varpi}
=\tfrac{1}{2\delta^2}\ipr{\phi}{x}\ipr{x}{x}
+\tfrac{\varpi}{2\delta^2}\ipr{\phi}{x}^2
\\
&=\tfrac{n}{2\delta}\ipr{x}{x}+\tfrac{n^2}{2}\varpi
-\tfrac{\varpi}{2\delta^2}\ipr{x}{x}
-\tfrac{n}{2}\tfrac{\varpi^2}{\delta}
+\tfrac{\varpi^3}{2\delta^2}. 
\end{split}
\end{equation}
for $\ipr{\phi}{x}=-\varpi+n\delta$. 
Combining this with 
\begin{equation}
-nQ(x+\phi\varpi)
=-\tfrac{n}{2\delta}\ipr{x+\phi \varpi}{x+\phi\varpi}
=
-\tfrac{n}{2\delta}\ipr{x}{x}
-n^2\varpi,
\end{equation}
we obtain
\begin{equation}
\begin{split}
S(x;-\phi,\varpi)-nQ(x)+\tbinom{n}{2}\varpi
&=
-\tfrac{\varpi}{2\delta^2}\ipr{x}{x}
-n\tfrac{\varpi(\varpi+\delta)}{2\delta}
+\tfrac{\varpi^3}{2\delta^2}
\\
&=
-\tfrac{\varpi}{2\delta^2}\ipr{x}{x}
-\tfrac{\varpi(\varpi+\delta)}{2\delta^2}
\ipr{\phi}{x}
-\tfrac{\varpi^2}{2\delta}
\end{split}
\end{equation}
by $n=(\ipr{\phi}{x}+\varpi)/\delta$.  
Hence the prefactor \eqref{eq:prefactor} for 
$\widetilde{\tau}^{(n)}(x)$ is determined as
\begin{equation}
e\big(-\tfrac{\varpi}{2\delta^2}\ipr{x}{x}
-\tfrac{\omega(\omega+\delta)}{2\delta^2}
\ipr{\phi}{x}
-\tfrac{\omega^2}{2\delta}
\big).
\end{equation}
This factor does not effect on the Hirota equations, 
and can be eliminated by Theorem \ref{thm:2B}, (1).  
We thus obtain the $\tau$-function $\tau_{+-}(x)$ 
of \eqref{eq:hgtau+-}
by replacing $e(x)=u$ in the last two factors 
of $\tau_{++}(x)$ 
with 
$e(x+\phi\omega)=p^\phi u=p^{\shf}u$.  
The other two functions 
$\tau_{-+}(x)$ and $\tau_{--}(x)$ 
are obtained by 
replacing $x$ in $\tau_{++}(x)$
and $\tau_{+-}(x)$ with $-x$, respectively. 
\end{proof}

We now apply Theorem \ref{thm:2B}, (3) for constructing 
hypergeometric $\tau$-functions with initial level $0$. 
Let $a\in P$ be a vector with $\ipr{\phi}{a}=1$, so that
\begin{equation}
D_{\phi,-\varpi}+a\varpi=D_{\phi,0}
=\bigsqcup_{n\in\bZ} H_{\phi,n\delta}. 
\end{equation}
Then, from 
\begin{equation}
\tau^{(n)}_{+-}(x)
=\Psi_n(p^{\shf}q^{\shf(1-n)}u;p,q)
=\Psi_n(q^{\shf}u^{-1};p,q)
\end{equation}
we obtain 
a hypergeometric $\tau$-function 
\begin{equation}
\begin{split}
\tau_{a}(x)&=e(S(x;a,\varpi))\tau_{+-}(x-a\varpi)\\
&=
e(S(x;a,\varpi))
\Psi_n(p^{-a}p^{\shf}q^{\shf(1-n)}u;p,q)
\\
&=
e(S(x;a,\varpi))
\Psi_n(p^{a}q^{\shf}u^{-1};p,q)
\end{split}
\end{equation}
on $D_{\phi,0}$, where
\begin{equation}
S(x;a,\varpi)=-\tfrac{1}{2\delta^2}\ipr{a}{x}\ipr{x}{x-a\varpi}. 
\end{equation}
When $a\in \Delta(E_8)$, namely $\ipr{a}{a}=2$, there 
are 56 choices of $a$ with $\ipr{\phi}{a}=1$:
\begin{equation}
a=v_k+v_l,\quad \phi-v_k-v_l\quad(0\le k<l\le 7).  
\end{equation}
Those $\tau$-functions 
$\tau_{a}(x)$ on $D_{\phi,0}$ correspond 
to the 56 hypergeometric $\tau$-functions in the 
trigonometric case studied by Masuda \cite{Masuda2011}. 

In general, let $\omega=k+l\varpi
\in\Omega$ $(k,l\in\bZ)$
a period, and take two vectors $a,b\in P$ with 
$\ipr{\phi}{a}=k$, $\ipr{\phi}{b}=l+1$.  Then
\begin{equation}
e(S(x;b,\varpi))\tau_{+-}(x-a-b\varpi)
\end{equation}
is a hypergeometric 
$\tau$-function on $D_{\phi,\omega}$.  
Furthermore, 
let $\alpha\in\Delta(E_8)$ be an arbitrary root 
of the root system of type $E_8$, and choose 
a $w\in W(E_8)$ such that $\alpha=w.\phi$.  
Then 
\begin{equation}
w(e(S(x;b,\varpi))\tau_{+-}(x-a-b\varpi))
=e(S(w^{-1}.x;b,\varpi))\tau_{+-}(w^{-1}.x-a-b\varpi)
\end{equation}
provides a hypergeometric $\tau$-function 
on $D_{\alpha,\omega}$ in the direction $\alpha$ 
with initial level $\omega.$  

\section{Relation to the framework of point configurations}
\label{sec:9}
In this section, we give some remarks on how 
ORG $\tau$-functions are related to the notion
of lattice $\tau$-functions associated with 
the configuration of generic nine points in $\bP^2$. 

\subsection{\boldmath 
Realization of the affine root system $E^{(1)}_8$}

We first recall from \cite{KMNOY2006} 
the realization of the affine root system 
of type $E^{(1)}_8$ in the context of 
the configuration of generic nine points in $\bP^2$
(see also Dolgachev-Ortland \cite{DO1989}).  
We consider the 10-dimensional complex vector space 
\begin{equation}
\frh=\frh_{3,9}=
\bC e_0\oplus\bC e_1\oplus\cdots\oplus \bC e_9
\end{equation}
with basis $\br{e_0,e_1,\ldots,e_9}$, 
and define a scalar product 
(nondegenerate symmetric $\bC$-bilinear form) 
$\ipr{\cdot}{\cdot}: \frh\times \frh\to\bC$ by 
\begin{equation}
\begin{array}{c}
\ipr{e_0}{e_0}=-1,\quad \ipr{e_j}{e_j}=1\quad(j\in\br{1,\ldots,9}),
\\[4pt]
\ipr{e_i}{e_j}=0\quad(i,j\in\br{0,1,\ldots,9};\ i\ne j). 
\end{array}
\end{equation}
This vector space 
is regarded 
as the complexification of the lattice
\begin{equation}
\begin{split}
L=L_{3,9}=
\bZ e_0\oplus \bZ e_1\oplus
\bZ e_2\oplus\cdots\oplus \bZ e_9\subset \frh=\frh_{3,9}
\end{split}
\end{equation}
endowed with the symmetric $\bZ$-bilinear form 
$\ipr{\cdot}{\cdot}: L\times L\to \bZ$.  
In geometric terms, $L=L_{3,9}$ is 
the {\em Picard lattice} associated  with 
the blowup of $\bP^2$ at generic nine points 
$p_1,\ldots,p_9$.  
The vectors $e_0$ and $e_1,\ldots,e_9$ denote 
 the class of lines in $\bP^2$ and 
those of exceptional divisors 
corresponding to $p_1,\ldots,p_9$ respectively, 
and the scalar product 
$\ipr{\cdot}{\cdot}$ on $L$ represents 
the intersection form of divisor classes multiplied by $-1$. 

The root lattice of type $E^{(1)}_8$ is realized as
\begin{equation}
Q(E^{(1)}_8)=\bZ\alpha_0\oplus
\bZ\alpha_1\oplus
\cdots\oplus
\bZ\alpha_8\subset L, 
\end{equation}
where 
the {\em simple roots\/} 
$\alpha_0,\alpha_1,\ldots,\alpha_8\in\frh$ 
are defined by 
\begin{equation}
\alpha_0=e_0-e_1-e_2-e_3,\quad 
\alpha_j=e_j-e_{j+1}\quad(j=1,\ldots,8)
\end{equation}
with Dynkin diagram 
\begin{equation}\label{eq:DynkinE18}
\begin{picture}(170,30)(0,20)
\multiput(0,20)(24,0){8}{\circle{4}}
\multiput(2,20)(24,0){7}{\line(1,0){20}}
\put(48,44){\circle{4}}\put(48,22){\line(0,1){20}}
\put(-4,8){\small $\alpha_1$}
\put(20,8){\small $\alpha_2$}
\put(44,8){\small $\alpha_3$}
\put(68,8){\small $\alpha_4$}
\put(92,8){\small $\alpha_5$}
\put(116,8){\small $\alpha_6$}
\put(140,8){\small $\alpha_7$}
\put(160,8){\small $\alpha_8$}
\put(52,46){\small $\alpha_0$}
\end{picture}
\vspace{10pt}
\end{equation}
(These $\alpha_j$ are called the {\em coroots} 
$h_j$ in \cite{KMNOY2006}).
Note also that 
\begin{equation}
\begin{split}
c&=3e_0-e_1-\cdots-e_9\\
&=3\alpha_0+2\alpha_1+4\alpha_2
+6\alpha_3+5\alpha_4+4\alpha_5+3\alpha_6+2\alpha_7+\alpha_8
\in Q(E^{(1)}_8)
\end{split}
\end{equation}
satisfies $\ipr{c}{\alpha_j}=0$ for $j=0,1,\ldots,8$. 
Denoting by $\frh^\ast=\Hom_{\bC}(\frh,\bC)$ the 
dual space of $\frh$, we take the linear functions 
$\vep_{j}=\ipr{e_j}{\,\cdot}\in\frh^\ast$ 
$(j=0,1,\ldots,9)$ so that 
$\frh^\ast=\bC\vep_0\oplus\bC\vep_1\oplus
\cdots\oplus\bC\vep_9$, and 
regard $\vep=(\vep_0;\vep_1,\ldots,\vep_9)$ 
as the canonical coordinates for $\frh$.  
We often identify $\frh^\ast$ with $\frh$  
through the isomorphism 
$\nu:\ \frh\isoto\frh^\ast$ 
defined by 
$\nu(h)=\ipr{h}{\,\cdot}$ ($h\in\frh$), 
and denote 
the induced scalar product 
by the same symbol $\ipr{\cdot}{\cdot}$. 
When we regard the simple roots 
as $\bC$-linear functions 
on $\frh$, they are expressed as 
$\nu(\alpha_0)=\vep_0-\vep_1-\vep_2-\vep_3$ and
$\nu(\alpha_j)=\vep_j-\vep_{j+1}$ $(j=1,\ldots,8)$.
Setting  
$\delta=\ipr{c}{\,\cdot}=3\vep_0-\vep_1-\cdots-\vep_9\in\frh^{\ast}$,
we regard below this null root $\delta\in\frh^\ast$ as  
the scaling unit for difference equations.  

The root lattice of type $E_8$ is specified 
as $Q(E_8)=\bZ \alpha_0\oplus\bZ \alpha_1\oplus\cdots\oplus\bZ \alpha_7\subset Q(E^{(1)}_8)$.  
The vector space $\frh$ is decomposed accordingly as 
\begin{equation}\label{eq:h0cd}
\frh=\overset{\scirc}{\frh}\oplus\bC c\oplus\bC d, 
\quad
\overset{\scirc}{\frh}=
\bC\alpha_0\oplus\bC\alpha_1\oplus
\cdots\oplus\bC\alpha_7,\quad d=-e_9-\hf c, 
\end{equation}
where
\begin{equation}
\ipr{c}{h}=0,\ \ 
\ipr{d}{h}=0\quad
(h\in\overset{\scirc}{\frh});\quad
\ipr{c}{c}=0, \ \ \ipr{c}{d}=1, \ \ \ipr{d}{d}=0.
\end{equation}
The 8-dimensional subspace 
$\overset{\scirc}{\frh}\subset \frh$ can be identified 
with the vector space 
$V=\bC^8$ 
that we have used throughout this paper 
for the realization of the root lattice $P=Q(E_8)$.  
Noting that 
$\overset{\scirc}{\frh}=\brm{h\in \frh}{\ipr{c}{h}=0,\ \ipr{d}{h}=0}$
and $Q(E_8)=L\cap\overset{\scirc}{\frh}$,
we define the orthonormal basis 
$\br{v_0,v_1,\ldots,v_7}$ for $\overset{\scirc}{\frh}$ by  
\begin{equation}
v_j=e_j-\hf(e_0-e_9)+\hf c\quad (j=1,\ldots,8),\quad
v_0=-v_8. 
\end{equation}
Then the highest root and the 
simple roots of type $E_8$ are 
expressed as 
\begin{equation}
\begin{split}
&
\phi=\hf(v_0+v_1+\cdots+v_7)=c-\alpha_8, 
\\
&\alpha_0=\phi-v_0-v_1-v_2-v_3,\quad
\alpha_j=v_{j}-v_{j+1}\quad(j=1,\ldots,7)
\end{split}
\end{equation}
respectively, 
which recovers the situation of Section \ref{sec:1}. 
In what follows, we identify $\overset{\scirc}{\frh}$ 
with $V$ 
through this orthonormal 
basis $\br{v_0,v_1,\ldots,v_7}$. 
The corresponding $\bC$-linear functions 
$x_j=\ipr{v_j}{\,\cdot}\in\frh^\ast$ are realized as  
\begin{equation}
x_j=\vep_j-\hf(\vep_0-\vep_9)+\hf\delta
\quad(j=1,\ldots,8),
\quad
x_0=-x_8.
\end{equation}
\par\medskip
For each $\alpha\in\frh$ with 
$\ipr{\alpha}{\alpha}\ne 0$, 
the {\em reflection} 
$r_\alpha: \frh\to\frh$ with respect to $\alpha$ 
is defined in the standard way by  
\begin{equation}\label{eq:ralphafrh}
r_{\alpha}(h)=h-\ipr{\alpha^\vee}{h}\alpha\quad(h\in\frh),
\quad \alpha^\vee=2 \alpha/\ipr{\alpha}{\alpha}. 
\end{equation}
The affine Weyl group 
$W(E^{(1)}_8)=\la s_0, s_1,\ldots,s_8\ra$
of type $E^{(1)}_8$ 
(Coxeter group associated with diagram 
\eqref{eq:DynkinE18})
then acts faithfully 
on $\frh$ through the {\em simple reflections} 
$s_j=r_{\alpha_j}$ 
($j=0,1,\ldots.8$). 
We remark that this group contains the symmetric group 
$\mathfrak{S}_9=\la s_1,\ldots,s_8\ra$ as a subgroup 
which permutes $e_1,\ldots,e_9$.  
The affine Weyl group $W(E^{(1)}_8)$
also acts on the dual 
space $\frh^\ast$ through 
$s_j=r_{\alpha_j}: \frh^\ast\to\frh^\ast$ defined 
in the same way as \eqref{eq:ralphafrh}.  
These actions of $W(E^{(1)}_8)$ on $\frh$ and $\frh^\ast$ 
leave the two scalar products invariant.
Note also that $c\in\frh$ and $\delta\in\frh^\ast$ are 
invariant under the action of $W(E^{(1)}_8)$. 

Setting $\frh_0=\brm{\alpha\in\frh}{\ipr{c}{\alpha}=0}$, 
for each 
$\alpha\in\frh_0$ we define the {\em Kac translation} (\cite{Kac1990}) 
$T_{\alpha}: \frh\to \frh$ with respect to $\alpha$ 
by 
\begin{equation}
T_{\alpha}(h)=h+\ipr{c}{h}\alpha
-\big(\hf\ipr{\alpha}{\alpha}\ipr{c}{h}+\ipr{\alpha}{h}\big)c
\quad(h\in\frh). 
\end{equation}
It is directly verified that 
these $\bC$-linear transformations 
$T_{\alpha}\in{\rm GL}(\frh)$ ($\alpha\in\frh_{0}$) 
satisfy 
\begin{equation}
\begin{array}{cl}
(1) & \ipr{T_{\alpha}(h)}{T_{\alpha}(h')}=\ipr{h}{h'}
\quad(\alpha\in\frh_0;\ h,h'\in\frh),
\\[4pt]
(2) & T_{\alpha}T_{\beta}=T_{\beta}T_{\alpha}=T_{\alpha+\beta}
\quad(\alpha,\beta\in\frh_0),\quad
T_{kc}={\rm id}_{\frh}\quad(k\in\bC),
\\[1pt]
(3) & 
wT_{\alpha}w^{-1}=T_{w.\alpha}
\quad
(\alpha\in\frh_0, w\in W(E^{(1)}_8). 
\end{array}
\end{equation}
Note also that 
\begin{equation}
T_{\alpha}(h)=h-\ipr{\alpha}{h}c
\quad(h\in\frh, \ipr{c}{h}=0),
\quad
T_{\alpha}(c)=c
\end{equation}
for any $\alpha\in\frh_0$. 
For a $\bZ$-submodule $Q$ of $\frh_0$ given, 
we denote 
by $T(Q)\subset {\rm GL}(\frh)$
the abelian subgroup 
of Kac translations $T_{\alpha}$ ($\alpha\in Q$). 
We remark that, 
if $\alpha\in\frh_0$ and $\ipr{\alpha}{\alpha}\ne 0$, then 
the Kac translation $T_{\alpha}$ is expressed in the form 
$T_{\alpha}=r_{c-\alpha^\vee}r_{\alpha^\vee}$ 
as a product of two reflections.  
This implies in particular that the Kac translations $T_{\alpha_j}$ ($j=0,1,\ldots,8$) by simple roots belong to the affine 
Weyl group $W(E^{(1)}_{8})=\la s_0,s_1,\ldots,s_8\ra\subset 
{\rm GL}(\frh)$.  
It is known (\cite{Kac1990}) in fact that 
$W(E^{(1)}_8)$ splits into 
the semidirect product 
\begin{equation}
W(E^{(1)}_8)=T(Q(E_8))\rtimes W(E_8),\quad
W(E_8)=\la s_0,s_1,\ldots,s_7\ra. 
\end{equation}
We remark here that 
the linear action of $W(E^{(1)}_8)$ on $\frh$ 
extends to a bigger group
$T(\frh_0)W(E^{(1)}_8)=T(V)\rtimes W(E_8)$ including 
the abelian group $T(V)$ of Kac translations 
with respect to 
$V=\bC\otimes_{\bZ}Q(E_8)$.  
Note that 
$T(V)\rtimes W(E_8)$ acts also on $\frh^\ast$ 
so that $\nu: \frh\isoto\frh^{\ast}$ intertwines 
its linear actions on $\frh$ and $\frh^\ast$.

\par\medskip
Before proceeding further, we clarify 
how the linear actions of Kac translations on $\frh$ 
are related to the affine-linear actions of 
parallel translations on $V$.  
Note first that, 
for each $\kappa\in\bC$, the hyperplane 
\begin{equation}
\frh_{\kappa}=\brm{h\in\frh}{\ipr{c}{h}=\kappa}\subset \frh 
\end{equation}
is stable by $W(E^{(1)}_8)$ and by $T(\frh_0)$. 
On this hyperplane 
$\frh_\kappa$ of level $\kappa$, 
the null root $\delta=\ipr{c}{\cdot}\in\frh^\ast$ is 
identified with the constant function $\delta=\kappa$. 
For each $(\mu,\kappa)\in\bC\times\bC^\ast$, 
we define a quadratic mapping $\gamma_{(\mu,\kappa)}: V\to\frh_{\kappa}$ by
\begin{equation}
\gamma_{(\mu,\kappa)}(x)=T_{\kappa^{-1}x}(\kappa d)-\mu c
=x-(\tfrac{1}{2\kappa}\ipr{x}{x}+\mu)c+\kappa d\quad (x\in V).  
\end{equation}
These mappings  $\gamma_{(\mu,\kappa)}$ 
induce the 
parametrization 
$\gamma_{\kappa}: V\times\bC\isoto
\frh_{\kappa}: (x;\mu)\mapsto\gamma_{(\mu,\kappa)}(x) 
$ of $\frh_{\kappa}$ 
for each $\kappa\in\bC^\ast$, 
as well as the isomorphism 
$\gamma: V\times\bC\times\bC^\ast
\isoto\frh\backslash\frh_0$
of affine varieties such that 
\begin{equation}
\gamma(x;\mu,\kappa)
=x-(\tfrac{1}{2\kappa}\ipr{x}{x}+\mu)c+\kappa d
\quad(x\in V, \mu\in\bC, \kappa\in\bC^\ast). 
\end{equation}
Furthermore, this isomorphism is equivariant 
with respect to the action of the group 
$T(V)\rtimes W(E_8)$ on 
$V\times\bC\times\bC^\ast$ specified by
\begin{equation}
T_{v}w.(x;\mu,\kappa)=(w.x+\kappa v;\mu,\kappa)
\quad(v\in V, w\in W(E_8)). 
\end{equation}
Through this isomorphism, 
the coordinates $(x;\mu,\kappa)=(x_0,x_1,\ldots,x_7;\mu,\kappa)$ 
for $V\times\bC\times\bC^\ast$ and $\vep=(\vep_0;\vep_1,\ldots,\vep_9)$ for 
$\frh\backslash \frh_{0}$ are transformed into each other 
through 
\begin{equation}
\begin{split}
&x_j=\vep_j-\hf(\vep_0-\vep_9)+\hf\delta\quad(j=1,\ldots,8),
\quad x_0=-x_8,
\\
&\mu
=-\tfrac{1}{2\delta}\ipr{\vep}{\vep}, 
\quad \kappa=\delta,
\end{split}
\end{equation}
where $\ipr{\vep}{\vep}=-\vep_0^2+\vep_1^2+\cdots+\vep_9^2$, 
and by 
\begin{equation}
\begin{split}
\vep_0&=2x_0-2\ipr{\phi}{x}
+3(\tfrac{1}{2\kappa}\ipr{x}{x}+\mu+\hf\kappa),
\\
\vep_j&=x_j+x_0-\ipr{\phi}{x}+\tfrac{1}{2\kappa}\ipr{x}{x}+\mu+\hf\kappa
\quad(j=1,\ldots,7),
\\
\vep_8&=-\ipr{\phi}{x}+\tfrac{1}{2\kappa}\ipr{x}{x}+\mu+\hf\kappa,\quad
\vep_9=\tfrac{1}{2\kappa}\ipr{x}{x}+\mu-\hf\kappa,
\end{split}
\end{equation}
where $\ipr{\phi}{x}=\hf(x_0+x_1+\cdots+x_7)$, 
$\ipr{x}{x}=x_0^2+x_1^2+\cdots+x_7^2$.  

\subsection{\boldmath 
Lattice $\tau$-functions vs. ORG $\tau$-functions}

We now consider the $W(E^{(1)}_8)$-orbit 
\begin{equation}
M=M_{3,9}=W(E^{(1)}_8)\br{e_1,\ldots,e_9}
=W(E^{(1)}_8)e_9
\subset L=L_{3,9}
\end{equation}
in the Picard lattice.  
Noting that $e_9\in M$ is $W(E_8)$-invariant, we 
see that the natural mapping 
$W(E^{(1)}_8)\to M: w\mapsto w.e_9$ induces the 
bijection $Q(E_8)\isoto M: \alpha\mapsto T_{\alpha}.e_9$.  
The orbit $M=W(E^{(1)}_8)e_9$ is intrinsically characterized as 
\begin{equation}
M=\brm{\Lambda\in L}
{\ \ipr{\Lambda}{\Lambda}=1,\ \ipr{c}{\Lambda}=-1}.  
\end{equation}
To see this,  
suppose that $\Lambda\in L$ satisfies 
$\ipr{\Lambda}{\Lambda}=1$ and $\ipr{c}{\Lambda}=-1$. 
Then, the difference $\beta=\Lambda-e_9$ satisfies 
$\ipr{\beta}{\beta}+2\ipr{e_9}{\beta}=0$ and 
$\ipr{c}{\beta}=0$. 
This implies 
$\beta=\Lambda-e_9\in Q(E^{(1)}_8)$ 
and $T_{\beta}.e_9=e_9+\beta=\Lambda$.  
Taking $\alpha=\beta+\ipr{e_9}{\beta}c\in Q(E_8)$, 
we see that $\Lambda$ is uniquely expressed in the form 
\begin{equation}
\Lambda=e_9+\alpha+\hf\ipr{\alpha}{\alpha}c,
\quad \alpha\in Q(E_8), 
\end{equation}
and hence 
$\Lambda=T_{\Lambda-e_9}^{-1}.e_9=T_{\alpha}^{-1}.e_9$.
We remark that $\alpha$ is the unique 
element in $Q(E_8)$ such that $\Lambda-e_9\equiv \alpha$ 
$({\rm mod}\ \bC c)$, which we call the {\em classical part} of 
$\Lambda-e_9\in\frh_0$.

In \cite{KMNOY2006}, a system of 
{\em lattice 
$\tau$-functions} 
associated with the configurations of generic nine points in 
$\bP^2$ 
is defined as 
a family 
of dependent variables 
$\tau_{\Lambda}$ 
indexed by $\Lambda\in M$ 
which admit an action of 
the affine Weyl group $W(E^{(1)}_8)$ 
such that 
\begin{equation}\label{eq:wtauLambda}
w.\tau_{\Lambda}=\tau_{w.\Lambda}\qquad 
(\Lambda\in M,\ w\in W(E^{(1)}_8))
\end{equation}
and satisfy the quadratic relations 
\begin{equation}\label{eq:Hirota39}
[\vep_{jk}][\vep_{jkl}]\tau_{e_i}\tau_{e_0-e_i-e_l}
+
[\vep_{ki}][\vep_{kil}]\tau_{e_j}\tau_{e_0-e_j-e_l}
+[\vep_{ij}][\vep_{ijl}]\tau_{e_k}\tau_{e_0-e_k-e_l}=0
\end{equation}
for all quadruples of mutually distinct  $i,j,k,l\in\br{1,\ldots,9}$,
where $\vep_{ij}=\vep_i-\vep_j$ and $\vep_{ijk}=\vep_0-\vep_i-\vep_j-\vep_k$.  
(Here we use the notation $\tau_{\Lambda}$ 
instead of $\tau(\Lambda)$ as in \cite{KMNOY2006} 
to make clear that $\Lambda$ 
is {\em not} an independent variable, but an index.)  
Under the condition \eqref{eq:wtauLambda}, 
$\tau_{e_9}$ is $W(E_8)$-invariant, and 
all the functions 
$\tau_{\Lambda}$ $(\Lambda\in M)$ are expressed as 
$\tau_{\Lambda}=T_{\Lambda-e_9}^{-1}.\tau$ 
in terms of a single 
$W(E_8)$-invariant function $\tau=\tau_{e_9}$. 
We also remark that, if equation \eqref{eq:Hirota39} holds 
for some quadruple of distinct $i,j,k,l\in\br{1,\ldots,9}$, then 
it holds for all quadruples as a result of 
the action of $\mathfrak{S}_9\subset W(E^{(1)}_8)$.
In the following, we use the notation
$\sigma_{h}=[\nu(h)]=[\ipr{h}{\cdot}]$ for $h\in\frh$, 
so that 
$\sigma_{e_{ij}}=[\vep_{ij}]$ and 
$\sigma_{e_{ijk}}=[\vep_{ijk}]$ 
where $e_{ij}=e_i-e_j$ and $e_{ijk}=e_0-e_i-e_j-e_k$. 
In this notation, equation \eqref{eq:Hirota39} is 
rewritten as 
\begin{equation}\label{eq:Hirota39A}
\sigma_{e_{jk}}\sigma_{e_{jkl}}
\tau_{e_i}\tau_{e_0-e_i-e_l}
+\sigma_{e_{ki}}\sigma_{e_{kil}}
\tau_{e_j}\tau_{e_0-e_j-e_l}
+\sigma_{e_{ij}}\sigma_{e_{ijl}}
\tau_{e_k}\tau_{e_0-e_k-e_l}=0. 
\end{equation}
As we will see below, equations \eqref{eq:Hirota39} 
can be rewritten in a $W(E^{(1)}_8)$-invariant form. 

To clarify the situation, let $X$ be a left 
$W(E^{(1)}_8)$-set. 
Noting that functions of the form 
$\sigma_{\alpha}=[\ipr{\alpha}{\cdot}]$ ($\alpha\in Q(E^{(1)}_8)$) are defined over $\overline{\frh}=\frh/\bC c$, 
we suppose that 
a $W(E^{(1)}_8)$-equivariant mapping 
$\gamma: X\to \overline{\frh}$ is given. 
Regarding those 
$\sigma_{\alpha}$
as functions defined on $X$ through 
$\gamma: X\to\overline{\frh}$, 
we can consider systems of lattice $\tau$-functions 
$\tau_{\Lambda}$ ($\Lambda\in M$) defined on $X$.  
In order to compare this notion with that of 
ORG $\tau$-functions on $X$, 
we assume that 
the extension $T(\hf Q(E_8))\rtimes W(E_8)$ 
of $W(E^{(1)}_8)$ 
acts on $X$ so that $\gamma$ is an equivariant mapping. 
For a function $\varphi$ defined on a subset $U\subseteq X$, 
we define the action 
of $w\in T(\hf Q(E_8))\rtimes W(E_8)$ on $\varphi$ 
to be the function $w.\varphi$ 
on $w.U$ such that 
$(w.\varphi)(x)=\varphi(w^{-1}.x)$ ($x\in w.U$). 

\begin{prop}\label{prop:9A}
Let $\tau$ be a $W(E_8)$-invariant function on $X$ 
and set $\tau_{\Lambda}=T_{\Lambda-e_9}^{-1}.\tau$
for each $\Lambda\in M$. 
Then the following three conditions are equivalent $:$
\par\smallskip\noindent
$({\rm a})$ The equation 
\begin{equation}\label{eq:Hirota39sA}
\begin{array}{c}
\sigma_{e_{jk}}\sigma_{e_{jkl}}
\tau_{e_i}\tau_{e_0-e_i-e_l}
+\sigma_{e_{ki}}\sigma_{e_{kil}}
\tau_{e_j}\tau_{e_0-e_j-e_l}
+\sigma_{e_{ij}}\sigma_{e_{ijl}}
\tau_{e_k}\tau_{e_0-e_k-e_l}=0
\end{array}
\end{equation}
holds for each 
quadruple of  
mutually distinct $i,j,k,l\in\br{1,\ldots,9}$. 
\newline
$({\rm b})$  The equation 
\begin{equation}\label{eq:Hirota39sB}
\sigma_{u_1-u_2}\sigma_{u_1-u_2'}
\tau_{u_0}\tau_{u_0'}
+
\sigma_{u_2-u_0}\sigma_{u_2-u_0'}
\tau_{u_1}\tau_{u_1'}
+
\sigma_{u_0-u_1}\sigma_{u_0-u_1'}
\tau_{u_2}\tau_{u_2'}=0
\end{equation}
holds for each sextuple 
of points $u_i,u_i'\in M$ $(i=0,1,2)$ such that
\begin{equation}
\begin{split}
u_0+u_0'=u_1+u_1'=u_2+u_2',
\quad
\ipr{u_i-u_i'}{u_j-u_j'}=4\delta_{ij}\quad(i,j=0,1,2).
\end{split}
\end{equation}
$({\rm c})$  The equation 
\begin{equation}\label{eq:Hirota39sD}
\sigma_{a_1\pm a_2}
T_{a_0}.\tau\,T_{a_0}^{-1}.\tau
+
\sigma_{a_2\pm a_0}
T_{a_1}.\tau\,T_{a_1}^{-1}.\tau
+
\sigma_{a_0\pm a_1}
T_{a_2}.\tau\,T_{a_2}^{-1}.\tau=0
\end{equation}
holds 
for each triple of vectors 
$a_0,a_1,a_2\in\hf Q(E_8)$
such that
\begin{equation}\label{eq:Hirota39sDCond}
\begin{split}
\ipr{a_i}{a_j}=\delta_{i,j},\quad {}\pm a_i\pm a_j\in Q(E_8)
\quad(i,j=0,1,2). 
\end{split}
\end{equation}
\end{prop}

\begin{proof}{Proof}
Note first that equation \eqref{eq:Hirota39sA} is 
a special case of \eqref{eq:Hirota39sB} 
where $u_0=e_i, u_1=e_j, u_2=e_k$ and $v=e_0-e_l$.  
Hence condition $({\rm b})$ implies $({\rm a})$.
We consider equation \eqref{eq:Hirota39sB} for 
a sextuple $u_i,u_i'\in M$ ($i=0,1,2$) as in $({\rm b})$. 
Introducing 
\begin{equation}
\begin{split}
&g=\hf(u_i+u_i'),\quad
b_i=\hf(u_i-u_i')\quad(i=0,1,2),
\\
&
u_i=g+b_i,\quad u_i'=g-b_i\quad(i=0,1,2), 
\end{split}
\end{equation}
we rewrite 
equation \eqref{eq:Hirota39sB} into the 
equation 
\begin{equation}\label{eq:Hirota39sC}
\sigma_{b_1\pm b_2}
\tau_{g+b_0}\tau_{g-b_0}
+
\sigma_{b_2\pm b_0}
\tau_{g+b_1}\tau_{g-b_1}
+
\sigma_{b_0\pm b_1}
\tau_{g+b_2}\tau_{g-b_2}=0
\end{equation}
for a sextuple 
of points $g\pm b_i\in M$ $(i=0,1,2)$ such that
\begin{equation}
\begin{split}
\ipr{b_i}{b_j}=\delta_{i,j},\quad {}\pm b_i\pm b_j\in Q(E^{(1)}_8)
\quad(i,j=0,1,2). 
\end{split}
\end{equation}
In this setting, 
$b_0, b_1, b_2$ and $g$ are characterized by the conditions
\begin{equation}\label{eq:Condgb012}
\begin{split}
&g\in \hf L,\quad
\ipr{g}{g}=0,\ \ipr{c}{g}=-1
\\
&
b_i\in \hf L,\ \ 
\ipr{c}{b_i}=0,\ \ipr{g}{b_i}=0,\ 
\ipr{b_i}{b_j}=\delta_{ij},\ {}\pm b_i\pm b_j\in Q(E^{(1)}_8).
\end{split}
\end{equation}
Since $b_i\in\hf Q(E^{(1)}_8)$ ($i=0,1,2$), 
they are expressed as
\begin{equation}\label{eq:gatob}
b_i=a_i+k_ic,\quad
k_i=\ipr{g}{a_i}\in\hf\bZ\quad
(i=0,1,2), 
\end{equation}
where 
\begin{equation}\label{eq:Conda012}
a_i\in \hf Q(E_8),\ \ 
\ipr{g}{a_i}\in \hf \bZ,\ 
\ipr{a_i}{a_j}=\delta_{ij},\ {}\pm a_i\pm a_j\in Q(E_8). 
\end{equation}
Note that $\ipr{c}{g-e_9}=0$ and hence 
$g-e_9\in\hf Q(E_8^{(1)})$.  
In this situation we have
\begin{equation}
T_{g-e_9}^{-1}.a_i=a_i+\ipr{g-e_9}{a_i}c=a_i+k_ic=b_i\quad(i=0,1,2), 
\end{equation}
and hence we see that equations \eqref{eq:Hirota39sC} and 
\eqref{eq:Hirota39sD} are transformed into each other by 
the actions of $T_{g-e_9}$ and $T_{g-e_9}^{-1}$.
These arguments show that condition $({\rm c})$ implies 
$({\rm b})$. 
Note also that equation \eqref{eq:Hirota39sD} is a 
special case of \eqref{eq:Hirota39sC} where $g=e_9+\hf c$ and 
$b_i=a_i$ ($i=0,1,2$). 
We finally show that $({\rm a})$ implies $({\rm c})$.  
Suppose that equation \eqref{eq:Hirota39sA} holds for some 
quadruple of distinct $i,j,k,l\in\br{1,\ldots,9}$. 
Since it is a special case of \eqref{eq:Hirota39sB}, we 
see that equation \eqref{eq:Hirota39sD} holds for some 
triple $a_0,a_1,a_2\in \hf Q(E_8)$ satisfying \eqref{eq:Hirota39sDCond}, 
namely for some $C_3$-frame $\br{\pm a_0,\pm a_1,\pm a_2}$ 
in the terminology of Section \ref{sec:1}.  
Since the Weyl group $W(E_8)$ 
acts transitively on the set of all $C_3$-frames, 
we see that equation \eqref{eq:Hirota39sD} holds for all $C_3$-frames, 
which implies $({\rm c})$.  
\end{proof}

By abuse of terminology, we say that a function $\tau$ 
is an ORG $\tau$-function if it satisfies the non-autonomous 
Hirota equations \eqref{eq:Hirota39sD} for all $C_3$-frames 
relative to $Q(E_8)$.  
Proposition \ref{prop:9A} means that a family of functions 
$\tau_{\Lambda}$ ($\Lambda\in M$) on $X$ 
is a system of lattice $\tau$-functions if and only if 
$\tau=\tau_{e_9}$ is a $W(E_8)$-invariant ORG $\tau$-function 
on $X$.  
General ORG $\tau$-functions which are not necessarily 
$W(E_8)$-invariant can be interpreted as the system of lattice 
$\tau$-functions over a covering space of $X$. 

Setting 
\begin{equation}
\widetilde{X}=W(E_8)\times X
=\brm{(w,x)}{w\in W(E_8),\ x\in X},
\end{equation}
we define an action of $T(\hf Q(E_8))\rtimes W(E_8)$ 
on $\widetilde{X}$ 
by 
\begin{equation}
T_{\alpha}w.(w', x)=(ww', T_{\alpha}w.x)
\qquad (\alpha\in \hf Q(E_8),\ w\in W(E_8))
\end{equation}
so that the projection $\widetilde{X}\to X$ is equivariant. 
Let $U$ be a subset of $X$ and suppose that 
$U$ is stable by the group $T(Q(E_8))$ of translations.  
Then the subset $\widetilde{U}\subseteq \widetilde{X}$ 
defined as 
\begin{equation}
\widetilde{U}
=\brm{(w,x)\in \widetilde{X}}{w\in W(E_8),\ x\in w.U}
=\bigsqcup_{w\in W(E_8)} \br{w}\times w.U
\end{equation}
is stable by the action of 
$W(E^{(1)}_8)=T(Q(E_8))\rtimes W(E_8)$.  
To a function $\tau$ defined on $U$, we associate a 
function $\widetilde{\tau}$ on $\widetilde{U}$ by setting 
\begin{equation}
\widetilde{\tau}(w,x)=(w.\tau)(x)=\tau(w^{-1}.x)
\quad (w\in W(E_8),\ x\in w.U). 
\end{equation}
The function $\widetilde{\tau}$ on $\widetilde{U}$
is $W(E_8)$-invariant 
and $\tau$ on $U$ is recovered from $\widetilde{\tau}$ 
as 
$\widetilde{\tau}(1,x)=\tau(x)$ $(x\in U)$. 
Also, any $W(E_8)$-invariant function on $\widetilde{U}$ 
is obtained in this way from a function on $U$. 
Note also that the function $\tau$ on $U$ 
satisfies equations 
\begin{equation}\label{eq:HirotaU}
\sigma_{a_1\pm a_2}
T_{a_0}.\tau\ T_{a_0}^{-1}.\tau
+
\sigma_{a_2\pm a_0}
T_{a_1}.\tau\ T_{a_1}^{-1}.\tau
+
\sigma_{a_0\pm a_1}
T_{a_2}.\tau\ T_{a_2}^{-1}.\tau=0
\end{equation}
for all $C_3$-frames $\br{\pm a_0,\pm a_1,\pm a_2}$ 
relative to $Q(E_8)$ 
if and only if 
the $W(E_8)$-invariant 
function $\widetilde{\tau}$ on $\widetilde{U}$
satisfies 
\begin{equation}
\sigma_{a_1\pm a_2}
T_{a_0}.\widetilde{\tau}\ T_{a_0}^{-1}.\widetilde{\tau}
+
\sigma_{a_2\pm a_0}
T_{a_1}.\widetilde{\tau}\ T_{a_1}^{-1}.\widetilde{\tau}
+
\sigma_{a_0\pm a_1}
T_{a_2}.\widetilde{\tau}\ T_{a_2}^{-1}.\widetilde{\tau}=0
\end{equation}
for all $C_3$-frames $\br{\pm a_0,\pm a_1,\pm a_2}$ 
relative to $Q(E_8)$. 
Applying Proposition \ref{prop:9A} to the mapping
\begin{equation}
\widetilde{\gamma}:\ \widetilde{U}\to\overline{\frh}:\quad 
\widetilde{\gamma}(w,x)=\gamma(x)
\quad(w\in W(E_8), x\in w.U),
\end{equation}
and 
$\widetilde{\tau}$ 
on $\widetilde{U}$, we obtain the following 
characterization of an ORG $\tau$-function on $U$.  
\begin{prop}\label{prop:9B}
For a function $\tau$ on a subset $U\subseteq X$, 
consider the function $\widetilde{\tau}$ defined on $\widetilde{U}$.  
Then $\tau$ is an ORG $\tau$-function on $U$ 
if and only if the functions 
$\widetilde{\tau}_{\Lambda}=T_{\Lambda-e_9}^{-1}.\widetilde{\tau}$ $(\Lambda\in M)$ 
form a system 
lattice $\tau$-functions on $\widetilde{U}$. 
\end{prop}

We apply this proposition for constructing lattice $\tau$-functions from ORG $\tau$-functions discussed in this paper.  
Fixing a nonzero constant $\kappa\in\bC^\ast$, 
let $D\subset V$ be a subset such that $D+Q(E_8)\kappa=D$,
and take an 
ORG $\tau$-function $\tau=\tau(x)$ on $D$ 
with $\delta=\kappa$
in the sense of Definition \ref{dfn:ORGtaufn}. 
We then define the action of $T_{v}w\in
T(V)\rtimes W(E_8)$ $(v\in V,\ w\in W(E_8))$
on $V$ by
\begin{equation}
T_{v}w.x=w.x+\kappa v\qquad(x\in V).
\end{equation}
We denote by 
$\pi:\ \frh\to V$ the orthogonal 
projection to 
$\overset{\scirc}{\frh}=V$ 
in \eqref{eq:h0cd}, 
and by $\pi_{\kappa}: \frh_{\kappa}\to V$ its restriction 
to $\frh_{\kappa}$. 
This projection $\pi_{\kappa}: \frh_{\kappa}\to V$ 
is equivariant with respect to the action of 
$T(V)\rtimes W(E_8)$
and 
compatible with the scalar product 
in the sense 
$\ipr{v}{\pi_{\kappa}(h)}=\ipr{v}{h}$, $(v\in V, h\in\frh)$. 
Introducing a subset $U$ of $\frh_{\kappa}$ by 
\begin{equation}
U=\pi_{\kappa}^{-1}(D)
=\brm{h\in\frh_{\kappa}}{\pi_{\kappa}(h)\in D}\subset \frh_{\kappa}, 
\end{equation}
we regard $\tau$ as a function on $U$ through the 
projection $\pi_{\kappa}: U\to D$.  
Note that the isomorphism 
\begin{equation}
\gamma_{\kappa}: V\times \bC\isoto\frh_{\kappa}:\ 
\gamma_{\kappa}(x,\mu)=x-
(\tfrac{1}{2\kappa}\ipr{x}{x}+\mu)c+\kappa d,
\end{equation}
induces the parametrization 
$\gamma_{\kappa}: D\times\bC \isoto U$ 
of $U=\pi_{\kappa}^{-1}(D)$ with an invariant 
parameter $\mu\in\bC$.  
Then obtain a $W(E_8)$-invariant function 
$\widetilde{\tau}$ on 
\begin{equation}\label{eq:Utilde}
\widetilde{U}=\bigsqcup_{w\in W(E_8)} \br{w}\times w.U
\subset W(E_8)\times \frh_{\kappa}
\end{equation}
by setting 
\begin{equation}
\widetilde{\tau}(w,h)=\tau(\pi_{\kappa}(w^{-1}.h))
=\tau(w^{-1}.x)
\quad(w\in W(E_8),\ h\in w.U),
\end{equation}
where $x=\pi_{\kappa}(h)\in w.D$. 
By Proposition \ref{prop:9B}, 
the family of functions 
$\widetilde{\tau}_{\Lambda}
=T_{\Lambda-e_9}^{-1}.\widetilde{\tau}$ 
gives a system of lattice $\tau$-functions 
on $\widetilde{U}$.  If we take 
the classical part 
$\alpha=\pi_{0}(\Lambda-e_9)\in Q(E_8)$ of 
$\Lambda-e_9$, 
then we have $\widetilde{\tau}_{\Lambda}=T_{\alpha}^{-1}.\widetilde{\tau}$, and hence
\begin{equation}
\begin{split}
\widetilde{\tau}_{\Lambda}(w,h)
=\widetilde{\tau}(w,T_{\alpha}.h)
=\tau(w^{-1}.\pi_{\kappa}(T_{\alpha}.h))
=\tau(w^{-1}.(x+\kappa \alpha)),
\end{split}
\end{equation}
for any $w\in W(E_8)$ and $h\in w.U$, where 
$x=\pi_{\kappa}(h)\in w.D$. 
\begin{thm}\label{thm:9C}
Let $\kappa\in\bC^\ast$ be a generic constant. 
Suppose that a subset $D\subseteq V$ is stable 
by $W(E_8)$ and $D+Q(E_8)\kappa=D$.  
Let $\tau(x)$ be an ORG $\tau$-function on $D$ 
with $\delta=\kappa$.  
For each $\Lambda\in M$, define  a function 
$\widetilde{\tau}_{\Lambda}$ on $\widetilde{U}$ of \eqref{eq:Utilde}
by 
\begin{equation}
\widetilde{\tau}_{\Lambda}(w,h)
=\tau(w^{-1}.(x+\kappa \alpha))
\quad(w\in W(E_8),\ h\in w.U)
\end{equation}
with $\alpha=\pi_0(\Lambda-e_9)\in Q(E_8)$ 
and $x=\pi_{\kappa}(h)\in w.D$.  
Then $\widetilde{\tau}_{\Lambda}$ 
$(\Lambda\in M)$ form 
a system of lattice $\tau$-functions on $\widetilde{U}$.
\end{thm}

For example, 
we consider the hypergeometric 
the ORG $\tau$-function $\tau(x)=\tau_{+-}(x)$ 
of Theorem \ref{thm:8A} defined on 
\begin{equation}
\begin{split}
&D=D_{\phi,-\varpi}=\bigsqcup_{n\in\bZ} H_{\phi,-\varpi+n\kappa},
\\ 
&H_{\phi,-\varpi+n\kappa}=\brm{x\in V}{\ipr{\phi}{x}=-\varpi+n\kappa}
\quad(n\in \bZ), 
\end{split}
\end{equation}
under the identification $\delta=\kappa$.  
The components $\tau^{(n)}=\tau|_{H_{\phi,-\varpi+n\kappa}}$ 
($n\in\bZ$)
are then given by $\tau^{(n)}(x)=0$ ($n<0$) and 
\begin{equation}
\tau^{(n)}(x)=\Psi_n(p^{\shf}q^{\shf(1-n)}u;p,q)=
\Psi_n(q^{\shf}u^{-1};p,q)
\quad(n=0,1,2,\ldots)
\end{equation}
in terms of the elliptic hypergeometric integral
\eqref{eq:Psin}.  Here, 
$p=e(\varpi)$, $q=e(\kappa)$ and 
$u_j=e(x_j)$ 
denote the multiplicative variables corresponding to $x_j$
$(j=0,1,\ldots,7)$. 
In this case, the subset 
$U=\pi_{\kappa}^{-1}(D)\subset \frh_{\kappa}$ is specified 
as  $U=\bigsqcup_{n\in\bZ} U_{n}$, where 
\begin{equation}
U_{n}=\brm{h\in \frh_{\kappa}}{\ipr{\phi}{h}
=-\varpi+n\kappa}
=
\brm{h\in \frh_{\kappa}}{\ipr{\alpha_8}{h}=\varpi+(1-n)\kappa}
\quad(n\in\bZ). 
\end{equation}
We regard the ORG $\tau$-function 
$\tau(x)$ as a function on $U$ through the coordinates
\begin{equation}
x_j=\vep_j-\hf(\vep_0-\vep_9)+\hf\kappa
\quad(j=1,\ldots,8),
\quad x_0=-x_8. 
\end{equation}
The corresponding 
lattice $\tau$-functions $\widetilde{\tau}_{\Lambda}$
$(\Lambda\in M)$ on 
\begin{equation}
\widetilde{U}=\bigsqcup_{w\in W(E_8)}\br{w}\times w.U
\end{equation}
are defined as 
\begin{equation}
\widetilde{\tau}_{\Lambda}(w,h)=\tau(w^{-1}.(x+\kappa \alpha))
\quad(w\in W(E_8),\ h\in w.U), 
\end{equation}
where 
$\alpha=\pi_0(\Lambda-e_9)\in Q(E_8)$ and $x=\pi_{\kappa}(h)$. 
Since 
\begin{equation}
e_j-e_9=v_0+v_j-\phi+c
\quad(j=1,\ldots,8), 
\end{equation}
the nine fundamental 
$\tau$-functions $\widetilde{\tau}_{e_j}$ 
($j=1,\ldots,9$)
are specified as
\begin{equation}
\begin{split}
&\widetilde{\tau}_{e_j}(w,h)
=\tau(w^{-1}.(x+\kappa(v_0+v_j-\phi))\quad (j=1,\ldots,7),
\\
&\widetilde{\tau}_{e_8}(w,h)=\tau(w^{-1}.(x-\kappa\phi)),
\qquad 
\widetilde{\tau}_{e_9}(w,h)=\tau(w^{-1}.x).
\end{split}
\end{equation}
for $w\in W(E_8)$ and $h\in w.U$. 

\subsection{\boldmath 
Remarks on the $\bP^1\times \bP^1$ picture}

The difference Painlev\'e equation of type $E_8$ can also 
be formulated in terms of the configuration of generic eight 
points in $\bP^1\times \bP^1$.   
In this case, following the formulation of \cite{KNY2015}
we use the 
Picard lattice 
\begin{equation}
L=\bZ \bsf{h}_1\oplus
\bZ \bsf{h}_2\oplus
\bZ \bsf{e}_1\oplus
\bZ \bsf{e}_2\oplus
\cdots\oplus
\bZ \bsf{e}_8 
\end{equation}
with the symmetric bilinear form 
$\ipr{\cdot}{\cdot}: L\times L\to \bZ$
defined by 
\begin{equation}
\begin{split}
&\ipr{\bsf{h}_i}{\bsf{h}_i}=0\ \ (i=1,2),\quad
\ipr{\bsf{h}_1}{\bsf{h}_2}=\ipr{\bsf{h}_2}{\bsf{h}_1}=-1,
\\
&
\ipr{\bsf{e}_i}{\bsf{e}_j}=\delta_{ij}\quad(i,j=1,\ldots,8). 
\end{split}
\end{equation}
If we denote by $(f,g)$ the inhomogeneous coordinates 
of $\bP^1\times \bP^1$,  
$\bsf{h}_1$ and $\bsf{h}_2$ represent the classes of lines 
$f=\mbox{\rm const.}$ and $g=\mbox{\rm const.}$ respectively, 
and $\bsf{e}_1,\ldots,\bsf{e}_8$ the classes of 
exceptional divisors corresponding to the generic eight
points.  
This Picard lattice and its symmetric bilinear form 
are identified with those we have used 
in the $\bP^2$ picture through the change of bases
\begin{equation}
\begin{split}
\bsf{h}_1=e_0-e_2,\  \bsf{h}_2=e_0-e_1,\  
\bsf{e}_1=e_0-e_1-e_2,\ 
\bsf{e}_j=e_{j+1}\  (j=2,\ldots,8),
\end{split}
\end{equation}
and 
\begin{equation}
e_0=\bsf{h}_1+\bsf{h}_2-\bsf{e}_1,\ 
e_1=\bsf{h}_1-\bsf{e}_1,\ 
e_2=\bsf{h}_2-\bsf{e}_1,\  
e_j=\bsf{e}_{j-1}\quad(j=3,\ldots,9). 
\end{equation}
The simple roots $\alpha_0,\alpha_1,\ldots,\alpha_8$ are now 
expressed as
\begin{equation}
\alpha_0=\bsf{e}_1-\bsf{e}_2,\ 
\alpha_1=\bsf{h}_1-\bsf{h}_2,\  
\alpha_2=\bsf{h}_2-\bsf{e}_1-\bsf{e}_2,\ 
\alpha_j=\bsf{e}_{j-1}-\bsf{e}_j\  (j=3,\ldots,8).  
\end{equation}
The complex vector space $\frh=\bC\otimes_{\bZ}L$ is 
decomposed as $\frh=\overset{\scirc}{\frh}\oplus\bC c
\oplus \bC d$ where
\begin{equation}
c=2\bsf{h}_1+2\bsf{h}_2-\bsf{e}_1-\cdots-\bsf{e}_8,\quad
d=-\bsf{e}_8-\hf c. 
\end{equation}
In this realization, 
the orthonormal basis $\br{v_0,v_1,\ldots,v_8}$ for 
$V=\overset{\scirc}{\frh}$ is given by 
\begin{equation}
\begin{split}
&v_1=\bsf{h}_1-\bsf{e}_1-
\hf(\bsf{h}_1+\bsf{h}_2-\bsf{e}_1-\bsf{e}_8)+\hf c,\\
&v_2=\bsf{h}_2-\bsf{e}_1
-\hf(\bsf{h}_1+\bsf{h}_2-\bsf{e}_1-\bsf{e}_8)+\hf c,\\
&v_j=\bsf{e}_{j-1}
-\hf(\bsf{h}_1+\bsf{h}_2-\bsf{e}_1-\bsf{e}_8)+\hf c\ \  
(j=3,\ldots,8),\ \ 
v_0=-v_8. 
\end{split}
\end{equation}
Accordingly, the coordinates $x=(x_0,x_1,\ldots,x_7)$ for 
$V$ are given by
\begin{equation}
\begin{split}
&x_1=\eta_1-\ep_1-\hf(\eta_1+\eta_2-\ep_1-\ep_8)+\hf \delta,\\
&x_2=\eta_2-\ep_1-\hf(\eta_1+\eta_2-\ep_1-\ep_8)+\hf \delta,\\
&x_j=\ep_{j-1}-\hf(\eta_1+\eta_2-\ep_1-\ep_8)+\hf \delta\ \  
(j=3,\ldots,8),\ \ 
x_0=-x_8. 
\end{split}
\end{equation}
where 
$\eta_i=\ipr{\bsf{h}_i}{\cdot}$ $(i=1,2)$
and $\ep_j=\ipr{\bsf{e}_j}{\cdot}$ $(j=1,\ldots,8)$.

In this framework, 
a system of lattice $\tau$-functions is defined as 
a family of dependent variables $\tau_{\Lambda}$, 
indexed by the same orbit 
\begin{equation}
M=W(E^{(1)}_8)\br{\bsf{e}_1,\ldots,\bsf{e}_8}
=W(E^{(1)}_8)\bsf{e}_8\subset L,
\end{equation}
which admit an action of $W(E^{(1)}_8)$ such that 
\begin{equation}
w.\tau_{\Lambda}=\tau_{w.\Lambda}\quad(w\in W(E^{(1)}_8),\ 
\Lambda\in M)
\end{equation}
and satisfy the quadratic relations 
\begin{equation}
\sigma_{\bssf{e}_{jk}}
\sigma_{\bssf{e}_{r;jk}}
\tau_{\bssf{e}_i}\tau_{\bssf{h}_r-\bssf{e}_i}
+
\sigma_{\bssf{e}_{ki}}
\sigma_{\bssf{e}_{r;ki}}
\tau_{\bssf{e}_j}\tau_{\bssf{h}_r-\bssf{e}_j}
+
\sigma_{\bssf{e}_{ij}}
\sigma_{\bssf{e}_{r;ij}}
\tau_{\bssf{e}_k}\tau_{\bssf{h}_r-\bssf{e}_k}=0
\end{equation}
for $r=1,2$ and for mutually distinct $i,j,k\in\br{1,\ldots,8}$, 
where $\bsf{e}_{ij}=\bsf{e}_i-\bsf{e}_j$ and $\bsf{e}_{r;ij}=\bsf{h}_r-\bsf{e}_i-\bsf{e}_j$. 
Through the expression 
$\tau_{\Lambda}=T_{\Lambda-\bssf{e}_8}^{-1}\tau_{\bssf{e}_8}$, 
this notion of lattice $\tau$-functions is interpreted 
by that of ORG $\tau$-functions 
in the same way as in Proposition \ref{prop:9B} and Theorem \ref{thm:9C}. 

\appendix

\section*{\LARGE Appendix}
\section{Proof of Theorem \ref{thm:3C}}
\label{sec:A}
In this Appendix, we give a proof of Theorem \ref{thm:3C}.  
The following proof is essentially the same as the 
argument in Masuda \cite[Section 3]{Masuda2011}.
We first prove Theorem \ref{thm:3C} under an 
additional assumption that 
$\tau^{(n-1)}(x)$ and $\tau^{(n)}(x)$ are $W(E_7)$-invariant.  
After that we explain how the general case can be reduced 
to the invariant case.  
\par\medskip

\subsection{Preliminary remark}
\par\medskip
We begin by a general remark on the 
Hirota equations associated with $C_3$-frames. 
Fixing a $C_l$-frame 
$A=\br{\pm a_1,\ldots,\pm a_l}$ ($l=3,\ldots,8$), 
we suppose that a function $\tau(x)$ satisfies the 
Hirota equation
\begin{equation}\label{eq:A0Hirota}
\tau(x\pm a_i\delta)[\ipr{a_j\pm a_k}{x}]+
\tau(x\pm a_j\delta)[\ipr{a_k\pm a_i}{x}]+
\tau(x\pm a_k\delta)[\ipr{a_i\pm a_j}{x}]=0
\end{equation}
for any triple $i,j,k\in\br{1,\ldots,l}$.  
We also assume that 
$[\ipr{a_i\pm a_j}{x}]\ne 0$ on the 
domain of definition of $\tau$
for any 
distinct pair $i,j\in\br{1,\ldots,l}$.  
Then \eqref{eq:A0Hirota} can be written as
\begin{equation}
\tau(x\pm a_k\delta)=
\frac{
\tau(x\pm a_i\delta)[\ipr{a_k\pm a_j}{x}]
-\tau(x\pm a_j\delta)[\ipr{a_k\pm a_i}{x}]
}
{[\ipr{a_i\pm a_j}{x}]}
\end{equation}
for any $k\in\br{1,\ldots,l}$.  
In view of this expression, for each $u\in V$ 
we define
\begin{equation}
f_{ij}(x;u)=\frac{\tau(x\pm a_i)[\ipr{u\pm a_j}{x}]
-\tau(x\pm a_j)[\ipr{u\pm a_i}{x}]}
{[\ipr{a_i\pm a_j}{x}]}
\end{equation}
for each distinct pair $i,j\in\br{1,\ldots,l}$, so that 
$f_{ij}(x;a_k)=\tau(x\pm a_k\delta)$ $(k\in\br{1,\ldots,l})$.  
A simple but important observation is that 
the three-term relation \eqref{eq:three-term}  
of the function 
$[z]$ implies the 
functional equation 
\begin{equation}
f_{ij}(x;u_0)[\ipr{u_1\pm u_2}{x}]+
f_{ij}(x;u_1)[\ipr{u_2\pm u_0}{x}]+
f_{ij}(x;u_2)[\ipr{u_0\pm u_1}{x}]=0
\end{equation}
for any $u_0,u_1,u_2\in V$.  
From this, 
for any distinct pair $r,s\in\br{1,\ldots,l}$ we obtain 
\begin{equation}
\begin{split}
f_{ij}(x;u)[\ipr{a_r\pm a_s}{x}]
&=
f_{ij}(x;a_r)[\ipr{u\pm a_s}{x}]-
f_{ij}(x;a_s)[\ipr{u\pm a_r}{x}]
\\
&=
\tau(x\pm a_r\delta)[\ipr{u\pm a_s}{x}]-
\tau(x\pm a_s\delta)[\ipr{u\pm a_r}{x}]
\\
&=f_{rs}(x;u)[\ipr{a_r\pm a_s}{x}], 
\end{split}
\end{equation}
and hence $f_{ij}(x;u)=f_{rs}(x;u)$.  
Summarizing these arguments, we have
\begin{lem}\label{lem:A1}
Let $A=\br{\pm a_1,\ldots,\pm a_l}$ be 
a $C_l$-frame and suppose that $\tau(x)$ satisfies the 
Hirota equation \eqref{eq:A0Hirota} for any triple 
$i,j,k\in\br{1,\ldots,l}$.  Then 
the function 
\begin{equation}
f(x;u)=\frac{\tau(x\pm a_i)[\ipr{u\pm a_j}{x}]
-\tau(x\pm a_j)[\ipr{u\pm a_i}{x}]}
{[\ipr{a_i\pm a_j}{x}]}
\quad(x,u\in V)
\end{equation}
does not depend on the choice of 
distinct $i,j\in\br{1,\ldots,l}$.  
Furthermore, it satisfies
\begin{equation}
f(x;u_0)[\ipr{u_1\pm u_2}{x}]+
f(x;u_1)[\ipr{u_2\pm u_0}{x}]+
f(x;u_2)[\ipr{u_0\pm u_1}{x}]=0
\end{equation}
for any $u_0,u_1,u_2\in V$, and 
\begin{equation}
f(x;a_k)=\tau(x\pm a_k\delta)\quad(k=1,\ldots,l). 
\end{equation}
\end{lem}

Returning to the setting of 
Theorem \ref{thm:3C}, 
we suppose furthermore that 
$\tau^{(n-1)}(x)$ on $H_{c+(n-1)\delta}$ and 
$\tau^{(n)}(x)$ on $H_{c+n\delta}$ are 
$W(E_7)$-invariant. 
For a function $\varphi=\varphi(x)$, the action $w.\varphi$ 
of $w\in W(E_7)$ 
is defined by $(w.\varphi)(x)=\varphi(w^{-1}.x)$. 
We say that $\varphi$ is invariant with respect to 
$w$ if $w.\varphi=\varphi$, namely, 
$\varphi(w^{-1}.x)=\varphi(x)$.  Note also that, 
for the function $\psi$ defined by 
$\psi(x)=\varphi(x+v)$ $(v\in V)$, we have 
$(w.\psi)(x)=(w.\varphi)(x+w.v)$, 
and hence $(w.\psi)(x)=\varphi(x+w.v)$ if $\varphi$ 
is $w$-invariant.

\subsection{\boldmath Definition of $\tau^{(n+1)}$: 
$W(E_7)$-invariance and $(\II_1)_n$}  
We first show that there exists a unique 
$W(E_7)$-invariant 
(meromorphic) function $\tau^{(n+1)}(x)$ on 
$H_{c+(n+1)\delta}$ that satisfies the bilinear 
equations of type $(\II_1)_n$. 

We consider the $C_8$-frame $A=\br{\pm a_0,\pm a_1,\ldots, \pm a_7}$ defined by 
\begin{equation}\label{eq:C801}
\begin{split}
&a_0=\hf(v_0-v_1+\phi),\quad a_1=\hf(v_1-v_0+\phi),
\\
&a_j=v_j+\hf(v_0+v_1-\phi)\quad(j=2,\ldots,7). 
\end{split}
\end{equation}
Note that $\ipr{\phi}{a_0}=\ipr{\phi}{a_1}=1$ and
$\ipr{\phi}{a_j}=0$ ($j=2,\ldots,7$).
Then, from the assumption 
that $\tau^{(n)}(x)$ satisfies the bilinear equations of type 
$(\II_0)_n$, by Lemma \ref{lem:A1}
it follows that 
the function 
\begin{equation}
f(x;u)=\frac{\tau^{(n)}(x\pm a_i\delta)[\ipr{u\pm a_j}{x}]
-\tau^{(n)}(x\pm a_j\delta)[\ipr{u\pm a_i}{x}]}
{[a_i\pm a_j]}
\end{equation}
does not depend on the choice of distinct $i,j\in\br{2,\ldots,7}$. 
In view of the bilinear equations of type $(\II_1)_n$, we 
define the function 
$\tau^{(n+1)}(x)$ on $H_{c+(n+1)\delta}$ by the equation 
\begin{equation}
\tau^{(n+1)}(x+a_0\delta)
\tau^{(n-1)}(x-a_0\delta)=f(x;a_0).  
\end{equation}
This implies that $\tau^{(n+1)}(x)$ satisfies 
\begin{equation}\label{eq:Toda0ij}
\begin{split}
&\tau^{(n+1)}(x+a_0\delta)
\tau^{(n-1)}(x-a_0\delta)[\ipr{a_i\pm a_j}{x}]
\\
&=
\tau^{(n)}(x\pm a_i\delta)[\ipr{a_0\pm a_j}{x}]
-\tau^{(n)}(x\pm a_j\delta)[\ipr{a_0\pm a_i}{x}]
\end{split}
\end{equation}
for any distinct $i,j\in \br{2,\ldots,7}$.  
By replacing $x$ with $x-a_0\delta$, 
these equations can be written as 
\begin{equation}
\begin{split}
&\tau^{(n+1)}(x)
\\
&=
\frac{\tau^{(n)}(x\!-\!(a_0\pm a_i)\delta)
[\ipr{a_0\pm a_j}{x}\!-\!\delta]
-\tau^{(n)}(x\!-\!(a_0\pm a_j)\delta)
[\ipr{a_0\pm a_i}{x}\!-\!\delta]}
{\tau^{(n-1)}(x-2a_0\delta)[\ipr{a_i\pm a_j}{x}]}
\end{split}
\end{equation}
for any distinct $i,j\in\br{2,\ldots,7}$. 
In terms of the basis $v_0,v_1,\ldots,v_7$ for $V$, 
we have 
\begin{equation}\label{eq:tau(n+1)(x)}
\begin{split}
\tau^{(n+1)}(x)
&=\frac{1}{
\tau^{(n-1)}(x-(\phi+v_0-v_1)\delta)
[\ipr{v_j-v_i}{x}]
[\ipr{\phi-v_{01ij}}{x}]}
\\
&\quad\cdot
\big(
\tau^{(n)}(x-v_{0i}\delta)
\tau^{(n)}(x-(\phi-v_{1i})\delta)
[\ipr{v_{0j}}{x}-\delta][\ipr{\phi-v_{1j}}{x}-\delta]
\\
&\qquad\mbox{}-
\tau^{(n)}(x-v_{0j}\delta)
\tau^{(n)}(x-(\phi-v_{1j})\delta)
[\ipr{v_{0i}}{x}-\delta][\ipr{\phi-v_{1i}}{x}-\delta]
\big)
\end{split}
\end{equation}
for any distinct $i,j\in\br{2,\ldots,7}$, 
where $v_{ab}=v_a+v_b$ and 
$v_{abcd}=v_{a}+v_{b}+v_{c}+v_{d}$. 

We next show that this function $\tau^{(n+1)}(x)$ 
is invariant under the action of the Weyl group
\begin{equation}
W(E_7)=\la s_0,s_1,\ldots,s_6\ra;
\quad
s_0=r_{\phi-v_{0123}},\ s_{j}=r_{v_j-v_{j+1}}\ \ (j=1,\ldots,6).
\end{equation}
Since 
$\tau^{(n-1)}(x)$ and $\tau^{(n)}(x)$ are $W(E_7)$-invariant, 
from the fact that 
the expression \eqref{eq:tau(n+1)(x)} does not 
depend on the choice of $i,j\in\br{2,\ldots,7}$, 
it follows that 
$\tau^{(n+1)}(x)$ 
is invariant with respect to $s_2,\ldots,s_6$.
To see the invariance with respect to $s_0$, 
we take $(i,j)=(2,3)$: 
\begin{equation}\label{eq:tau(n+1)(x)23}
\begin{split}
\tau^{(n+1)}(x)
&=\frac{1}{
\tau^{(n-1)}(x-(\phi+v_0-v_1)\delta)
[\ipr{v_3-v_2}{x}]
[\ipr{\phi-v_{0123}}{x}]}
\\
&\quad\cdot
\big(
\tau^{(n)}(x-v_{02}\delta)
\tau^{(n)}(x-(\phi\!-\!v_{12})\delta)
[\ipr{v_{03}}{x}-\delta][\ipr{\phi\!-\!v_{13}}{x}-\delta]
\\
&\qquad\mbox{}-
\tau^{(n)}(x-v_{03}\delta)
\tau^{(n)}(x-(\phi\!-\!v_{13})\delta)
[\ipr{v_{02}}{x}-\delta][\ipr{\phi\!-\!v_{12}}{x}-\delta]
\big).
\end{split}
\end{equation}
Since $s_0(v_{ij})=\phi-v_{kl}$ for $\br{i,j,k,l}=\br{0,1,2,3}$,
this expression is manifestly invariant with respect to $s_0$. 
It remains to show that $\tau^{(n+1)}(x)$ is invariant 
with respect to $s_1$.  
The $s_1$-invariance of $\tau^{(n+1)}(x)$ is equivalent to 
the equality of 
\begin{equation}
\begin{split}
L&=
\tau^{(n-1)}(x-(\phi+v_0-v_2)\delta)
[\ipr{v_3-v_1}{x}]
\\
&\quad\cdot
\big(
\tau^{(n)}(x-v_{02}\delta)
\tau^{(n)}(x-(\phi-v_{12})\delta)
[\ipr{v_{03}}{x}-\delta][\ipr{\phi-v_{13}}{x}-\delta]
\\
&\qquad\mbox{}-
\tau^{(n)}(x-v_{03}\delta)
\tau^{(n)}(x-(\phi-v_{13})\delta)
[\ipr{v_{02}}{x}-\delta][\ipr{\phi-v_{12}}{x}-\delta]
\big)
\end{split}
\end{equation}
and
\begin{equation}
\begin{split}
R&=
\tau^{(n-1)}(x-(\phi+v_0-v_1)\delta)
[\ipr{v_3-v_2}{x}]
\\
&\quad\cdot
\big(
\tau^{(n)}(x-v_{01}\delta)
\tau^{(n)}(x-(\phi-v_{12})\delta)
[\ipr{v_{03}}{x}-\delta][\ipr{\phi-v_{23}}{x}-\delta]
\\
&\qquad\mbox{}-
\tau^{(n)}(x-v_{03}\delta)
\tau^{(n)}(x-(\phi-v_{23})\delta)
[\ipr{v_{01}}{x}-\delta][\ipr{\phi-v_{12}}{x}-\delta]
\big).
\end{split}
\end{equation}
Expanding these as $L=L_1-L_2$ and $R=R_1-R_2$, 
we look at 
\begin{equation}
\begin{split}
&L_1-R_1
\\
&=
\tau^{(n)}(x-(\phi-v_{12})\delta)
[\ipr{v_{03}}{x}-\delta]
\\
&\quad\cdot\big(
\tau^{(n-1)}(x-(\phi+v_0-v_2)\delta)
\tau^{(n)}(x-v_{02}\delta)
[\ipr{v_3-v_1}{x}]
[\ipr{\phi-v_{13}}{x}-\delta]
\\
&\qquad\mbox{}-
\tau^{(n-1)}(x-(\phi+v_0-v_1)\delta)
\tau^{(n)}(x-v_{01}\delta)
[\ipr{v_3-v_2}{x}]
[\ipr{\phi-v_{23}}{x}-\delta]
\big).
\end{split}
\end{equation}
Setting $u_i=\hf \phi-v_i$
and $y=x-(\hf\phi+v_0)$, we compute 
the last factor as
\begin{equation}
\begin{split}
&
\tau^{(n-1)}(y-u_2\delta)\tau^{(n)}(y+u_2\delta)
[(u_1\pm u_3|y)]
\\
&\quad\mbox{}
-\tau^{(n-1)}(y-u_1\delta)\tau^{(n)}(y+u_1\delta)
[(u_2\pm u_3|y)]
\\
&=
\tau^{(n-1)}(y-u_3\delta)
\tau^{(n)}(y+u_3\delta)
[(u_1\pm u_2|y)]
\\
&=
\tau^{(n-1)}(x-(\phi+v_0-v_3)\delta)
\tau^{(n)}(x-v_{03}\delta)
[\ipr{v_2-v_1}{x}][\ipr{\phi-v_{12}}{x}-\delta]
\end{split}
\end{equation}
by the bilinear equation of type $(\I)_{n-1/2}$ 
for the $C_3$-frame $\br{\pm u_1,\pm u_2,\pm u_3}$.
Hence we have
\begin{equation}
\begin{split}
L_1-R_1&=
\tau^{(n-1)}(x-(\phi+v_0-v_3)\delta)
\tau^{(n)}(x-v_{03}\delta)
\tau^{(n)}(x-(\phi-v_{12})\delta)
\\
&\quad\cdot
[\ipr{v_{03}}{x}-\delta]
[\ipr{v_2-v_1}{x}][\ipr{\phi-v_{12}}{x}-\delta]. 
\end{split}
\end{equation}
On the other hand, we have 
\begin{equation}
\begin{split}
&L_2-R_2
\\
&=
\tau^{(n)}(x-v_{03}\delta)[\ipr{\phi-v_{12}}{x}-\delta]
\\
&\quad\cdot
\big(
\tau^{(n-1)}(x-(\phi+v_0-v_2)\delta)
\tau^{(n)}(x-(\phi\!-\!v_{13}))
[\ipr{v_3\!-\!v_1}{x}]
[\ipr{v_{02}}{x}-\delta]
\\
&\qquad\mbox{}-
\tau^{(n-1)}(x-(\phi+v_0-v_1)\delta)
\tau^{(n)}(x-(\phi\!-\!v_{23})\delta)
[\ipr{v_3\!-\!v_2}{x}]
[\ipr{v_{01}}{x}-\delta]
\big). 
\end{split}
\end{equation}
Setting $b_i=\hf v_{0123}-v_i$ ($i=0,1,2,3$) 
and $z=x-(\phi-b_0)\delta$, we compute the last factor as
\begin{equation}
\begin{split}
&
\tau^{(n-1)}(z-b_2\delta)\tau^{(n)}(z+b_2\delta)
[\ipr{b_1\pm b_3}{z}]
\\
&\quad\mbox{}-
\tau^{(n-1)}(z-b_1\delta)\tau^{(n)}(z+b_1\delta)
[\ipr{b_2\pm b_3}{z}]
\\
&=
\tau^{(n-1)}(z-b_3\delta)\tau^{(n)}(z+b_3\delta)
[\ipr{b_1\pm b_2}{z}]
\\
&=
\tau^{(n-1)}(x-(\phi+v_0-v_3)\delta)
\tau^{(n)}(x-(\phi-v_{12})\delta)
[\ipr{v_2-v_1}{x}][\ipr{v_{03}}{x}-\delta]
\end{split}
\end{equation}
by the bilinear equation of type $(\I)_{n-1/2}$ 
for the $C_3$-frame $\br{\pm b_1,\pm b_2,\pm b_3}$. 
Hence
\begin{equation}
\begin{split}
L_2-R_2&=
\tau^{(n-1)}(x-(\phi+v_0-v_3)\delta)
\tau^{(n)}(x-v_{03}\delta)
\tau^{(n)}(x-(\phi-v_{12})\delta)
\\
&\quad\cdot
[\ipr{v_{03}}{x}-\delta]
[\ipr{v_2-v_1}{x}]
[\ipr{\phi-v_{12}}{x}-\delta]
\\
&=L_1-R_1, 
\end{split}
\end{equation}
which implies $L=R$ as desired. 

Recall that $W(E_7)$ acts transitively on 
the set of all $C_3$-frames of 
type $\II_1$.  
Since $\tau^{(n+1)}(x)$ is 
$W(E_7)$-invariant and 
satisfies \eqref{eq:Toda0ij}, 
it readily 
satisfies the bilinear equations 
of type $(\II_1)_n$ for all $C_3$-frames of type $\II_1$. 

\subsection{\boldmath $(\II_1)_n\ \Longrightarrow\ (\II_2)_n$}

Since $\tau^{(n+1)}(x)$ is $W(E_7)$-invariant, 
we have only to show that it satisfies the bilinear 
equation of type $(\II_2)_n$ for a particular 
$C_3$-frame of type $\II_2$. 

Taking the $C_8$-frame \eqref{eq:C801} of type $\II$ 
as before, we look at the bilinear relations
\begin{equation}\label{eq:II10}
\begin{split}
&\tau^{(n+1)}(x+a_k\delta)
\tau^{(n-1)}(x-a_k\delta)
\\
&=
\frac{\tau^{(n)}(x\pm a_2\delta)[\ipr{a_k\pm a_3}{x}]
-
\tau^{(n)}(x\pm a_3\delta)[\ipr{a_k\pm a_2}{x}]}
{[\ipr{a_2\pm a_3}{x}]}
\end{split}
\end{equation}
of type $(\II_1)_n$ for 
$k=0,1$. 
From these we have
\begin{equation}
\begin{split}
&\tau^{(n+1)}(x+a_0\delta)
\tau^{(n-1)}(x-a_0\delta)[\ipr{a_2\pm a_1}{x}]
\\
&\quad\mbox{}
-
\tau^{(n+1)}(x+a_1\delta)
\tau^{(n-1)}(x-a_1\delta)[\ipr{a_2\pm a_0}{x}]
\\
&=
\frac{\tau^{(n)}(x\pm a_2\delta)}{[\ipr{a_2\pm a_3}{x}]}
\big(
[\ipr{a_0\pm a_3}{x}]
[\ipr{a_2\pm a_1}{x}]
-
[\ipr{a_1\pm a_3}{x}]
[\ipr{a_2\pm a_0}{x}]
\big)
\\
&
=\tau^{(n)}(x\pm a_2\delta)
[\ipr{a_0\pm a_1}{x}],
\end{split}
\end{equation}
which is the bilinear equation 
of type $(\II_2)_n$ for the $C_3$-frame 
$\br{\pm a_0,\pm a_1,\pm a_2}$ of type $\II_2$.  

\subsection{\boldmath $(\II_1)_n\ \Longrightarrow\ (\I)_{n+1/2}$}

We have only to show that the bilinear equation 
$(\I)_{n+1/2}$ holds for some particular 
$C_3$-frame of type $\I$. 

Taking the $C_8$-frame of \eqref{eq:C801}, 
we look at the bilinear relation
\begin{equation}
\begin{split}
&\tau^{(n+1)}(x)\tau^{(n-1)}(x-(\phi+v_0-v_1)\delta)
[\ipr{v_3-v_2}{x}][\ipr{\phi-v_{0123}}{x}]
\\
&=
\tau^{(n)}(x-(\phi-v_{12})\delta)
\tau^{(n)}(x-v_{02}\delta)
[\ipr{v_{03}}{x}-\delta]
[\ipr{\phi-v_{13}}{x}-\delta]
\\
&\quad\mbox{}-
\tau^{(n)}(x-(\phi-v_{13})\delta)
\tau^{(n)}(x-v_{03})\delta)
[\ipr{v_{02}}{x}-\delta]
[\ipr{\phi-v_{12}}{x}-\delta]
\end{split}
\end{equation}
of type $(\II_1)_n$. 
Replacing $x$ with $x+(v_0-v_1)\delta$,
we rewrite this formula into 
\begin{equation}\label{eq:Toda}
\begin{split}
&\tau^{(n+1)}(x+(v_0-v_1)\delta)\tau^{(n-1)}(x-\phi\delta)
[\ipr{v_3-v_2}{x}][\ipr{\phi-v_{0123}}{x}]
\\
&=
\tau^{(n)}(x-(\phi-v_{02})\delta)
\tau^{(n)}(x-v_{12}\delta)
[\ipr{v_{03}}{x}]
[\ipr{\phi-v_{13}}{x}]
\\
&\quad\mbox{}-
\tau^{(n)}(x-(\phi-v_{03})\delta)
\tau^{(n)}(x-(v_{13})\delta)
[\ipr{v_{02}}{x}]
[\ipr{\phi-v_{12}}{x}]. 
\end{split}
\end{equation}
Multiplying this by 
$\tau^{(n)}(x+(v_{01}-\phi)\delta)[\ipr{\phi-v_{23}}{x}]$, 
we obtain 
\begin{equation}
\begin{split}
&
\tau^{(n+1)}(x+(v_0-v_1)\delta)
\tau^{(n)}(x+(v_{01}-\phi)\delta)
[\ipr{v_3-v_2}{x}][\ipr{\phi-v_{23}}{x}]
\\
&\quad\cdot
\tau^{(n-1)}(x-\phi\delta)
[\ipr{\phi-v_{0123}}{x}]
\\
&=
\tau^{(n)}(x-u_{01}\delta)
\tau^{(n)}(x-u_{02}\delta)
\tau^{(n)}(x-v_{12}\delta)
[\ipr{u_{13}}{x}]
[\ipr{u_{23}}{x}]
[\ipr{v_{03}}{x}]
\\
&\quad\mbox{}-
\tau^{(n)}(x-u_{01}\delta)
\tau^{(n)}(x-u_{03}\delta)
\tau^{(n)}(x-v_{13}\delta)
[\ipr{u_{12}}{x}]
[\ipr{u_{23}}{x}]
[\ipr{v_{02}}{x}]
\end{split}
\end{equation}
where $u_i=\hf\phi-v_i$ and $u_{ij}=u_i+u_j=\phi-v_{ij}$. 
Then, applying the cyclic permutation $(123)$ 
to this formula, we see that the sum of those three 
vanishes term by term.  
It means that 
\begin{equation}
\begin{split}
\tau^{(n+1)}(x+(v_0-v_1)\delta)
\tau^{(n)}(x+(v_{01}-\phi)\delta)
[\ipr{v_3-v_2}{x}][\ipr{\phi-v_{23}}{x}]+\cdots=0, 
\end{split}
\end{equation}
namely
\begin{equation}
\begin{split}
\tau^{(n+1)}(x+(u_1-u_0)\delta)
\tau^{(n)}(x-(u_0+u_1)\delta)
[\ipr{u_2\pm u_3}{x}]+\cdots=0. 
\end{split}
\end{equation}
Replacing $x$ by $x+u_0\delta$, we obtain 
the bilinear equation 
\begin{equation}
\begin{split}
\tau^{(n+1)}(x+u_1\delta)
\tau^{(n)}(x-u_1\delta)
[\ipr{u_2\pm u_3}{x}]+\cdots=0
\end{split}
\end{equation}
of type $(\I)_{n+1/2}$ for the $C_3$-frame 
$\br{\pm u_1,\pm u_2,\pm u_3}$ of type 
$\I$. 

\subsection{\boldmath $(\II_1)_{n}\ \Longrightarrow\ (\II_0)_{n+1}$}

In order to show that $\tau^{(n+1)}(x)$ satisfies 
the bilinear equations of type $(\II)_{n+1}$, we 
take the $C_8$-frame $A=\br{\pm a_0,\pm a_1,\ldots
\pm a_7}$ of type $\II$ defined by 
\begin{equation}
\begin{split}
&a_j=v_j+\hf(v_{67}-\phi)\quad(j=0,1,2,3,4,5);
\\
&a_6=\hf(v_6-v_7+\phi),\ \ a_7=\hf(v_7-v_6+\phi). 
\end{split}
\end{equation}
We prove that $\tau^{(n+1)}(x)$ satisfies the bilinear 
equation 
\begin{equation}
\tau^{(n+1)}(x+(a_0\pm a_3)\delta)[\ipr{a_4\pm a_5}{x}]
+\cdots=0
\end{equation}
of type $(\II_0)_{n+1}$ for the $C_3$-frame 
$\br{\pm a_3,\pm a_4,\pm a_5}$.
In terms of the basis $v_0,v_1,\ldots,v_7$ for $V$, 
this equation is expressed as 
\begin{equation}
\tau^{(n+1)}(x+(v_0-v_3)\delta)
\tau^{(n+1)}(x+(\phi-v_{1245})\delta)
[\ipr{v_4-v_5}{x}]
[\ipr{\phi-v_{0123}}{x}]
+\cdots=0. 
\end{equation}

We now look at the bilinear equation 
\begin{equation}
\begin{split}
&\tau^{(n+1)}(x+(v_0-v_4)\delta)\tau^{(n-1)}(x-\phi\delta)
[\ipr{v_2-v_1}{x}][\ipr{\phi-v_{0124}}{x}]
\\
&=
\tau^{(n)}(x-(\phi-v_{01})\delta)
\tau^{(n)}(x-v_{14}\delta)
[\ipr{v_{02}}{x}]
[\ipr{\phi-v_{24}}{x}]
\\
&\quad\mbox{}-
\tau^{(n)}(x-(\phi-v_{02})\delta)
\tau^{(n)}(x-v_{24}\delta)
[\ipr{v_{01}}{x}]
[\ipr{\phi-v_{14}}{x}]
\end{split}
\end{equation}
of type $(\II_1)_n$. 
Applying $s_0$ to this formula, we obtain 
\begin{equation}
\begin{split}
&\tau^{(n+1)}(x+(\phi-v_{1234})\delta)
\tau^{(n-1)}(x-\phi\delta)
[\ipr{v_2-v_1}{x}][\ipr{v_{3}-v_{4}}{x}]
\\
&=
\tau^{(n)}(x-v_{23}\delta)
\tau^{(n)}(x-v_{14}\delta)
[\ipr{\phi-v_{13}}{x}]
[\ipr{\phi-v_{24}}{x}]
\\
&\quad\mbox{}-
\tau^{(n)}(x-v_{13})\delta)
\tau^{(n)}(x-v_{24}\delta)
[\ipr{\phi-v_{23}}{x}]
[\ipr{\phi-v_{14}}{x}]. 
\end{split}
\end{equation}
Using these formulas, we compute 
\begin{equation}
\begin{split}
&\tau^{(n+1)}(x+(v_0-v_3)\delta)
\tau^{(n+1)}(x+(\phi-v_{1245})\delta)
[\ipr{v_4-v_5}{x}]
[\ipr{\phi-v_{0123}}{x}]
\\
&\quad\cdot
\tau^{(n-1)}(x-\phi)^2
[\ipr{v_2-v_1}{x}]^2
\\
&
=
\big(
\tau^{(n)}(x-u_{01}\delta)
\tau^{(n)}(x-v_{13}\delta)
[\ipr{v_{02}}{x}]
[\ipr{u_{23}}{x}]
\\
&\qquad\mbox{}-
\tau^{(n)}(x-u_{02})
\delta)
\tau^{(n)}(x-v_{23}\delta)
[\ipr{v_{01}}{x}]
[\ipr{u_{13}}{x}]
\big)
\\
&\quad\cdot
\big(
\tau^{(n)}(x-v_{24}\delta)
\tau^{(n)}(x-v_{15}\delta)
[\ipr{u_{14}}{x}]
[\ipr{u_{25}}{x}]
\\
&\qquad\mbox{}-
\tau^{(n)}(x-v_{14})\delta)
\tau^{(n)}(x-v_{25}\delta)
[\ipr{u_{24}}{x}]
[\ipr{u_{15}}{x}]
\big)
\\
&=
\tau^{(n)}(x-u_{01}\delta)[\ipr{v_{02}}{x}]
\\
&\quad\cdot
\big\{
\tau^{(n)}(x-v_{13}\delta)
\tau^{(n)}(x-v_{24})\delta)
\tau^{(n)}(x-v_{15}\delta)
[\ipr{u_{23}}{x}]
[\ipr{u_{14}}{x}]
[\ipr{u_{25}}{x}]
\\
&\qquad\mbox{}-
\tau^{(n)}(x-v_{13}\delta)
\tau^{(n)}(x-v_{14})\delta)
\tau^{(n)}(x-v_{25}\delta)
[\ipr{u_{23}}{x}]
[\ipr{u_{24}}{x}]
[\ipr{u_{15}}{x}]
\big\}
\\
&\quad\mbox{}+
\tau^{(n)}(x-u_{02}\delta)[\ipr{v_{01}}{x}]
\\
&\quad\cdot
\big\{
\tau^{(n)}(x-v_{23}\delta)
\tau^{(n)}(x-v_{14})\delta)
\tau^{(n)}(x-v_{25}\delta)
[\ipr{u_{13}}{x}]
[\ipr{u_{24}}{x}]
[\ipr{u_{15}}{x}]
\\
&\qquad\mbox{}
-\tau^{(n)}(x-v_{23}\delta)
\tau^{(n)}(x-v_{24})\delta)
\tau^{(n)}(x-v_{15}\delta)
[\ipr{u_{13}}{x}]
[\ipr{u_{14}}{x}]
[\ipr{u_{25}}{x}]
\big\}
\end{split}
\end{equation}
Applying the cyclic permutation $(345)$ to this formula, 
we can directly observe that the sum of those three vanishes 
term by term. 
This completes the proof of Theorem 
\ref{thm:3C} in the case where 
$\tau^{(n-1)}(x)$ and $\tau^{(n)}(x)$ are 
$W(E_7)$-invariant.

\subsection{General case}
We now consider the general case where 
$\tau^{(n-1)}(x)$ and 
$\tau^{(n)}(x)$ are not necessarily $W(E_7)$-invariant. 
In such a situation, we need to deal with all the  
transforms $w.\tau^{(n-1)}$ and $w.\tau^{(n)}$ 
by $w\in W(E_7)$ simultaneously.  

For each hyperplane 
$H_\kappa=\brm{x\in V}{\ipr{\phi}{x}=\kappa}$ 
($\kappa\in\bC$) perpendicular to $\phi$, 
we introduce the covering space
\begin{equation}
\widetilde{H}_{\kappa}=W(E_7)\times H_{\kappa}, 
\end{equation}
and define the action of $w\in W(E_7)$
and the translation $T_{v}$ ($v\in H_0$) by
\begin{equation}
w.(g,x)=(wg,w.x),\quad T_{v}.(g,x)=(g,x+v\delta)
\quad(g\in W(E_7),\ x\in H_{\kappa}),  
\end{equation}
so that $wT_{v}=T_{w.v}w$.  
For each function $\psi$ on $\widetilde{H}_{\kappa}$, 
the induced actions 
$w.\psi$ ($w\in W(E_7)$) and $T_v.\psi$ ($v\in H_0$) 
are described as 
\begin{equation}
(w.\psi)(g,x)
=\psi(w^{-1}g,w^{-1}x),
\quad
(T_v.\psi)(g,x)
=\psi(g,x-v\delta). 
\end{equation}
For each function $\varphi$ on $H_{\kappa}$, 
we define a function 
$\widetilde{\varphi}$ on $\widetilde{H}_{\kappa}$ by 
\begin{equation}
\widetilde{\varphi}(g,x)=(g.\varphi)(x)=\varphi(g^{-1}.x)
\qquad(g\in W(E_7),\ x\in H_{\kappa}). 
\end{equation} 
Then this function 
$\widetilde{\varphi}$ 
is $W(E_7)$-invariant, and $\varphi$ is recovered 
from $\widetilde{\varphi}$ by  
$\varphi(x)=\widetilde{\varphi}(1,x)$ ($x\in H_{\kappa}$). 
Conversely, a function $\psi$ on 
$\widetilde{H}_{\kappa}$ is $W(E_7)$-invariant 
if and only if it is obtained as the lift 
$\psi=\widetilde{\varphi}$ 
of a function $\varphi$ on $H_{\kappa}$. 

Returning to the setting of Theorem \ref{thm:3C}, 
for the functions $\tau^{(k)}$ on $H_{c+k\delta}$
we consider the lifts 
$\widetilde{\tau}^{(k)}$ on $\widetilde{H}_{c+k\delta}$
($k=n-1,n$).  
Note that, for each $g\in W(E_7)$, 
$g.\tau^{(n-1)}$ and $g.\tau^{(n)}$ also satisfy 
the bilinear equations $(\I)_{n-1/2}$ and $(\II_0)_n$.  
Hence, the lifted functions $\widetilde{\tau}^{(k)}$ 
$(k=n-1,n)$ satisfy the 
bilinear equations corresponding to 
$(\I)_{n-1/2}$ and $(\II_0)_n$.
Formally, those bilinear equations can be written as
\begin{equation}
(\I)_{n-1/2}:\quad
[\ipr{a_1\pm a_2}{x}]\,
\widetilde{\tau}^{(n-1)}(g,x-a_0)
\widetilde{\tau}^{(n)}(g,x+a_0)+\cdots=0
\end{equation}
and 
\begin{equation}
(\II)_{n}:\quad
[\ipr{a_1\pm a_2}{x}]\,\widetilde{\tau}^{(n)}(g,x\pm a_0)+\cdots=0.  
\end{equation}
As in \eqref{eq:Toda0ij}, 
we can define a function 
$\psi=\psi(g,x)$ on $\widetilde{H}_{c+(n+1)\delta}$ so that 
\begin{equation}
\begin{split}
&\psi(g,x+a_0\delta)\widetilde{\tau}^{(n-1)}(g,x-a_0\delta)
[\ipr{a_i\pm a_j}{x}]
\\
&=
\widetilde{\tau}^{(n)}(g,x-(a_0\pm a_i)\delta)
[\ipr{a_0\pm a_j}{x}]
-
\widetilde{\tau}^{(n)}(g,x\pm a_j\delta)[\ipr{a_0\pm a_i}{x}]
\end{split}
\end{equation}
for any distinct $i,j\in\br{2,\ldots,7}$.  
Since $\widetilde{\tau}^{(k)}$ 
are $W(E_7)$-invariant on $\widetilde{H}_{c+k\delta}$
$(k=n-1,n)$, 
applying the previous arguments to the lifted functions, 
we see that $\psi$ is also $W(E_7)$-invariant and 
satisfies the bilinear equations corresponding to 
$(\II_1)_n$, 
$(\II_2)_n$, 
$(\I)_{n+1/2}$, 
$(\II_0)_{n+1}$. 
Since $\psi$ is $W(E_7)$-invariant, 
it is expressed as $\psi=\widetilde{\tau}^{(n+1)}$ 
with a function on $\tau^{(n+1)}$ on $H_{c+(n+1)\delta}$: 
\begin{equation}
\psi(g,x)=\widetilde{\tau}^{(n+1)}(g,x)=(g.\tau^{(n+1)})(x), 
\quad \tau^{(n+1)}(x)=\psi(1,x).  
\end{equation}
Then the bilinear equations 
$(\II_1)_n$, 
$(\II_2)_n$, 
$(\I)_{n+1/2}$, 
$(\II_0)_{n+1}$ for 
$\widetilde{\tau}^{(n-1)}$, 
$\widetilde{\tau}^{(n)}$ and 
$\psi=\widetilde{\tau}^{(n+1)}$ 
means 
that $\tau^{(n+1)}$ satisfies the 
corresponding 
bilinear 
equations of four types 
and 
that the recursive construction of  
$\tau^{(n+1)}$ from $\tau^{(n-1)}$, $\tau^{(n)}$ 
is equivariant with respect to the action of $W(E_7)$.  

\par\bigskip\noindent

\section*{Acknowledgments}

The author would like to thank his coworkers, 
Kenji Kajiwara, Tetsu Masuda, Yasuhiro Ohta and 
Yasuhiko Yamada, who allowed him to write in this form 
a paper largely based on the collaboration with them.  
He is also grateful to Yasushi Komori and Jun-ichi 
Shiraishi for providing valuable comments and suggestions 
in various stages of this work.

\section*{Funding}
 
This work was supported by JSPS Kakenhi Grant (B)15H03626. 


\end{document}